\documentclass[letterpaper,11pt]{article}

\usepackage[margin=1in]{geometry}
\usepackage{graphicx}

\usepackage{authblk}

\usepackage{todonotes}

\usepackage{tikz}
\usetikzlibrary{arrows,calc,shapes}
\usepackage{tikz-cd}

\tikzstyle{arc}=[->,shorten <=3pt, shorten >=3pt,
                 >=stealth, line width=1.1pt]
\tikzstyle{edge}=[shorten <=2pt, shorten >=2pt,
                  >=stealth, line width=1.1pt]
\tikzstyle{blueE}=[shorten <=2pt, shorten >=2pt,
                  >=stealth, line width=1.5pt, blue]
\tikzstyle{redE}=[shorten <=2pt, shorten >=2pt,
                  >=stealth, line width=1.5pt, dashed, red]
\tikzstyle{vertex}=[circle, fill=white, draw,
                    minimum size=7pt,
                    inner sep=0pt, outer sep=0pt]
\tikzstyle{redA}=[->, dashed, shorten <=3pt, shorten >=3pt, red, >=stealth,
                  line width=1.5pt]
\tikzstyle{blueA}=[->, shorten <=3pt, shorten >=3pt, blue, >=stealth,
                   line width=1.5pt]

\usepackage{float}

\usepackage{xcolor}

\usepackage[hidelinks]{hyperref} 

\usepackage{amsmath}
\usepackage{amssymb}
\usepackage{bbm} 
\usepackage{xfrac}

\usepackage{amsthm}

\usepackage[capitalize]{cleveref}

\usepackage{multicol}

\newtheorem{theorem}{Theorem}[section]
\newtheorem{lemma}[theorem]{Lemma}
\newtheorem{corollary}[theorem]{Corollary}
\newtheorem{proposition}[theorem]{Proposition}
\newtheorem{observation}[theorem]{Observation}
\newtheorem{problem}[theorem]{Problem}
\newtheorem{question}[theorem]{Question}
\newtheorem{example}[theorem]{Example}

\newcommand{\calG}{\mathcal G}
\newcommand{\calC}{\mathcal C}
\newcommand{\calB}{\mathcal B}
\newcommand{\calD}{\mathcal D}
\newcommand{\calF}{\mathcal F}
\newcommand{\calA}{\mathcal A}
\newcommand{\calP}{\mathcal P}
\newcommand{\calQ}{\mathcal Q}
\newcommand{\calN}{\mathcal N}
\newcommand{\calM}{\mathcal M}
\newcommand{\LOR}{\mathcal{LOR}}
\newcommand{\OR}{\mathcal{OR}}
\newcommand{\DI}{\mathcal{DI}}
\newcommand{\AO}{\mathcal{AO}}
\newcommand{\LOOR}{\mathcal{LOOR}}
\newcommand{\COR}{\mathcal{COR}}
\newcommand{\SO}{\mathcal{SO}}
\newcommand{\PO}{\mathcal{PO}}
\newcommand{\T}{\mathcal T}

\newcommand{\bbA}{\mathbb A}
\newcommand{\bbB}{\mathbb B}
\newcommand{\bbC}{\mathbb C}
\newcommand{\bbD}{\mathbb D}
\newcommand{\bbX}{\mathbb X}
\newcommand{\bbY}{\mathbb Y}

\DeclareMathOperator{\ex}{ex}
\DeclareMathOperator{\co}{co-}
\DeclareMathOperator{\Csp}{CSP}
\DeclareMathOperator{\Mod}{Mod}

\DeclareMathOperator{\Fwd}{Fwd}
\DeclareMathOperator{\QF}{QF}

\title{Local expressions of hereditary classes}

\author{Santiago Guzm\'an-Pro\thanks{This project has received funding from the European Union
(Project POCOCOP, ERC Synergy Grant 101071674).
Views and opinions expressed are however those of the author(s) only and do not
necessarily reflect those of the European Union or the European Research
Council Executive Agency. Neither the European Union nor the granting
authority can be held responsible for them.}\\
\footnotesize{santiago.guzman\_pro@tu-dresden.de}}

\affil{Institut f\"ur Algebra, TU Dresden, Germany}

\begin{document}
\date{}

\maketitle
\begin{abstract}
A well-established research line in structural and algorithmic graph theory
is characterizing graph classes by listing their minimal obstructions. When
this list is finite for some class $\calC$ we obtain a polynomial-time
algorithm for recognizing graphs in $\calC$, and from a logic point of
view, having finitely many obstructions corresponds to  being definable by a
universal sentence. However, in many cases we study classes with
infinite sets of minimal obstructions, and this might have neither algorithmic nor logic
implications for such a class.
Some decades ago, Skrien (1982) and Damaschke (1990) introduced finite expressions
of graph classes by means of forbidden orientations and forbidden linear orderings,
and  recently, similar research lines appeared in the literature, such as 
expressions by forbidden circular orders, by forbidden tree-layouts, and by forbidden
edge-coloured graphs.
In this paper, we introduce \textit{local expressions} of graph classes; a general
framework for characterizing graph classes by  forbidden equipped graphs.
In particular, it encompasses all  research lines mentioned above,
and we provide some new examples of such characterizations.  Moreover, we see
that every local expression of a class $\calC$ yields a polynomial-time
certification algorithm for graphs in $\calC$. Finally, from a logic point of view,
we show that being locally expressible corresponds to being definable
in the logic SNP introduced by Feder and Vardi (1999). 

\end{abstract}


\section{Introduction}

A natural way of describing hereditary graph classes is by listing their sets of minimal
obstructions. For instance, forests and chordal graphs are defined as graphs with
no induced cycles and graphs with no induced cycles of length at  least four,
respectively. Similarly, the class of bipartite graphs is characterized as
the class of graphs with no odd induced cycles.
Finding these sets of minimal obstructions for hereditary graph classes is a
natural research line that has motivated several works on structural graph
theory. Possibly the most celebrated case is the resolution of the strong perfect graph
conjecture~\cite{berge1961}, which after almost 40 years of research culminated in
the strong perfect graph theorem~\cite{chusdnovskyAM164} characterizing perfect graphs
as those graphs with no induced odd-holes nor odd-antiholes.

These characterizations have strong limitations. For instance, the expressive power
of descriptions by \textit{finitely} many minimal  obstructions is not strong enough to capture
several well-structured graph classes such as forests, chordal graphs, and bipartite
graphs (among many others). Moreover, even by considering descriptions by infinitely
many minimal obstructions, it seems out of reach at the moment to list all minimal obstructions
of standard graph classes. For example, for each $k\ge 3$, the list of minimal obstructions to
$k$-colourability remains unknown, and even for efficiently recognizable classes such as circular-arc
graphs, and $(2,1)$-colourable graphs, their lists of minimal obstructions are still unknown.
It is often possible to overcome these limitations by equipping graphs with extra structure,
and forbidding finitely many equipped graphs. For example, the Roy-Gallai-Hasse-Vitaver Theorem
characterizes $k$-colourable graphs as those graphs that admit an orientation with no directed path
on $k+1$ vertices~\cite{gallaiPCT, hasseIMN28, royIRO1, vitaverDAN147}. Similarly, Fulkerson and
Gross~\cite{fulkersonPJM15} proved that a graph $G$ is a chordal graph if and only if there
is  a linear ordering $\le$ of $V(G)$ such that $(G,\le)$ avoids the structure
$\left(\{v_1 \le v_2 \le v_3\}, \{v_{1}v_{2}, v_{1}v_{3}\}\right)$, i.e., if $G$ admits a \textit{perfect
elimination ordering.}

We say that a class of graphs $\mathcal{C}$ is \textit{expressible by forbidden
orientations} (resp.\ \textit{linear orderings)} if there is a finite set $F$ of 
oriented graphs (resp.\ linearly ordered graphs) such that $\mathcal{C}$ is the class of 
graphs that admit an $F$-free orientation (resp.\ linear ordering). It is straightforward to 
notice that any class with a finite set of minimal obstructions is expressible by forbidden 
orientations and by forbidden linear orderings, i.e., these characterization 
\textit{templates} have a stronger expressive power than expressions by finitely many 
forbidden induced subgraphs. In \cref{tab:examples}, we list some graph classes
with infinitely many minimal obstructions that are expressible by forbidden orientations
or by forbidden linear orderings.

\begin{table}[ht!]

\begin{center}
    \begin{tabular}{| c | c |}
    \hline
    \textbf{Forbidden orientations} & \textbf{Forbidden linear orderings} \\ \hline
    Comparability graphs \cite{skrienJGT6} & Comparability graphs \cite{damaschkeTCGT1990}\\ \hline
    Unicyclic graphs \cite{guzmanAR} & Forests \cite{damaschkeTCGT1990} \\ \hline
    Proper circular-arc graphs \cite{skrienJGT6} &  Interval graphs~\cite{olariuIPL37} \\ \hline
    Proper Helly circular-arc graphs \cite{guzmanAR} & Proper interval graphs \cite{feuilloleyJDM} \\ \hline
    Bipartite graphs \cite{guzmanAR}  &  Bipartite graphs \cite{feuilloleyJDM} \\ \hline
    $\Csp(C_{{2n+1}})$ \cite{guzmanEJC28} & $(k,l)$-colourable graphs \cite{feuilloleyJDM}\\ \hline
    \end{tabular}
    \caption{Examples of graph classes $\mathcal{C}$ with an infinite set of minimal obstructions, such that
    $\mathcal{C}$ is  expressible by forbidden orientations or by forbidden linear orderings.
    See Appendix~\ref{ap:graph-classes} for a glossary of graph classes.}
    \label{tab:examples}
    \end{center}
  \end{table}

As far as we are aware, Skrien~\cite{skrienJGT6} and Damaschke~\cite{damaschkeTCGT1990}
were the first authors to systematically study expressions of graph classes by forbidden
orientations and by forbidden linear orderings, respectively.\footnote{Skrien called 
them $F$-graphs classes, and Damaschke $FOSG$-classes.} Their work has recently been 
extended in~\cite{guzmanEJC28, guzmanAR}  for forbidden orientations, and 
in~\cite{feuilloleyAR, feuilloleyJDM} for forbidden linear orderings. Furthermore, there 
has been a growing interest in equipping graphs with further structure to provide 
different templates for finite descriptions of graph classes. For instance,
in~\cite{guzmanAMC438} the authors consider circularly ordered graphs,
in~\cite{paulPREPRINT} Paul and Protopapas introduce the  notion of tree-layout graphs,
and in~\cite{bokIP} expressions by forbidden $2$-edge-coloured graphs were considered. 
This work stems from this rising topic in structural graph theory: we propose a general 
framework that encompasses all these recent characterisation templates of graph classes.
To do so, we will use simple model theoretic notions, and basic category theory nomenclature
introduced in Section~\ref{sec:prelim}. \newpage

The main contributions of the present work are:
 \begin{enumerate}
    \item Introducing \textit{local expressions} of graph classes by \textit{
        concrete functors}, a unified framework to talk about characterizations of graph
        classes by finitely many forbidden equipped graphs. In particular, we survey
        several known characterisations that correspond to certain local expressions, and
        provide new examples. 
    \item Showing that concrete functors between categories of relational structures are
        categorical duals of quantifier-free definitions between relational signatures.
        An important application of this result is that concrete functors are polynomial-time
        computable, and thus, local expressions of graph classes $\mathcal C$ yield
        polynomial-time certificates for membership in $\mathcal C$.
    \item Introducing the \textit{Characterization Problem}, the \textit{Complexity Problem}, and
    the \textit{Expressibilty Problem}: three generic problems that arise from local expressions
    of graph classes, and that encompass most of the problems considered in this rising area
    of structural graph theory, e.g., in~\cite{bodirskyAR, damaschkeTCGT1990, duffusRSA7, feuilloleyJDM, feuilloleyAR,
    guzmanAMC438, guzmanEJC28, guzmanEJC105, hellESA2014, paulPREPRINT, skrienJGT6}.
 \end{enumerate}

The rest of this paper is structured as follows. First, in Section~\ref{sec:prelim} we 
introduce model and category theory notions needed for later sections.
Then, in Section~\ref{sec:local-expresions} we introduce local expressions of graph
classes by concrete functors and provide a wide range of examples of such
expressions.  In Section~\ref{sec:duality}, we consider local expressions arising from so-called
quantifier-free definitions and prove the categorical duality of concrete functors and
quantifier-free definitions. Later, in Section~\ref{sec:pseudo-local}, we introduce the
notion of pseudo-local classes which arises from considering local expressions by concrete
functors with local domain. Finally, in Section~\ref{sec:conclusion} we
introduce the Characterization Problem, the Complexity Problem, and the
Expressibilty Problem. We also survey and  propose some particular instances of
these problems that remain open.


\section{Preliminaries}
\label{sec:prelim}

This work is mainly addressed to graph theorists, and thus we assume
familiarity with graph theory. For standard notions the reader can
refer to~\cite{bondy2008}, non-standard notions will be introduced
throughout the manuscript as needed, and a glossary of graph classes
can be found in Appendix~\ref{ap:graph-classes}.

No involved model theory is needed for this work, nonetheless, we
carefully introduce concepts from this area that will be useful for
this work, and we illustrate these notions with examples from graph theory.
Readers acquainted with elementary model theory and classical first-order logic
may safely skip Sections~\ref{sub:rel-structures} and~\ref{sub:logic}.

\subsection{Relational structures}
\label{sub:rel-structures}

We follow~\cite{hodges1993} for classical concepts in model theory. A
\textit{relational signature} is a set $\tau$ of relation symbols $R,S,\dots$
each equipped with a positive integer $k$ called the \textit{arity} of the 
corresponding relational symbol. A \textit{$\tau$-structure} $\mathbb A$
consists of the following building blocks.
\begin{itemize}
	\item A  \textit{vertex set} $V(\mathbb A)$, also called the \textit{domain}
    of $\mathbb A$.
	\item For each $R \in \tau$ of arity $k$, a $k$-ary relation on $V(\mathbb A)$.
	We write $R(\mathbb A)$ to denote the relation on $V(\mathbb A)$ named by $R$. We also say that is
    the \textit{interpretation} of $R$ in $\mathbb A$.
\end{itemize}

When the structure $\mathbb A$ is clear from context, we will simply write $V$ instead of
$V(\mathbb A)$, and unless stated otherwise we will work with finite structures, that
is, structures with finite vertex sets. Moreover, we will restrict ourselves to
finite relational signatures. In particular, this implies that for each relational 
signature $\tau$ considered in this work, there are finitely many $\tau$-structures,
up to isomorphism. We denote by $\Omega$ the \textit{empty signature}, i.e.,
$\Omega$ is the signature with no relation symbols. 

In this scenario we can think of a graph $G$ as a relational structure where the
vertex set of $G$ is $V(G)$ and the edge set $E(G)$ is a binary relation over $V(G)$
named by the relation symbol $E$. Similarly, a linearly ordered graph $(G,\le)$ is a
relational structure with two binary relations, one named after $E$ and one named
after $\le$. 

A \textit{substructure} $\mathbb B$ of a $\tau$-structure $\mathbb A$
is a $\tau$-structure with domain $V(\mathbb B) \subseteq V(\mathbb A)$
such that for each $R\in \tau$ of arity $k$, the interpretation of $R$ in $\mathbb B$
is the intersection $R(\mathbb A)\cap V(B)^k$. Given a subset $U\subseteq V(\mathbb A)$,
we denote by $\mathbb A[U]$ the substructure of $\mathbb A$ with vertex set $U$. 
Notice that an $\{E\}$-substructure of a graph $G$ corresponds to the graph
theoretical notion of \textit{induced subgraph}.

Graph homomorphisms naturally generalize to arbitrary relational structures. 
Consider a pair of $\tau$-structures $\mathbb A$ and $\mathbb B$. 
A \textit{homomorphism} $f\colon\mathbb{A \to B}$ is a function $f\colon
V(A)\to V(B)$ such that for every positive integer $k$, for every $k$-ary relation
symbol $R\in \tau$ and for every $k$-tuple $(a_1,\dots, a_k)\in V(\mathbb A)^k$, 
if $(a_1,\dots, a_k)\in R(\mathbb A)$ then $(f(a_1), \dots, f(a_n))\in R(B)$.
If such homomorphism exists we say that $\mathbb A$ is \textit{homomorphic} to
$\mathbb B$  and write $\mathbb A\to \mathbb B$; otherwise we write
$\mathbb A\not\to \mathbb B$. We write $\Csp(\mathbb B)$ to denote the class
of $\tau$-structures homomorphic to $\mathbb B$. We abuse nomenclature, and
whenever $H$ is a graph, we denote by $\Csp(H)$ the class of graph homomorphic
to $H$, e.g., $\Csp(K_3)$ is the class of $3$-colourable graphs. 

An \textit{embedding} is an injective homomorphism $f\colon A\to B$
such that for every $k$-ary relation symbol $R\in \tau$ and every $k$-tuple
$(a_1,\dots, a_k)\in V(\mathbb A)^k$, it is the case that $(a_1,\dots, a_k)\in R(\mathbb A)$
if and only  if $(f(a_1), \dots, f(a_k))\in R(\mathbb B)$. We write
$\mathbb{A < B}$ to denote that $\mathbb A$ embeds in $\mathbb B$.
If an embedding $f\colon \mathbb{A\to B}$ is also surjective, 
we say that $f$ is an \textit{isomorphism}. If such isomorphism exists we write
$\mathbb{A\cong B}$ and say that $\mathbb A$ and $\mathbb B$ are \textit{isomorphic}
$\tau$-structures. Note that there is an embedding $f\colon \mathbb{A\to B}$ if
and only if $\mathbb A\cong \mathbb B[U]$ for some set  $U\subseteq V(\mathbb B)$. 
In particular, $\mathbb A\cong \mathbb B[f[V(\mathbb A)]]$. Thus, whenever
$\mathbb{A < B}$ we can identify  $\mathbb A$ with an induced substructure of $\mathbb B$.
Finally, given a set $\mathcal F$ of $\tau$-structures, a $\tau$-structure
$\mathbb A$ is \textit{$\mathcal F$-free} if and only if $\mathbb A$ does not embed any
structure in $\mathcal F$.


\subsection{Local classes}

We will identify a property of relational structures with the class of structures
that satisfy this property. Conversely, we identify a class of relational structures
with the property defined by membership to this class. We always assume that properties are 
closed under isomorphisms.  A \textit{hereditary class (property)}
of $\tau$-structures is a class $\mathcal{C}$ of $\tau$-structures
closed under embeddings, i.e., if $\mathbb A\in \mathcal{C}$ and $\mathbb{B < A}$,
then $A\in \mathcal{C}$ --- equivalently, $\mathcal{C}$ is closed under substructures. 
We denote by $\Mod_\tau$ the class of all $\tau$-structures.
For instance,  for every set $\mathcal F$ of $\tau$-structures the class of
$\mathcal F$-free $\tau$-structures is a hereditary class. Moreover, for every
hereditary class $\mathcal{C}$ of $\tau$-structures there is a set
$\mathcal F$ such that $\mathcal{C}$ is the set of $\mathcal F$-free
structures. 

A class $\mathcal{C}$ of $\tau$-structures is a \textit{local class}
if there is a positive integer $N$ such that a $\tau$-structure $\mathbb A$
belongs to $\mathcal{C}$ if and only if every substructure $\mathbb A$ of at
most $N$ vertices belongs to $\mathcal{C}$. Intuitively speaking, a class
$\mathcal C$ is local if every ``large'' structure $\mathbb A$ belongs to $\mathcal C$
whenever every ``small'' substructure of $\mathbb A$ belongs to $\mathcal C$.
For instance, we say that a $\tau$-structure $\mathbb A$ is an \textit{empty}
$\tau$-structure if the interpretation of every relation symbol $R\in \tau$
in $\mathbb A$ is empty. In other words, $\mathbb A$ is a vertex set
with no further structure. One can soon notice that, since we are
working with finite signatures, the class of empty $\tau$-structures is a local
class (it suffices to let $N$ be the maximum arity of a relational symbol in $\tau$).
From a graph theory perspective, it is not hard to notice that a class of graphs 
$\mathcal C$ is a local class if and only if it has a finitely many minimal
obstructions. This naturally generalizes to classes of relational structures
as follows. 

Consider a hereditary class of $\tau$-structures $\mathcal C$. A
\textit{minimal bound} of $\mathcal C$ is a $\tau$-structure $\mathbb B$ that
does not belong to $\mathcal C$ but every proper substructure
$\mathbb A$ of $\mathbb B$ belongs to $\mathcal{C}$. 
A \textit{finitely bounded} class of $\tau$-structures $\mathcal C$ is a class with
a finite set of minimal bounds, up to isomorphism.

To illustrate minimal bounds consider the class of loopless graphs 
$\mathcal G$, and the class of loopless oriented graphs $\OR$. So,
the $\{E\}$-minimal bound of $\mathcal{G}$ are $(\{x\},\{(x,x)\})$ and
$(\{x,y\},\{(x,y)\})$, i.e., a loop and an antisymmetric arc, and the
minimal bounds of $\OR$  are $(\{x\},\{(x,x)\})$ and
$(\{x,y\},\{(x,y),(y,x)\})$, i.e., a loop and a symmetric pair of arcs.
Thus, loopless graphs and loopless oriented graphs a examples of finitely 
bounded classes of $\{E\}$-structures.
The following relation between local classes and finite bounds follows
from the assumption that we are working with finite relational signatures. 

\begin{observation}\label{obs:localfinite}
    Given a finite relational signature $\tau$, a class $\mathcal C$ of
    $\tau$-structures is a local class if and only if it is finitely bounded. 
\end{observation}

It will be convenient to introduce the following relative version of local classes. 
Consider a pair $\mathcal{C}_1,\mathcal{C}_2$ of hereditary classes of $\tau$-structures
such that $\calC_1\subseteq \calC_2$. We say that $\calC_1$ is a
\textit{local class relative to $\calC_2$} if there is a positive integer
$N$ such that a structure $\mathbb A\in \calC_2$ belongs
to $\calC_1$ if and only if every substructure of $\mathbb A$ induced by 
at most $N$ vertices belongs to $\calC_1$.
Equivalently, one can soon notice that  $\calC_1$ is a local class
relative to $\calC_2$ if and only if there are only finitely many bounds of 
$\calC_1$ in $\calC_2$, up to isomorphism.
In this case we write $\calC_1\subseteq_l \calC_2$.
In particular, since $\Mod_\tau$ is the class of all $\tau$-structure, 
a class of $\tau$-structures $\mathcal{C}$ is a local class 
if and only if $\calC \subseteq_l \Mod_\tau$. 

Going back to structural graph theory, one can observe that the standard notion
of minimal obstructions of a graph class $\mathcal C$ corresponds to the minimal
bounds of $\calC$ in $\calG$.

It is not hard to notice that the relation of ``being local relative to''
is transitive. Moreover, since the class of $\tau$-structures is a local class
of $\sigma$-structure whenever $\tau \subseteq \sigma$, it follows that
the property of being a local class does not depend on the signature. We
make these statements precise in the following immediate observations. 

\begin{observation}\label{obs:localtransitive}
Let $\tau$ be a relational signature and consider three
classes of $\tau$-structures $\calC_1$, $\calC_2$, and $\calC_3$.
If  $\calC_1 \subseteq_l\calC_2$ and
$\calC_2 \subseteq_l \calC_3$, then  $\calC_1\subseteq_l \calC_3$.
\end{observation}

\begin{observation}\label{obs:absolute}
    Consider a pair of relational signatures $\tau,\sigma$ such that
    $\tau \subseteq \sigma$. A class of $\tau$-structures $\mathcal C$ is a local
    class of $\tau$-structures if and only if $\mathcal C$ is a local class of
    $\sigma$-structures.
\end{observation}

In particular, Observation~\ref{obs:absolute} allows us to talk about local
classes of relational structures without needing to specify the signature.
We conclude this subsection by showing that the union and intersection of relative
local classes are relative local classes as well.

\begin{proposition}\label{prop:inter+union}
Let $\tau$ be a finite relational signature, and consider three classes
of $\tau$-structures $\calC_1$, $\calC_2$ and $\calC_3$. 
If $\calC_1 \subseteq_l \calC_3$ and $\calC_2  \subseteq_l  \calC_3$, then
$\calC_1 \cap \calC_2 \subseteq_l \calC_3$ and $\calC_1 \cup \calC_2
\subseteq_l \calC_3$.
\end{proposition}
\begin{proof}
It is straightforward to notice that every minimal bound of
$\mathcal{C}_1\cap \mathcal{C}_2$ in
$\calC_3$ is a either a minimal bound of $\calC_1$ in $\calC_3$
or a minimal bound of $\calC_2$ in $\calC_3$. So
$\calC_1\cap \calC_2$ is a local class relative to $\calC_3$
whenever $\calC_1$ and $\calC_2$ are local relative to $\calC_3$. 

To prove that $\mathcal{C}_1 \cup \mathcal{C}_2$ is a local class relative to
$\calC_3$, first recall that for every positive integer $n$, we can generate
finitely many non-isomorphic $\tau$-structures on $n$ vertices.
Also note that if a pair of $\tau$-structures $\bbA, \bbB$ embed
in a $\tau$-structure $\mathbb D$ then,
the there is a substructure $\mathbb D'$ of $\mathbb D$
with at most $|V(\bbA)|+|V(\bbB)|$ vertices such that $\bbA,\bbB <\mathbb D' <
\mathbb D$. Thus, by letting $N$ be the maximum of $|V(\bbA)| +|V(\bbB)|$
where $\bbA$ ranges over the finitely many bounds of $\calC_1$ in $\calC_3$, and
$\bbB$ over the finitely many bounds of $\calC_2$ in $\calC_3$, one can notice
that a structure $\mathbb D\in \calC_3$ belongs  to $\calC_1\cup \calC_2$ 
if and only if every substructure of $\mathbb D$ with at most $N$ vertices
belongs to $\calC_1 \cup \calC_2$. Therefore, the union $\calC_1\cup \calC_2$
is a local class relative to $\calC_3$. 
\end{proof}

\begin{corollary}\label{cor:inter+union}
Consider a pair of classes of $\tau$-structures $\calC_1$ and
$\calC_2$. If $\calC_1$ and $\calC_2$ are local classes, then 
$\calC_1\cup \calC_2$ and $\calC_1\cap \calC_2$
are local classes as well.
\end{corollary}

\subsection{Formulas and sentences}
\label{sub:logic}

As previously mentioned, all standard model theoretic definitions and concepts
are taken from \cite{hodges1993}, and this subsection may be safely skipped
by many readers. To illustrate some of the concepts introduced
ahead, we consider digraphs and their signature $\{E\}$.

Given a relational signature $\tau$ we construct certain formal languages
to talk about $\tau$-structures. 
To begin with, we need a stack of \textit{variables}
which will be written
as $u, v, x, y, z, x_0, x_1$ and so on. The \textit{atomic $\tau$-formulas}
are defined as follows. 
\begin{itemize}
	\item If $x$ and $y$ are variables, then $x = y$ is an atomic formula of $\tau$.
	\item Let $k$ be a positive integer, let $R\in \tau$ be an $k$-ary relation symbol,
	and consider any variables $x_1,\dots, x_k$. Then $R(x_1,\dots, x_k)$ is
	an atomic formula of $\tau$.
\end{itemize}

If we introduce an atomic formula $\phi$ as $\phi(\protect\overline{x})$ it means that
$\overline{x}$ is a finite sequence $(x_1,x_2,\dots , x_n)$ of distinct variables, and 
every variable that occurs in $\phi$ lies in $\overline{x}$. 
When $\overline{x}$ is replaced by a sequence $\overline{a}$, then
$\phi(\overline{a})$ means the atomic formula obtained by replacing the variables
in $\overline{x}$ by the variables in $\overline{a}$
(we always assume that $\overline{a}$ is at least as long as $\overline{x}$).
In particular, when $\overline{a}$ is a sequence $a_1,\dots, a_m$ of vertices of
a $\tau$-structure $\bbA$,
$\phi$ makes a statement about $\bbA$ and $\overline{a}$. If this statement
is true, we say that $\phi(\overline{x})$ \textit{is true}
of $\overline{a}$ in $\bbA$
or that $\overline{a}$ \textit{satisfies} $\phi(\overline{x})$  in $\bbA$. In symbols we
write $\bbA\models \phi(\overline{a})$, and we write $\bbA\not\models\phi(\overline{a})$
when $\overline{a}$ does not satisfy $\phi(\overline{x})$ in $\bbA$.
We formally define $\models$ as follows. Let $\phi(\overline{x})$ be an atomic
formula of $\tau$, let $\bbA$ be a $\tau$-structure, and let $\overline{a}\in V(\bbA)^m$.
If $\phi(\overline{x})$ is the formula $x_i = x_j$ for some indices $i$ and $j$,
then $\bbA\models \phi(\overline{a})$ if and only if $a_i = a_j$; and if $\phi(\overline{x})$
is the formula $R(x_{i_1},\dots, x_{i_l})$ for some $R\in \tau$ and $i_1,\dots, i_l\in\{1,\dots, k\}$,
then $\bbA\models \phi(\overline{a})$ if and only if $(a_{i_1},\dots, a_{i_l})\in R(\bbA)$.
For instance, if $\phi$ be the atomic formula $E(x,y)$ and $D$ is a digraph, then
$D\models \phi(u,v)$ if and only if there is an arc from $u$ to $v$ in $D$. 

\textit{First-order} $\tau$-formulas are recursively defined from the symbols
of $\tau$ together with: $\top$ ``true'', $\bot$ ``false'', $\land$ ``and'',
$\lor$ ``or'', $\lnot$ ``not'', $\forall$ ``for all elements'', and $\exists$
``there is an element''. The recursive definition is as follows: every atomic $\tau$-formula
is a first-order $\tau$-formula, and so are $\top$ and $\bot$; if $\phi$ and $\psi$ are
first-order $\tau$-formulas,  then  $\lnot\phi$, $\phi\land \psi$, and $\phi\lor\psi$ are
first-order $\tau$-formulas;  and if $\phi$ is a first-order $\tau$-formula and $y$ is a
variable, then $\forall y\phi$ and $\exists y\phi$ are first-order $\tau$-formulas.

The symbols $\forall$ and $\exists$ are called \textit{quantifiers}, and whenever
a variable appears next to a quantifier -- $\forall y$ ``for all $y$'' or $\exists y$
``there is $y$'' -- we say that it is a \textit{bounded} occurrence of the variable $y$;
otherwise it is a \textit{free}  occurrence of $y$.
The \textit{free variables} of a $\tau$-formula $\phi$ are those with free occurrences
in $\phi$. \textit{Quantifier-free} $\tau$-formulas are first-order $\tau$-formulas without
quantifiers, i.e., with no symbols $\forall$ nor $\exists$. 

Similar to atomic  $\tau$-formulas, when we introduce a first-order $\tau$-formula
$\phi(\overline{x})$, it means that all the free variables of $\phi$ are in $\overline{x}$. So,
$\phi(\overline{a})$ means that we substitute $x_i$ by $a_i$. We extend the definition of $\models$
 from atomic $\tau$-formulas to first-order $\tau$-formulas as follows: if $\phi$
is an atomic $\tau$-formula, then $\bbA \models\phi$ is defined as above; for each 
$\tau$-structure $\bbA$, $\bbA \models\top$, and $\bbA \not\models \bot$;
$\bbA\models \lnot\phi(\overline{a})$ if and only if  $\bbA\not\models\phi(\overline{a})$;
$\bbA\models \phi(\overline{a})\land \psi(\overline{b})$ if and only if $\bbA\models
\phi(\overline{a})$ and $\bbA\models \psi(\overline{b})$;
$\bbA\models \phi(\overline{a})\lor \psi(\overline{b})$ if and only if $\bbA\models \phi(\overline{a})$ or
$\bbA\models \psi(\overline{b})$; if $\phi$ is $\forall y\psi$, where $\psi$ is $\psi(y,\overline{x})$,
then  $\bbA\models \phi(\overline{a})$  if and only if for every element $b\in V(\bbA)$, $\bbA\models
\psi(b,\overline{a})$; finally, if $\phi$ is $\exists y\psi$, where $\psi$ is $\psi(y,\overline{x})$, then
$\bbA\models \phi(\overline{a})$  if and only if there is an element $b\in V(\bbA)$ such that 
$\bbA\models \psi(b,\overline{a})$.

Going back to the signature $\{E\}$, consider the formula $\phi$, $\phi =
\exists y(E(x,y))$, and let $D$ be a digraph and $u$ a vertex of $D$.
Then, $D\models \phi(u)$ if and only if $u$ has an out-neighbour.

Every first-order formula $\phi$ with $k$ free variables, defines
a $k$-ary relation $R_\phi(\bbA)$ over any $\tau$-structure $\bbA$
by considering the $k$-tuples $\overline{a} \in V(\bbA)^k$, such that
$\bbA\models \phi(\overline{a})$. We talk about $R_\phi(\bbA)$ as
the \textit{relation defined} by $\phi$ in $\bbA$. We say that a pair of
formulas, $\phi$ and $\psi$,  are \textit{equivalent in} $\bbA$ if
$R_\phi(\bbA) = R_\psi(\bbA)$. We say that $\phi$ and $\psi$ are
\textit{logically equivalent} if they are equivalent in every finite
$\tau$-structure, in symbols we write $\phi \leftrightarrow \psi$.

We proceed to introduce a family of quantifier-free formulas
that we will often use in Section~\ref{sec:duality}. Let $\bbA$ be
a $\tau$-structure with and an enumeration $v_1,\dots, v_n$ of its
vertex set. The set of atomic and negated atomic $\tau$-formulas $\phi$
with free variables $(x_1,\dots, x_n)$ that are true of $(v_1,\dots, v_n)$ in
$\bbA$, is called the \textit{diagram} of $\bbA$. We find it convenient to 
consider the conjunction of the formulas if the diagram of $\bbA$, which
we denote by $\chi_\tau(\bbA)$ and call it the \textit{characteristic $\tau$-formula}
of $\mathbb A$. Notice that $\chi_\tau(\bbA)$ depends on the enumeration
of $V(\bbA)$, but this does not matter as any two enumerations yield logically
equivalent formulas. The following lemma shows that the characteristic $\tau$-formula
of $\bbA$ is a syntactic description of $\bbA$ with respect to embeddings,
and it is an adaptation of the Diagram lemma (see, e.g., Lemma 1.4.2 in \cite{hodges1993})
to the present context and nomenclature.

\begin{lemma}\cite{hodges1993}\label{lem:characteristic}
Let $\tau$ be a relational signature,  $\bbA$ and $\bbB$ a pair of $\tau$-structures,
and $\overline a = (a_1,\dots, a_n)$ an enumeration of $V(\bbA)$. The following statements are
equivalent of an $n$-tuple $\overline{b}$ of $V(\bbB)$.
\begin{itemize}
    \item $B\models \chi_\tau(\bbA)(\overline{b})$.
    \item  The mapping  $\overline{a}\mapsto \overline{b}$ defines an embedding $f\colon \bbA \to \bbB$.
    \item Every quantifier-free $\tau$-formula $\phi$ with $n$ variables is true of $\overline a$ in $\bbA$ if and only
    if $\phi$ is true of $\overline b$ in $\bbB$.
\end{itemize}
\end{lemma}

A \textit{sentence} is a formula with no free variables. A \textit{theory}
is a set of sentences. Suppose that $\phi$ is a sentence and 
$\bbA$ is an $\tau$-structure. Since there are no free variables in $\phi$,
then $\bbA\models \phi(\overline{a})$ for some sequence of vertices
$\overline{a}$ if and only if $\bbA\models \phi(\overline{b})$ for the empty
sequence $\overline{b}$. Thus, we simply write $\bbA\models \phi$ if
the empty sequence satisfies $\phi$ in $\bbA$, and we say that $\bbA$
\textit{models} $\phi$ or that $\bbA$ is a \textit{model} of $\phi$.
We extend this definition to theories in the obvious way, $\bbA$
\textit{models} a theory $T$ if $\bbA$ is a model of every sentence in
$T$, and in this case we also write $\bbA\models T$. 
For instance, let $\phi$ be the sentence $\forall xE(x,x)$, and $\psi$ the
sentence $\exists x,yE(x,y)$. So, if $D$ is a digraph, then
$D\models \phi$  if and only if every vertex in $D$ has a loop,
and $D\models\psi$ if and only if there is an arc in $D$.

Consider a theory $T$ and a class of $\tau$-structures $\mathcal{C}$. We
say that $T$ \textit{axiomatizes} $\mathcal{C}$ if $\mathcal{C}$ is the class
of all $\tau$-structures that model $T$.  
We say that a class $\mathcal{C}$ is \textit{definable}
if there is a finite theory that axiomatizes $\mathcal{C}$. Equivalently, if there is 
a sentence $\phi$ such that $\calC$ is the class of structures that model $\phi$
(simply take $\phi$ to be the conjunction of all formulas in $T$). Going back to
our previous example, the class of reflexive digraphs is defined by the 
sentence $\forall xE(x,x)$. Similarly, if
$\mathcal{C}$ is the class of digraphs with an isolated vertex, 
then $\mathcal{C}$ is defined by the sentence $\phi$, 
\[
\phi:= \exists x\forall y( x = y \lor (\lnot E(x,y)\land\lnot E(y,x))).
\]

A formula $\phi$ is a \textit{universal formula}
if it is build from quantifier-free formulas
by means of $\lor$, $\land$ and $\forall$. In symbols, we write that
$\phi$ is an \textit{$\forall_1$-formula}. Similarly, an \textit{existential}
formula is build up from $\lor$, $\land$ and $\exists$, and we write \textit{$\exists_1$-formula}.
A \textit{universal} (resp.\ \textit{existential}) theory is a set of universal
(resp.\ existential) sentences, and we say that a class $\mathcal C$
is \textit{$\forall_1$-definable}  if it is axiomatized by a finite universal theory.
It is not hard to observe that if $T$ is a universal theory, then the class of
all models of $T$ is a hereditary class. Moreover, $\forall_1$-definable classes
correspond to local classes as the following statement asserts.

\begin{theorem}\label{thm:equivlocal}
Let $\tau$ be a finite relational signature. The following statements are equivalent
for a hereditary class of $\tau$-structures $\mathcal{C}$.
	\begin{itemize}
		\item $\calC$ is a local class.
		\item $\calC$ is finitely bounded.
		\item $\calC$ is $\forall_1$-definable.
	\end{itemize}
\end{theorem}
\begin{proof}
The first two items are equivalent as stated in Observation~\ref{obs:localfinite}. The equivalence between the
second and third statement is a well-known fact (see, e.g., Lemma 2.3.14 in \cite{bodirsky2021}). The third item
is a particular case of the fourth one. Finally, the latter implies the former since every hereditary class can
be axiomatized by a universal theory -- take for instance the set $\{\forall\overline
x(\lnot \chi_\tau(\bbA)(\overline x))|\bbA$ is a bound of $\mathcal C\}$ -- and so, a universal definition of $\calC$
can be obtained using the assumption that $\calC$ is definable, and standard compactness arguments (see, e.g.,
Corollary 6.1.2 from~\cite{hodges1993}).
\end{proof}

\subsection{Category theory}

We assume familiarity with basis notions from category theory such as \textit{functors,
objects, morphisms,} and \textit{commutative diagrams}. For such basic concepts of category
theory we refer the reader to~\cite{maclane1978}, or to~\cite{riehl2017} for a thorough
study of category theory. Besides this simple nomenclature from category theory, we also
use the notion of  \textit{concrete functors}. These functors are defined in the general
setting of concrete categories, but we choose restrict the definition only to our context.
 
When we talk about functors $F\colon\mathcal{C\to D}$ between classes of relational structures,
we think  $\calC$ and $\mathcal D$ as categories of relational structures whose morphisms
are embeddings. Given a pair of hereditary classes of relational structures $\mathcal{C}$ and
$\mathcal{D}$, a \textit{concrete functor} from $\mathcal{C}$ to $\mathcal{D}$ is a functor
$F\colon\mathcal{C\to D}$ such that for every $\bbA,\bbB\in \calC$ and every embedding
$f\colon\bbA\to\bbB$ the equality $F(f) = f$ holds. In simple terms, concrete functors are 
mappings between classes of relational structures that preserve embeddings.
For instance,  the mapping $S\colon \mathcal{DI \to G}$ that maps a digraph to its
underlying graph, preserves embeddings and so, it is a concrete functor according to
the present context.

\begin{proposition}\label{prop:concretefunct}
Consider a pair $\mathcal{C},\mathcal{D}$ of hereditary classes 
of relational structures.
If $F\colon \mathcal{C \to D}$ is a concrete functor, then the following statements hold:
\begin{itemize}
	\item For every $\bbA\in \calC$, the equality $V(\bbA) = V(F(\bbA))$ holds.
	\item For every $\bbA\in \calC$ and every embedding $f\colon \mathbb X\to F(\bbA)$, there is a
    unique embedding $f'\colon \mathbb Y\to \bbA$ such that $F(f') = f$.
\end{itemize}
\end{proposition}
\begin{proof}
The first item holds since $F(f) = f$ for every embedding of structures
in $\mathcal{C}$. In particular, for every $\bbA\in \mathcal{C}$ if $f$
is the identity on $\bbA$, then the domain of the function $F(f)$ is
$V(\bbA)$, and so $V(F(\bbA)) = V(\bbA)$. 

To prove the second statement, first notice that if such embedding exists, then
it is unique. Indeed, if $F(f') = f$ and $F(f'') = f$, then, it follows from the definition
of concrete functor that $f' = f = f''$. Now we show that such embedding exists. 
Consider an $\tau$-structure $\bbA\in\calC$ and an embedding
$f\colon \mathbb X\to F(\bbA)$ for some $\mathbb X\in \mathcal D$.
Let $\mathbb Y$ be the structure with vertex set $V(\mathbb X)$, and for each
relation symbol $R$ in the signature of structure in $\calC$, define 
$R(\mathbb Y)$ to be the set of tuple $(x_1,\dots, x_k)$ such that
$(f(x_1),\dots, f(x_k))\in R(\bbA)$. By construction of $\mathbb Y$, it must
be the case that $f\colon \mathbb Y\to \bbA$ is an embedding, so it follows from
the fact that $\mathcal{C}$ is a hereditary class that $\mathbb Y\in \calC$. 
Again, since $F$ is a concrete functor, we conclude that $F(f) = f$, and the claim follows.
\end{proof}

As a side note,  functors that satisfy the second item of
Proposition~\ref{prop:concretefunct} are called \textit{discrete fibrations}
so, this proposition asserts  that every concrete functor is a discrete fibration.

The following statements show that images and preimages (under
concrete functors) of hereditary classes, remain hereditary classes.

\begin{lemma}\label{lem:images}
Consider a pair $\calC$ and $\mathcal{D}$ of hereditary classes 
of relational structures. If $F\colon \mathcal{C\to D}$ is a
concrete functor, then for every hereditary subclass $\calC'\subseteq \calC$,
its image $F[\calC']$ is a hereditary subclass of $\mathcal{D}$.
\end{lemma}
\begin{proof}
Let $\bbA\in F[\mathcal{C}']$, and $\bbB\in \mathcal{C}'$ such that $F(\bbB) = \bbA$. Suppose
that there is an embedding $f\colon \mathbb X\to \bbA$. We want to show that
$\mathbb X\in F[\mathcal{C}']$. By the second item of Proposition~\ref{prop:concretefunct},
there is an embedding $f'\colon \mathbb Y\to \bbB$ such that $F(f') = f$. 
In particular, this means that $F(\bbY) = \bbX$, and  the claim follows since
$\mathcal{C}'$ is a hereditary class.
\end{proof}

\begin{lemma}\label{lem:preimages}
Consider a pair $\calC$ and $\mathcal{D}$ of hereditary classes 
of relational structures. If $F\colon \mathcal{C\to D}$ is a
concrete functor, then for every hereditary subclass $\mathcal{D}'\subseteq \mathcal{D}$
its preimage $F^{-1}[\mathcal{D}']$ is a hereditary subclass of  $\mathcal{C}$.
Moreover, if $\mathcal F$ is the set of minimal bounds of $\mathcal{D}'$ in $\mathcal{D}$, 
then $F^{-1}[\mathcal F]$ is the set of minimal bounds of $F^{-1}[\mathcal{D}']$ in
$\mathcal{C}$.
\end{lemma}
\begin{proof}
Let $\mathcal{C}' = F^{-1}[\mathcal{D}]$ and $\calF' = F^{-1}[\calF]$. We want to
show that a structure $\bbA\in \mathcal{C}$ belongs to $\mathcal{C}'$ if and only
if $\bbA$ is $\calF'$-free. We prove both implications by contraposition.
First suppose that $\bbA\in\calC$ is not $\calF'$-free, so there is a structure $\bbB\in \calF'$
and an embedding $f\colon \bbB\to \bbA$. Since $F$ is a concrete functor,
then $f\colon F(\bbB)\to F(\bbA)$ is an embedding. By the choice of $\bbB$,
we know that $F(\bbB)\in \calF$, so $F(\bbA)$ is not $\calF$-free. Hence,
$F(\bbA)\not\in \mathcal{D}'$ and thus,  $\bbA\in \calC\setminus F^{-1}[\mathcal{D}']
=\calC\setminus \mathcal{C}'$. Now we prove the converse implication.
Let  $\bbA\in \calC\setminus \mathcal{C}'$, i.e., $F(\bbA) \in \mathcal{D}\setminus \mathcal D'$, and
so, there is a structure $\bbB\in \calF$ and an embedding $f\colon \bbB\to F(\bbA)$. By the
second part of Proposition, there is a structure $\mathbb Y\in \calC$ and
an embedding $f'\colon \bbY \to \bbA$
such that $F(f') = f$. Hence, $F(\bbY) = \bbB\in \calF$, and thus
$\bbY\in F^{-1}[\calF] = \calF'$. Since $\bbY$ into $\bbA$, we see that
$\bbA$ is not $\calF'$-free, which concludes the proof.
\end{proof}

To conclude this brief section we introduce a particular family of concrete functors.
Consider a pair of relational signatures  $\tau$ and $\sigma$ such that
$\tau\subseteq \sigma$. Let $\bbA$ be an $\tau$-structure, and $\mathbb X$ a
$\sigma$-structure. We say that $\bbA$ is the \textit{$\tau$-reduct}
of $\mathbb X$ if $V(\bbA) = V(\bbX)$, and for every $R\in \tau$ the
interpretations $R(\bbA)$ and $R(\bbX)$ coincide. Note that the $\tau$-reduct
of any $\sigma$-structure is unique, thus we denote it by $\mathbb X_\tau$,
and we say that $\bbX$ is a \textit{$\sigma$-expansion}  of $\bbA$. Whenever
there is no risk of ambiguity regarding the signature $\tau$, we call
$\bbX_\tau$ the reduct of $\bbX$. Finally, suppose that $\mathcal{C}$ is a class
of $\sigma$-structures, and $\calC_\tau$ is the class of $\tau$-reducts of
$\calC'$, then the functor $F\colon \calC'\to \calC$ defined by $F(\bbA) = \bbA_\tau$
is a concrete functor. These functors are called \textit{forgetful functors}.
As a natural example let $\LOR$ be the class of linearly ordered graphs
considered as $\{E,\le\}$-structures, and $\mathcal G$ the class of graphs. 
With this setting, $\mathcal G$ is the class of $\{E\}$-reducts of 
$\LOR$. 



\section{Local Expressions}
\label{sec:local-expresions}

Across different areas of mathematics one encounters several examples of
classes defined with the following scheme. First, consider
two classes of mathematical objects $\mathcal C$ and $\mathcal D$ together
with some functorial mapping $F\colon \mathcal{C\to D}$. Secondly, by taking
a subclass $\mathcal{C}'$ of $\mathcal{C}$, define a class
$\mathcal{D}'$ of objects in  $\mathcal{D}$ as those objects that belong
to $F[\mathcal{C}']$.  Examples of such classes include metrizable spaces,
orderable groups, orientable matroids, and $k$-colourable graphs. 

Several graph classes can be defined with the previous scheme
by considering the functor  $S\colon\mathcal{DI\to G}$ that maps a digraph
(a binary relation) to its underlying graph (its symmetric closure). With
this functor in mind, the class of comparability
graphs is defined as the image of the class of irreflexive partially ordered
sets~\cite{gallaiAMA18}, the class of perfectly orientable graphs is the
image of the class of out-tournaments~\cite{skrienJGT6}, and the class of
$2$-edge-connected graphs is characterized as the image of strongly connected
digraphs (see, e.g., Theorem 5.10 in~\cite{bondy2008}). In general, a graph class
$\calC$ is expressible by forbidden orientations if there is a local class of
oriented graphs $\mathcal D$ such that $S[\mathcal D] = \calC$.
Similarly, expressions by forbidden linearly ordered graphs -- Damaschke's notion
of FOSG-classes~\cite{damaschkeTCGT1990} -- 
can be interpreted in the present context by considering the he consider the
functor $sh_L\colon \mathcal{LOR\to G}$ that forgets the linear ordering of 
a linearly ordered graph.

\subsection{Local expressions}
\label{ExpressivePower}

Throughout the rest of this work, we will only deal with classes and categories 
of relational structures with embeddings.  We will introduce several examples
of concrete functors, and their definition can also be found in
Appendix~\ref{ap:concrete-functors}.\footnote{In
the electronic version, the reader can find a hyperlink to the definition of
each functor over the corresponding symbol. For example,
$\hyperlink{AO}{S_{|\AO}}$, $ex(\hyperlink{LOR}{sh_L})$
and $\hyperlink{OR}{S_{|\OR}\colon\mathcal{OR \to G}}$.}

To begin with, we denote by $\LOR$ the class of linearly ordered graphs, by
$\DI$ the class of digraphs, by $\OR$ the class of oriented graphs, and by
$\AO$ the class of acyclic oriented graphs. The functor $\hyperlink{LOR}{sh_L}\colon\LOR\to \calG$
forgets the linear ordering, the functor $\hyperlink{DI}{S}\colon\DI\to \calG$
maps a digraphs to it underlying graph, and $\hyperlink{OR}{S_{|\OR}}$ and
$\hyperlink{AO}{S_{|\AO}}$ are the restrictions of $S$ to the class of oriented
graphs and of acyclic oriented graphs, respectively. 

Consider a surjective concrete functor $F\colon \mathcal{C\to D}$ where
$\mathcal{C}$ and $\mathcal D$ are hereditary classes --- notice that by
Lemma~\ref{lem:images}, the latter condition is equivalent to
$\calC$ be a hereditary class. 
We say that a class $\mathcal{D' \subseteq D}$ is  \textit{locally expressible by $F$}
if there is a local class $\mathcal{C'}$ relative to $ \calC$ such that
$F[\mathcal{C'}] = \mathcal{D'}$. The \textit{expressive power}
of $F$ is the collection of classes locally expressible
by $F$, and we denote this collection by $\ex(F)$. 
For instance, for any hereditary class $\mathcal{C}$ the expressive power of
the identity $Id_\calC$ is the class of local classes relative to $\calC$.
Clearly, the collection of classes expressible by forbidden linearly
ordered graphs (resp.\ by forbidden orientations) corresponds to the expressive power
of the  forgetful functor $\hyperlink{LOR}{sh_L\colon\mathcal{LOR\to G}}$
(resp.\ to the restriction $\hyperlink{OR}{S_{|\OR}\colon\mathcal{OR\to G}}$
of $\hyperlink{DI}{S\colon \mathcal{DI \to G}}$ to the class of oriented graphs).

The problem of finding hereditary graph classes not expressible by forbidden
(acyclic) oriented graphs considered in~\cite{guzmanEJC105} arises from the
basic observation that there are only countably many hereditary graph classes
expressible by forbidden (acyclic) orientations. In terms of the present
notation, $\ex(\hyperlink{OR}{S_{|\OR}})$ is
a countable set, and one can soon notice that holds for any concrete functor.

\begin{proposition}\label{prop:basicexpressiveproperties}
Let $\mathcal{C}$ and $\mathcal{D}$ be a pair of hereditary classes of
relational structures. For any concrete functor $F\colon\mathcal{C\to D}$ a surjective
concrete functor the following statements hold.
\begin{enumerate}
	\item $\ex(F)$ is a countable set.
	\item If $\mathcal{P,Q}\in \ex(F)$, then $\mathcal{P\cup Q} \in \ex(F)$.
	\item If $\mathcal{P}\in \ex(F)$ and $\mathcal{Q}\subseteq_l \mathcal{P}$,
	then $\mathcal{Q}\in \ex(F)$.
\end{enumerate}
\end{proposition}
\begin{proof}
    The first item holds since we are considering only finite relational signatures, 
    so the are countably many finite sets of structures in $\calC$, and thus
    countably many images of local classes relative to $\calC$. The second item
    follows from the simple fact that $F[\calC_1\cup \calC_2] = F[\calC_1]\cup F[\calC_2]$,
    and because if $\calC_1,\calC_2\subseteq_l \calC$, then $\calC_1\cup\calC_2\subseteq_l 
    \calC$ (Proposition~\ref{prop:inter+union}). The third item follows with a 
    straightforward argument as in the proof of Lemma~\ref{lem:preimages}.
\end{proof}

As mentioned in the introduction, it is not hard to notice that every class of graphs with a
finite set of minimal obstructions is expressible by forbidden linearly ordered graphs
and by forbidden orientations. This holds in general for local expressions by concrete functors
and follows from the third item in Proposition~\ref{prop:basicexpressiveproperties}.

\begin{corollary}\label{cor:relativelocar-exp}
Consider a pair of hereditary classes of relational structures $\mathcal{C}$ and $\mathcal{D}$.
If $F\colon\mathcal{C\to D}$ is a surjective concrete functor, then every local class
relative to $\mathcal{D}$ is locally expressible by $\ex(F)$.
\end{corollary}

It was observed in \cite{guzmanAMC438} (resp.\ in~\cite{guzmanEJC105}) that every
class expressible by forbidden circular orderings (resp.\ by forbidden acyclic orientations),
is also expressible by forbidden linear orderings. This can be seen as an easy consequence
of a simple algebraic relation between the corresponding forgetful functors, which we proceed
to introduce.

Consider a pair of concrete functors with the same codomain
$F\colon \mathcal{C\to D}$ and $G\colon\mathcal{B\to D}$. We say
that $G$ \textit{factors} $F$ or that $F$
\textit{factors through} $G$ if there is a surjective concrete functor
$H\colon\mathcal{C\to B}$ such that $F = G\circ H$; that is, if the following
diagram commutes.
\begin{center}
\begin{tikzcd}[column sep=1.8cm, row sep=1.1cm]
  \mathcal{C}  \arrow[d, twoheadrightarrow, "H", dashed] \arrow[dr, "F"] &  \\
  \mathcal{B}  \arrow[r, "G"] & \mathcal{D}
\end{tikzcd}
\end{center}

A \textit{local inclusion} is a concrete functor
$i\colon\mathcal{B\hookrightarrow C}$ such that $i$ is an isomorphism
from $\mathcal{B}$ to $i[\mathcal{B}]$ and
$i[\mathcal{B}]\subseteq_l \mathcal{C}$. 
We say that $G$ is a \textit{local restriction} of
$F$ if there is  a local inclusion $i\colon \mathcal{B\hookrightarrow C}$ such that 
$G = F\circ i$; i.e., the following diagram commutes.
\begin{center}
\begin{tikzcd}[column sep=1.8cm, row sep=1.1cm]
  \mathcal{C}   \arrow[dr, "F"] &  \\
  \mathcal{B} \arrow[u, hookrightarrow, "i", dashed]  \arrow[r, "G"] & \mathcal{D}
\end{tikzcd}
\end{center}

For instance,  consider the forgetful functor $\hyperlink{LOR}{sh_L\colon \mathcal{LOR\to G}}$
and the  symmetric functor \hyperlink{AO}{$S_{|\AO}\colon\mathcal{AO\to G}$}.
In this case, $\hyperlink{LOR}{sh_L}$ factors through
$\hyperlink{AO}{S_{|\AO}}$.
Indeed, consider the functor 
$\Fwd\colon \mathcal{LOR\to AO}$ that orients every edge forward and
forgets the linear ordering of the vertices. Clearly, 
$\hyperlink{LOR}{sh_L} = \hyperlink{AO}{S_{|\AO}}\circ \Fwd$. 
As the following proposition asserts, the fact that every
class expressible by forbidden acyclic orientation is also expressible by
forbidden linear orderings follows from the fact that $\hyperlink{LOR}{sh_L}$
factors through $\hyperlink{AO}{S_{|\AO}}$.

\begin{proposition}\label{prop:factor-restriction}
Let $\calB$, $\calC$, and $\calD$ be three hereditary classes of relational structures. 
The following statements hold for any pair of concrete functors
$F\colon \mathcal{C\to D}$ and $G\colon\mathcal{B\to D}$:
\begin{enumerate}
	\item If $G$  factors $F$, then $\ex(G)\subseteq \ex(F)$.
	\item If $G$ is a local restriction of $F$, then $\ex(G)\subseteq \ex(F)$.
\end{enumerate}
\end{proposition}
\begin{proof}
To prove the first statement consider a concrete functor 
$H\colon\mathcal{C\twoheadrightarrow B}$
such that $F = G\circ H$. Let $\mathcal{P}\in \ex(G)$ and
 $\mathcal{Q\subseteq}_l \calB$ such that
 $G[\mathcal{Q}] = \mathcal{P}$. By Corollary~\ref{cor:relativelocar-exp},
 $\mathcal{Q}\in \ex(H)$ so there is a class $\mathcal{Q'\subseteq}_l\mathcal{C}$
 such that $H[\mathcal{Q'] = Q}$. Thus, $F[\mathcal{Q}] = G\circ H[\mathcal{Q}] 
 = \mathcal{P}$ and so, $\mathcal{P}\in \ex(F)$.
 The second statement follows from the fact that if $i\colon \mathcal{B\hookrightarrow C}$
 is a local inclusion and $\mathcal{Q}$ is a local class relative to $\mathcal{B}$,
 then  $i[\mathcal{Q}]$ is a local class relative to $\mathcal{C}$ (see, e.g., Observation~\ref{obs:localtransitive}).
 Thus, if  $\mathcal{P\subseteq D}$ and $G[\mathcal{Q}] = \mathcal{P}$ for some
 $\mathcal{Q\subseteq}_l\mathcal{B}$, then $F[i[\mathcal{Q}]] = \mathcal{P}$
 and $i[\mathcal{Q}]\subseteq_l\mathcal{C}$, so $\mathcal{P}\in \ex(F)$.
\end{proof}

Above, we provided particular instances of the first item of 
Proposition~\ref{prop:factor-restriction}, as an example of the 
second item one can consider that every class expressible by $k$-coloured
graphs, is expressible by $(k+1)$-coloured graphs. Formally, for every positive
integer $k$ denote by $\mathcal{G}_k$ the class of vertex coloured graphs with colours
$\{1,\dots, k\}$. We denote by
$\hyperlink{Gk}{FC_k\colon\mathcal{G}_k\to \mathcal{G}}$ the
functor that forgets all colours. It is not hard to notice that for every positive
integer $k$ the class of (properly) $k$-colourable graphs belongs to
$\ex(\hyperlink{Gk}{FC_k})$. 
Consider now the inclusion functor $i_k\colon \mathcal{G}_k\hookrightarrow
\mathcal{G}_{k+1}$. 
Clearly, $i_k$ is a local inclusion since a $(k+1)$-coloured graph is $k$-coloured
if and only if it has no vertex with colour $k+1$, so $\mathcal{G}_i\subseteq_l
\mathcal{G}_{k+1}$. Naturally, $\hyperlink{Gk}{FC_k = FC_{k+1}}\circ i_k$ so 
$\ex(\hyperlink{Gk}{FC_k})\subseteq \ex(\hyperlink{Gk}{FC_{k+1}})$ for
every positive integer $k$. 

\begin{question}\label{qst:k-k+1}
Is there a positive integer $k$ such that the expressive power
$\ex(\hyperlink{Gk}{FC_k})$ equals the expressive power
$\ex(\hyperlink{Gk}{FC_{k+1}})$?
\end{question}

A natural candidate to show that $\ex(\hyperlink{Gk}{FC_k})$ is a proper subset
of $\ex(\hyperlink{Gk}{FC_{k+1}})$ is the class of $k$-colourable graphs. In particular, 
we believe that the class of $3$-colourable graphs is not locally expressible
by $\hyperlink{Gk}{FC_2}$.

\begin{problem}\label{prob:3col-2colours}
Determine if there is a finite set $\calF$ of $2$-coloured graphs such that a graph
$G$ is $3$-colourable if and only if it admits an $\calF$-free $2$-colouring.
\end{problem}

Allow us to recall a couple of simple observations of expressions by
linearly ordered graphs and oriented graphs that generalize to
concrete functors and their expressive power. Suppose that $\calF$ is a finite
set of linearly ordered graphs and let $\calF^\ast$ be the set
$\{(G,\le^\ast)\colon (G,\le)\in \calF\}$, where  $\le^\ast$ is the inverse
linear ordering of $\le$. Clearly, a graph $G$ admits an
$\calF$-free linear ordering if and only if $G$ admits an $\calF^\ast$-free linear 
ordering --- this is called the mirror property in~\cite{feuilloleyJDM}. Similarly,
for a set $\calF$ of oriented graphs, let $\overleftarrow{\calF}$ be the set obtained
by reversing the orientation of every oriented graph in $\calF$. 
In this case, the set of graphs that admit an $\calF$-free orientation is the same
as the class of graphs that admit an $\overleftarrow{\calF}$-free orientation.

Again, both observations in the previous paragraph are particular instances
of a simple ``categorical'' observation. 
Consider a concrete functor $F\colon\mathcal{C\to D}$
and a concrete isomorphism $\nu\colon\mathcal{C\to C}$,
i.e.,  a concrete functor which is bijective on objects
and embeddings. If $F = F\circ \nu$
then for every class $\mathcal{P\subseteq D}$ and any  
class $\mathcal{Q\subseteq}_l\mathcal{ C}$ such that $F[\mathcal{Q}] =
\mathcal{P}$ the equality $F\circ\nu[\mathcal{Q] = P}$ trivially holds,
and $\nu[\mathcal{Q}]\subseteq_l \mathcal{C}$. 
In the context of the previous paragraph, consider the functors
$re\colon \mathcal{OR\to OR}$ 
and $\ast\colon \mathcal{LOR\to LOR}$
where $re$ reverses the orientation of every oriented graph and $\ast$
that maps a linearly ordered graph to the same graph with the inverse linear
ordering of the vertices. Clearly, $re$ and $\ast$ are isomorphisms and
the following diagrams commute.

\begin{center}
\begin{tikzcd}[column sep=1.8cm, row sep=1.1cm]
 \OR \arrow[d,"\hyperlink{OR}{S_{|\OR}}"] \arrow[r, "re"]
 & \OR \arrow[dl,"\hyperlink{OR}{S_{|\OR}}"] \\
 \mathcal{G} & 
 \end{tikzcd}\qquad
\begin{tikzcd}[column sep=1.8cm, row sep=1.1cm]
 \LOR \arrow[d,"\hyperlink{LOR}{sh_L}"] \arrow[r, "\ast"]
 & \LOR \arrow[dl,"\hyperlink{LOR}{sh_L}"] \\
 \mathcal{G}& 
\end{tikzcd}
\end{center}

This trivial observation can be slightly strengthened to yield the following
simple result.

\begin{proposition}\label{prop:symmetries}
Consider a concrete functor $F\colon \mathcal{C\to D}$ between
a pair of hereditary classes of relational structures, and  let 
  $\mu\colon\mathcal{D\to D}$ be a concrete isomorphism. If there is a concrete isomorphism
$\nu\colon\mathcal{C\to C}$ such that
$F\circ \nu = \mu\circ F$ then, for every
hereditary class $\mathcal{P\subseteq D}$,
\[
\mathcal{P}\in \ex(F) \text{ if and only if } \mu[P]\in \ex(F).
\]
\end{proposition}
\begin{proof}
This statement is an immediate implication of the following two facts.
If $\mathcal{Q}\subseteq_l \mathcal{C}$ then $\mathcal{\nu[Q]}\subseteq_l
\mathcal{C}$ and if $F[\mathcal{Q}] = \mathcal{P}$ then
$F\circ\nu[\mathcal{Q}]  = \mu[\mathcal{P}]$.
\end{proof}

As previously stated, the mirror property from~\cite{feuilloleyJDM} is captured
by Proposition~\ref{prop:symmetries} by taking $\nu$ to be $\ast$, and $\mu$
to be the identity. Now we provide a simple example where $\mu$ is not the identity. 
Let $co\colon\mathcal{G\to G}$ be the functor that maps a graph $G$ to if complement
$\overline{G}$, and $co_L\colon\mathcal{LOR\to LOR}$ the functor that maps a linearly
ordered graph $(G,\le)$ to $(\overline{G},\le)$. Clearly, $co\colon \mathcal{G\to G}$ and
$co_L\colon\mathcal{LOR\to LOR}$ are isomorphisms and the following diagram
commutes.
\begin{center}
\begin{tikzcd}[column sep=1.8cm, row sep=1.1cm]
 \LOR \arrow[d,"\hyperlink{LOR}{sh_L}"] \arrow[r, "co_L"]
 & \LOR \arrow[d,"\hyperlink{LOR}{sh_L}"] \\
 \mathcal{G} \arrow[r, "co"] & \mathcal{G} &
\end{tikzcd}
\end{center}
 So, $\mathcal{P}\in \ex(\hyperlink{LOR}{sh_L})$ if and only if
 $co[\mathcal{P}] \in \ex(\hyperlink{LOR}{sh_L})$ --- this is called
 the exchange-complement property in~\cite{feuilloleyJDM}. 

It is not hard to notice that if $F\colon \calC\to \calD$ and $G\colon \calC'\to \calD'$
are a pair of concrete functors such that $\ex(F)\subseteq \ex(G)$, then $D\subseteq D'$
because $\calD\in \ex(F)$ (see, e.g., Proposition~\ref{prop:basicexpressiveproperties}). 
So, if $\ex(F) = \ex(G)$, then $F$ and $G$ have the same image. Are there any
stronger algebraic relations between two functors with the same expressive power?
Of course we are particularly interested in the case when the codomain is the
class of graphs.

\begin{question}\label{qst:relationfunctors}
Given a pair of concrete functors $F\colon\mathcal{C\to G}$ and $F'\colon\mathcal{D\to G}$
such that $\ex(F) \subseteq \ex(F')$, is there a meaningful relation between $F$ and $F'$?
\end{question}

\begin{question}\label{qst:equivalencefunctors}
Given a concrete functor $F\colon\mathcal{C\to G}$,
is there a meaningful characterization of the functors
$F'\colon\mathcal{D\to G}$ such that $\ex(F) = \ex(F')$?
\end{question}

Above, we considered simple algebraic relations between functors $F$ and $G$
that translate to relations between their expressive powers. To conclude this
subsection, we consider simple algebraic operations between $F$ and $G$ that
have a neat relation with their expressive powers. To begin with, given a
pair of functors $F\colon\mathcal{B\to D}$ and $G\colon\mathcal{C\to D}$
we consider their disjoint union $F\oplus G$. 
\footnote{In order to stay within our context, 
the disjoint union of $\mathcal{B}$ and $\mathcal{C}$ must be a category
of relational structures with embeddings. This can be attained by colouring all
the vertices of structures in $\mathcal{B}$ with colour $U_B$
and the vertices of structures in $\mathcal{C}$ with colour $U_C$, so 
the disjoint union of these classes (with coloured vertices) together with 
embeddings is isomorphic to the disjoint union of the categories $\mathcal{B}$
and $\mathcal{C}$.
}
It turns out
that the expressive power of the disjoint union of a pair of functors can 
easily be described as follows.

\begin{proposition}\label{prop:disjointunion}
Consider a pair $F\colon\mathcal{B\to D}$ and $G\colon\mathcal{C\to D}$
of concrete functors. The expressive power of $F\oplus G$ is the collection
\[
\bigl\{\mathcal{P\cup Q\subseteq D}\colon
 \mathcal{P}\in \ex(F), \mathcal{Q}\in \ex(G)\bigl\}.
 \]
 \end{proposition}
 \begin{proof}
 
Let  $A = \{\mathcal{P\cup Q\subseteq D}\colon \mathcal{P}\in \ex(F),
\mathcal{Q}\in \ex(G)\}$.
Consider the canonical inclusions $i_B\colon\mathcal{B\hookrightarrow B\oplus C}$
and $i_C\colon\mathcal{C\hookrightarrow B\oplus C}$. Denote by $M_B$
and $M_C$ the structures in $\mathcal{B}$ and $\mathcal{C}$
with one vertex. Since we are working with  finite relational signatures,
$M_B$ and $M_C$ are finite sets. Clearly, $i_B[\mathcal{B}]$ is the class
of $i_C[M_C]$-free structures in the disjoint union $\mathcal{B\sqcup C}$,
and $i_C[\mathcal{C}]$
is the class of $i_B[M_B]$-free structures in $\mathcal{B\cup C}$. 
Thus, $i_B$ and $i_C$ are local embeddings and by definition of $F\oplus G$
we know that $F = (F\oplus G)\circ i_B$ and $G = (F\oplus G)\circ i_C$. 
Thus, by the second part of Proposition~\ref{prop:factor-restriction},
$\ex(F)\cup \ex(G)\subseteq \ex(F\oplus G)$. 
Moreover, since expressive powers are closed under unions
(second part of Proposition~\ref{prop:basicexpressiveproperties}), 
we conclude that $A \subseteq \ex(F\oplus G)$. 

In order to simplify notation, we identify $\mathcal{B}$ and $\mathcal{C}$ with
$i_B[\mathcal{B}]$ and $i_C[\mathcal{C}]$, respectively. Let $\mathcal{Q}$
be a local class relative $\mathcal{B\sqcup C}$. It is straightforward to verify
that $\mathcal{Q\cap B}$ and $\mathcal{Q\cap C}$
are local properties with respect to $\mathcal{B}$ and $\mathcal{C}$ respectively.
So, if $\mathcal{P}\in \ex(F\oplus G)$ and $\mathcal{Q\subseteq}_l
\mathcal{B\sqcup C}$ is such that  $\mathcal{P} = (F\oplus G)[\mathcal{Q}]$, then
$\mathcal{P} = F[\mathcal{Q\cap B}]\cup G[\mathcal{Q\cap C}]$
where $\mathcal{Q\cap B}\subseteq_l\mathcal{B}$ and 
$\mathcal{Q\cap C}\subseteq_l \mathcal{C}$. Clearly, 
$F[\mathcal{Q\cap B}]\subseteq \ex(F)$ and $G[\mathcal{Q\cap C}]\in \ex(G)$. 
Therefore, $\mathcal{P}\in  A$ and the claim follows.
 \end{proof}
 
The previous statement has the following consequence
regarding certain minimality of $\ex(F\oplus G)$.
 
\begin{corollary}\label{cor:disjointunion}
Consider three concrete functors $F$, $G$ and $H$ with codomain
$\mathcal{D}$.  If $\ex(F)\subseteq \ex(H)$ and $\ex(G)\subseteq \ex(H)$
then $\ex(F\oplus G)\subseteq \ex(H)$. 
\end{corollary}
\begin{proof}
By  Proposition~\ref{prop:basicexpressiveproperties}, expressive powers
are closed under disjoint unions. So, if $ex (F)\cup \ex(G)\subseteq \ex(H)$
then $\{\mathcal{P\cup Q\subseteq D}\colon
\mathcal{P}\in \ex(G), \mathcal{F}\in \ex(G)\} \subseteq \ex(H)$. Thus, by 
Proposition~\ref{prop:disjointunion}, $\ex(F\oplus G)\subseteq \ex(H)$.
\end{proof}

There is another natural but more elaborated way of combining 
expressive powers. Namely, their \textit{pullback} which we carefully
introduce now. 
Consider a pair $\mathcal{A}$ and $\mathcal{B}$ of classes of relational
structures with disjoint signatures $\sigma_A$ and $\sigma_B$ and denote
by $\tau$ the disjoint union $\sigma_A\sqcup \sigma_B$. Let $sh_A$ and
$sh_B$  be
the forgetful functors that map a $\tau$-structure to a $\sigma_A$-structure and
a $\sigma_B$-structure, respectively.
Let $F\colon\mathcal{A\to D}$ and $G\colon\mathcal{B\to D}$ be a pair
of concrete functors. The pullback
$P\colon\mathcal{A}\times_\calD\mathcal{ B\to D}$ is defined as follows. 
The class $\mathcal{A}\times_\calD\mathcal{B}$ is defined as the class of $\tau$-structures
$\bbC$ such that $sh_A(\bbC)\in \mathcal{A}$, $sh_B(\bbC)\in\mathcal{B}$
and $F(sh_A(\bbC)) = G(sh_B(\bbC))$. For every $\bbC\in\mathcal{B}\times_\calD\mathcal{C}$
we define  $P(\bbC)$ to be  $F(sh_A(\bbC))$, or equivalently
$P(\bbC)  = G(sh_B(\bbC))$. In other words, $\mathcal{A}\times_\calD\mathcal{B}$
is the maximal subclass of $sh_A^{-1}[\mathcal{A}]\cap sh_B^{-1}[\mathcal{B}]$ that
makes the following diagram commute. 

\begin{center}
\begin{tikzcd}[column sep=1.8cm, row sep=1.1cm]
 \mathcal{A}\times_\calD\mathcal{B} \arrow[d,"sh_A"] \arrow[r, "sh_B"] 
\arrow[dr, dashed, "P"] 
& \mathcal{B} \arrow[d,"G"] \\
 \mathcal{A} \arrow[r, "F"] & \mathcal{D} 
\end{tikzcd}
\end{center}

\begin{corollary}\label{cor:pullback-disjoint}
Let $F\colon\mathcal{A\to D}$ and $G\colon\mathcal{B\to D}$ be a pair
of surjective concrete functors. If
$P \colon\mathcal{A}\times_\calD\mathcal{ B\to D}$ is the pullback of
$F$ and $G$, then $\ex(F\oplus G)\subseteq \ex(P)$. 
\end{corollary}
\begin{proof}
It follows from the first part of Proposition~\ref{prop:factor-restriction}
that  $\ex(F)\cup \ex(G)\subseteq \ex(P)$. So, Corollary~\ref{cor:disjointunion}
implies that $\ex(F\oplus G)\subseteq \ex(H)$.
\end{proof}

Besides expressing unions of classes expressible by $F$ and $G$, the
pullback of $F$ and $G$ also expresses intersections of classes
locally expressible by $F$ with classes locally expressible by $G$.

\begin{lemma}\label{lem:pullback}
Let $F\colon\mathcal{A\to D}$ and $G\colon\mathcal{B\to D}$ be a pair
of surjective concrete functors. If $\calP\in \ex(F)$ and
$\calQ\in \ex(G)$, then $\mathcal{P\cap Q} \in \ex(F \times_\calD G)$.
\end{lemma}
\begin{proof}
    Let $\calA' \subseteq \calA$ (resp.\ $\calB'\subseteq \calB$)
    be a class with a finite set of minimal obstructions $\calM$ (resp.\ $\calN$),
    and such that $F[\calA'] = \mathcal{P}$ (resp.\ $G[\calB'] = \calQ$).
    It follows from Lemma~\ref{lem:preimages} that the intersection
    $sh_\calA^{-1}[\calA'] \cap sh_\calB^{-1}[\calB']$ is the class
    of $(sh_\calA^{-1}[\calN] \cup sh_\calB^{-1}[\calN])$-free structures
    in $\calA \times_\calD \calB$, and thus
    $sh_\calA^{-1}[\calA'] \cap sh_\calB^{-1}[\calB']$ is a local class. 
    Now we prove that $P[sh_\calA^{-1}[\calA'] \cap sh_\calB^{-1}[\calB']]
    = \calP \cap \calQ$. Clearly,
    $P[sh_\calA^{-1}[\calA'] \cap sh_\calB^{-1}[\calB']] \subseteq P[sh_\calA^{-1}[\calA']]
    \cap P[sh_\calB^{-1}[\calB']]$ and so, it follows from the definition of the pullback
    that  $P[sh_\calA^{-1}[\calA'] \cap sh_\calB^{-1}[\calB']] \subseteq \calP
    \cap \calQ$. Finally, suppose that a structure $\bbD \in \calP \cap\calQ$
    and let $\bbA \in \calA'$ and $\bbB\in \calB'$ be such that 
    $F(\bbA ) = \bbD  = G(\bbB)$. Thus, by considering the
    $(\sigma_\calA \sqcup \sigma_\calB)$-structure $\bbC$ whose
    $\sigma_\calA$-reduct is $\bbA$ and $\sigma_\calB$-reduct of $\bbB$, 
    we conclude that $\calC \in sh_\calA^{-1}[\calA'] \cap sh_\calB^{-1}[\calB']$
    and $P(\calC) =  F(sh_A(\bbC)) = G(sh_B(\bbC)) = \bbD$. Therefore, 
    $\calP\cap \calQ \subseteq P[sh_\calA^{-1}[\calA'] \cap sh_\calB^{-1}[\calB']]$
    and the claim follows. 
\end{proof}

We provide a simple application of Lemma~\ref{lem:pullback} in Example~\ref{ex:2-arc-col}.
Also, notice that  Lemma~\ref{lem:pullback}
implies that for any pair of classes $\calP,\calQ$ expressible by some
concrete functor $F\colon\calC\to \calD$, their intersection $\calP\cap\calQ$ is 
expressible by the pullback $F\times_\calD F$. A priory, $\calP\cap\calD$
could be locally expressible already by $F$, nonetheless we cannot find
any evident argument why this is true of false. 

\begin{question}\label{qst:intersectionclosed}
Is there a concrete functor $F$ such that $\ex(F)$ is not closed under intersections?
\end{question}

Proposition~\ref{prop:disjointunion} asserts that the disjoint unions $F\oplus G$
of a pair of concrete functors has the minimal expressive power that contains
$\ex(F)\cup \ex(G)$. It might be too much to hope for, but in light of Lemma~\ref{lem:pullback} 
it would be nice if the pullback $F\times_\calD G$ has the minimal expressive power
that contains $\calP\cap\calQ$ for every $\calP\in \ex(F)$ and $\calQ\in\ex(G)$.
In any case, we are interested in finding a description of the expressive power
of pullbacks similar to the description of the expressive power
of disjoint unions proposed in Proposition~\ref{prop:disjointunion}.

\begin{question}\label{qst:pullbackDescription}
Given a pair of surjective concrete functors $F\colon \mathcal{A\to D}$ and
$G\colon\mathcal{B\to D}$, is there a description of the expressive power of
their pullback in terms of the expressive powers of $F$ and $G$?
\end{question}


\subsection{Examples}
\label{sec:examples}

The aim of the rest of this section is to  provide several examples of
concrete functors $F\colon\mathcal{X\to G}$ that help illustrate the previous
definitions and observations. So far we have considered examples arising from
characterizations by forbidden (acyclic) orientations~\cite{guzmanAR,skrienJGT6},
by forbidden linear orderings~\cite{damaschkeTCGT1990,feuilloleyJDM},
and by forbidden circular orderings~\cite{guzmanAMC438}. As far as we are concerned,
the  examples proposed below have not yet been systematically studied (except 
for so-called tree-layouts~\cite{paulPREPRINT}), and we do not intend this
to be such a thorough study, but rather a handful of examples that could
motivate future research.\\

\noindent\textbf{Vertex-coloured graphs.}
One of the simplest way to equip graphs with further structure is by adding a
vertex colouring. In model theoretic terms, this corresponds to expanding the
signature of graph $\{E\}$ with unary predicates $U_1,\dots, U_k$, and considering
$\{E,U_1,\dots, U_k\}$-structures $\mathbb A$, where $(V(\bbA),E)$ is a graph
and $(U_1(\bbA),\dots, U_k(\bbA))$ is a partition of $V(\bbA)$ (with possibly
empty classes). Recall that we denote by $\calG_k$ the class of $k$-vertex coloured
graph, and by $FC_k\colon \calG_k\to\calG$ the functor the forgets the vertex-colouring. 
For $G\in\{K_2,2K_1\}$ and a pair of colours $i,j\in\{1,\dots, k\}$, we denote
by $(G,\mathbf{i},\mathbf{j})$ the vertex-colouring of $G$ where one vertex has
colour $i$ and the other has colour $j$.

\begin{example}\label{ex:M-partition}
    Let $M$ be a $k\times k$-matrix with $\{0,1,\ast\}$-entries. For each
    $i,j\in \{1,\dots,k\}$ define $\calF_k$ as follows: if $M_{ij} = \ast$ let
    $\calF_{ij} = \varnothing$; else, if $M_{ij} = 1$ let $\calF_{ij} = \{(2K_1,\mathbf{i},
    \mathbf{j})\}$; otherwise, if $M_{ij} = 0$ let $\calF_{ij} = \{(K_2,\mathbf{i},
    \mathbf{j})\}$. It is not hard to observe that a graph $G$ admits an $M$-partition
    (in the sense of \cite{hellEJC35}) if and only if $G$ admits an
    $(\cup_{i,j\in[k]} \calF_{ij})$-free $k$-vertex-colouring. Thus, local expressions
    by finitely many forbidden vertex-colour graphs extend the expressive power of matrix partitions. 
\end{example}

From the example above, it is natural to consider \textit{generalized} matrix partitions
by considering matrices with entries in $\{0,1,\ast,\oplus\}$, where $M_{ij} = \oplus$
represents forbidding the set $\{(2K_1,\mathbf{i}, \mathbf{j}), (K_2,\mathbf{i}, \mathbf{j})\}$.
This notion was introduced in~\cite{guzmanPHD}, but yet remains unexplored. 

\begin{example}
    Given a pair of hereditary classes $\calP, \calQ$ a $(\calP,\calQ)$-colouring
    of a graph $G$ is a partition $(X_1,X_2)$ of $V(G)$ such that the subgraph of
    $G$ induced by $X_1$ belongs to $\calP$, and the subgraph of $G$ induced by $X_2$
    belongs to $\calQ$. It is straightforward to observe that if $\calP$ and $\calQ$
    have finitely many minimal obstructions, then the class of $(\calP,\calQ)$-colourable
    graphs is locally expressible by $FC_2$.
\end{example}

Local expressions by forbidden vertex-coloured graphs corresponds to the so-called
\textit{monadic SNP} logic for graphs (see, e.g., \cite{federJC28}, or
subsections 1.4.2 and 1.4.3 in~\cite{bodirsky2021}).
Moreover, if we restrict to local class $\calC$ of vertex-coloured graphs that are also closed under inverse
homomorphisms (i.e., if $G$ is a vertex-coloured graph in $\calC$ and $G\to H$, then
$G\in \calC$),  then we recover monotone monadic SNP for graphs. It follows
from~\cite{federJC28,kunC33} and the celebrated CSP dichotomy~\cite{bulatovFOCS,zhukACM67} that
if a class of graphs $\calC$ is expressible in monotone monadic SNP (by means of forbidden
homomorphisms from vertex-coloured graphs), then the recognition problem of $\calC$ is either
in P or NP-complete.\\

\noindent\textbf{Equivalence graphs.} An \textit{equivalence graph} is a graph $G$
together with an equivalence relation $\sim$ on $V(G)$. For instance $(K_n,V(K_n)^2)$ is the complete
graph on $n$ vertices with one equivalence class. We denote the class of equivalence graphs
by $\mathcal{EQ}$, and by $\hyperlink{EQ}{sh_\sim\colon \mathcal{EQ}\to \mathcal{G}}$
the functor that forgets the equivalence relation. Given a graph $G$, we denote by
$EQ(G)$ the set of equivalence graphs  $(H,\sim)$ such that $H/\sim$ is isomorphic to $G$, 
and for each $v\in V(H)$ the quotient $(H-v)/\sim$ is not isomorphic to $G$. It is straightforward to 
observe that $EQ(G)$ is a finite set for for every finite graph $G$. The following example
is a particular instance of a more general example due to Bodirsky~\cite{bodirsky-personal}.
His example shows that every finite CSP with signature $\sigma$ is the reduct of a finitely
bounded class of $\sigma \cup \{\sim\}$-structures, where $\sim$ is a binary relation symbol
used to  describe an equivalence relation. 

\begin{example}\label{ex:equivalence-CSP}\cite{bodirsky-personal}
Let $EQ_n$ be the set of all equivalence graphs on $n$ vertices with unitary equivalence classes.
Given a graph $H$ on $n$ vertices let $\calF_H$ be the union of $\{(K_2,V(K_2)^2)\}$, of $EQ_{n+1}$, and of
the sets $EQ(H')$ where $H'$ ranges over all graphs on at most $n$ vertices that are not a subgraph of
$H$. It follows that a graph $G$ is $H$-colourable if and only if there is an equivalence
relation $\sim$ of $V(H)$ such that $(H,\sim)$ is an $\calF_H$-free equivalence graph. Thus, every
class of $H$-colourable graphs is locally expressible by $\hyperlink{EQ}{sh_\sim\colon \mathcal{EQ}\to
\mathcal{G}}$.
\end{example}

\noindent\textbf{Edge-coloured graphs.} After colouring vertices one can naturally
consider colouring edges. Given a positive integer $k$, we denote by $\mathcal{EG}_k$
the class of $2$-edge-coloured graphs, and let $\hyperlink{EG2}{sh_k\colon \mathcal{EG}_k\to
\mathcal{G}}$ be the functor that forgets the colourings. 
Evidently, the class of graphs with edge-chromatic number at most $k$ can
be described by finitely many forbidden $k$-edge-coloured graphs, i.e. locally
expressible by $\hyperlink{EG2}{sh_k}$.
And it turns out that several natural graph classes, such as line graphs 
of bipartite graphs and co-bipartite graphs, can be described by
means of forbidden $2$-edge-coloured graphs~\cite{bokIP}. Also,
amongst the many papers working towards solving the strong perfect
graph conjecture~\cite{berge1961}, Chv\'atal and Sbihi~\cite{chvatalJCTB44}
considered the class graphs that admit a $2$-edge-colouring with no
monochromatic induced  $P_3$, as building blocks for a decomposition theorem
for claw-free perfect graphs.

\begin{example}[from~\cite{bokIP}]\label{ex:co-bip}
Let $\calF$ be the set of $2$-edge-coloured graphs depicted in Figure~\ref{fig:2-edge-col}.
Notice that if $G$ is a co-bipartite graph $(X,Y)$, then by colouring all edges $XY$-edge with blue,
and all $XX$-edges and $YY$-edges with red, we obtain an $\calF$-free edge-colouring of $G$
--- an $XY$-edge (resp.\ $XX$-edge, and $YY$-edge), is an edge $xy$ such that $x\in X$ and $y\in Y$
(resp.\ $x,y\in X$, and $x,y\in Y$). 
Conversely, it is not hard to notice that if $G$ admits an $\calF$-free $2$-edge-colouring,
then either: the red edges in such a colouring are a cover of $G$ that induce two disjoint cliques;
or the red edges are a complete graph that covers all but one vertex of $G$; or $G$ has at most
two vertices. In either case, we conclude that $G$ is a co-bipartite graph.
\end{example}

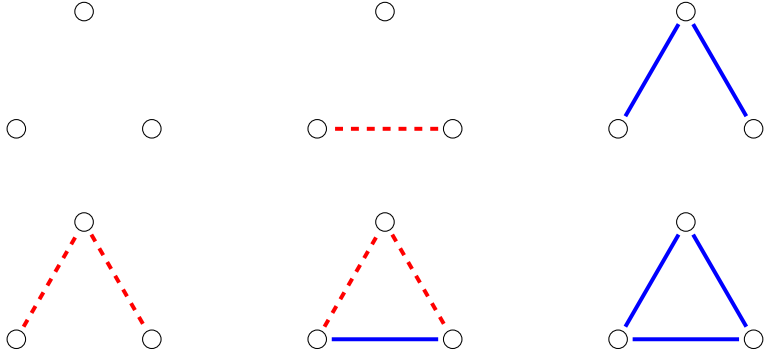
\begin{figure}[ht!]
\begin{center}
\begin{tikzpicture}[scale = 0.8]

\begin{scope}
    \foreach \i in {0,1,2}
      \draw ({(360/3)*\i + 90}:1.3) node(\i)[vertex]{};
      
  \end{scope}
  
  \begin{scope}[xshift=5cm]
    \foreach \i in {0,1,2}
      \draw ({(360/3)*\i + 90}:1.3) node(\i)[vertex]{};

    \draw [redE] (2) to (1);
  \end{scope}

    \begin{scope}[yshift=-3.5cm]
    \foreach \i in {0,1,2}
      \draw ({(360/3)*\i + 90}:1.3) node(\i)[vertex]{};

    \draw [redE] (1) to (0);
    \draw [redE] (0) to (2);
  \end{scope}

   \begin{scope}[xshift=5cm,yshift=-3.5cm]
    \foreach \i in {0,1,2}
      \draw ({(360/3)*\i + 90}:1.3) node(\i)[vertex]{};

    \draw [redE] (1) to (0);
    \draw [redE] (0) to (2);
    \draw [blueE]  (2) to (1);
  \end{scope}

  \begin{scope}[xshift=10cm,yshift=-3.5cm]
    \foreach \i in {0,1,2}
      \draw ({(360/3)*\i + 90}:1.3) node(\i)[vertex]{};
    
    \draw [blueE] (1) to (0);
    \draw [blueE]  (2) to (1);
    \draw [blueE] (2) to (0);
  \end{scope}

  \begin{scope}[xshift=10cm]
    \foreach \i in {0,1,2}
      \draw ({(360/3)*\i + 90}:1.3) node(\i)[vertex]{};

    \draw [blueE] (1) to (0);
    \draw [blueE] (0) to (2);
  \end{scope}

\end{tikzpicture}
\caption{A set $\calF$ of $2$-edge-coloured graph such that $G$ admits an $\calF$-free
$2$-edge-colouring if and only if $G$ is a co-bipartite graph (Example~\ref{ex:co-bip})}
\label{fig:2-edge-col}
\end{center}
\end{figure}

Similar as in the vertex case, expressions by means of forbidden homomorphisms of
edge-coloured graphs (a particular instance of forbidden induced edge-coloured graphs)
correspond to the logic  MMSNP$_2$ (see, e.g., Proposition 4.3.3 in~\cite{barsukovPhD}).
Contrary to MMSNP, it is still open whether the recognition problem of graphs
classes expressible in MMSNP$_2$ exhibit P vs.\ NP-complete dichotomy as finite domain CSPs.\\

\noindent\textbf{Digraphs.}
 Now, we consider  a slight variation on oriented expressions
 of graph classes. Namely, we consider the symmetric functor
$\hyperlink{DI}{S\colon \mathcal{DI\to G}}$. 
Since $\OR$ is the class of digraphs with no symmetric arcs, 
$\OR$ is a local subclass of $\DI$ and so, $\hyperlink{OR}{S_{|\OR}}$ is
a local restriction of $\hyperlink{DI}{S}$. Thus, it follows from
Proposition~\ref{prop:factor-restriction} that  $\ex(\hyperlink{OR}{S_{|\OR}})
\subseteq \ex(\hyperlink{DI}{S})$. 
A natural question that arises, is if by considering the extension of
$\hyperlink{OR}{S_{|\OR}}\colon \mathcal{OR\to G}$ to
$\hyperlink{DI}{S\colon\mathcal{DI\to G}}$ we actually increase the expressive power.

\begin{question}\label{qst:OR-S}
Is there a graph property locally expressible by $\hyperlink{DI}{S}$ but not by
$\hyperlink{OR}{S_{|\OR}}$?
\end{question}

As a brief note, observe that $\AO\subseteq \OR$, but $\AO$ is not a local class relative
to $\OR$ since all directed cycles are minimal obstructions of $\AO$ in $\OR$. Thus,
Proposition~\ref{prop:factor-restriction} does not imply that
$\ex(\hyperlink{AO}{S_{|\AO}}) \subseteq \ex(\hyperlink{OR}{S_{|\OR}})$. Actually, 
the previous inclusion does not hold since chordal graphs are expressible by
$\hyperlink{AO}{S_{|\AO}}$~\cite{skrienJGT6}, but are not expressible by
$\hyperlink{OR}{S_{|\OR}}$~\cite{guzmanEJC105}.

We conclude this brief dive into the expressive power of $\hyperlink{DI}{S}$
considering certain graph  families introduced by Gavril~\cite{gavrilIPL73}.
Given a hereditary class of graphs $\mathcal{P}$, a \textit{$\mathcal{P}$-mixed
graph}  is a graph $G$ whose edge set can be partitioned into $E_1$ and $E_2$,
where $(V(G),E_1)\in \mathcal{P}$, and $(V(G),E_2)$ admits a transitive orientation
$(V(G),E'_2)$ such that if $(x,y)\in E'_2(G)$ and $yz\in E_1$, then $xz\in E_1$.
These graph classes were introduced in order to describe certain intersection graphs.

\begin{example}\label{ex:Pmixed}
    If $\mathcal{P}$ is a local class of graphs,  then the class of $\mathcal{P}$-mixed
    graphs belongs to $\ex(S)$. Indeed, let $\calF$ be the set of minimal obstruction of
    $\mathcal P$, and let $M$ consists of $\overrightarrow{P}_3$
    and three digraphs with vertex set $\{x,y,z\}$ and edge sets $\{(x,y),(y,z),(z,y)\}$,
    $\{(x,y),(y,z),(z,y),(x,z)\}$, and $\{(x,y),(y,z),(z,y),(z,x)\}$ respectively.
    It is routine work to verify that a graph is a $\mathcal P$-mixed graph if and only 
    if there is an $(\mathcal F\cup M)$-free digraph $D$ such that $S(D) = G$.  Thus, the class
    of $\mathcal{P}$-mixed graphs is locally expressible by $\hyperlink{DI}{S}$.
\end{example}

\noindent\textbf{Orientations of edges and non-edges.} 
Denote by $\T_2$ the class of tournaments with a $\{red,blue\}$-arc
colouring, and let $\hyperlink{T2}{SB\colon \T_2\to \mathcal{G}}$ be the functor that 
deletes the red arcs, and takes the symmetric closure of the blue arcs. 
In other words, a tournament $T$ in $\T_2$ can be interpreted
as an orientation of the edges and non-edges of $\hyperlink{T2}{SB(T)}$.
It is not hard to notice that $\hyperlink{OR}{S_{|\OR}}$ factors
$\hyperlink{T2}{SB}$, e.g., the functor $F_b\colon \T_2\to \OR$
that deletes the red arcs, is a surjective functor such that the following
diagram commutes.
\begin{center}
\begin{tikzcd}[column sep=1.8cm, row sep=1.1cm]
  \T_2  \arrow[d, twoheadrightarrow, "F_b", dashed]
  \arrow[dr, "\hyperlink{T2}{SB}"] &  \\
   \OR  \arrow[r, "\hyperlink{OR}{S_{|\OR}}"] & \mathcal{G}
\end{tikzcd}
\end{center}
By Proposition~\ref{prop:factor-restriction}, and by the previous arguments
we conclude that the expressive power of $\hyperlink{T2}{SB}$ extends the
expressive power of forbidden orientations, i.e., of $\hyperlink{OR}{S_{\OR}}$.
It is also straightforward to observe that $\hyperlink{T2}{SB}$ extends
the expressive power of forbidden linear orderings. Indeed, consider the local
inclusion $i\colon \LOR \to \T_2$ defined as follows. 
Given a linearly ordered graph $(G,\le)$,  the $2$-arc-coloured tournament 
$i(G,\le)$ has the same vertex set as $G$, it has as a blue arc set, all edges
of $G$ oriented forward (according to $\le$), and its red arc set, contains
all pairs of non-adjacent vertices also oriented forward  (according to $\le$).
It is not hard to notice that this  construction defines an injective concrete
functor $i\colon\mathcal{LOR\to T}_2$. Moreover,  the image $i[\LOR]$ equals
the class of $2$-arc-coloured transitive tournaments\footnote{The
interpretation of linearly ordered graphs as $2$-arc-coloured transitive
tournaments was suggested by Reza Naserasr.}. So, $\hyperlink{LOR}{sh_L}$
is a local restriction of $\hyperlink{T2}{SB}$. This means that
$i$ is a local inclusion that makes the following diagram commute.
\begin{center}
\begin{tikzcd}[column sep=1.8cm, row sep=1.1cm]
 \T_2   \arrow[dr, "\hyperlink{T2}{SB}"] &  \\
  \LOR \arrow[u, hookrightarrow, "i", dashed] 
  \arrow[r, "\hyperlink{LOR}{sh_L}"] & \mathcal{G}
\end{tikzcd}
\end{center}
By Proposition~\ref{prop:factor-restriction}, and by previous arguments
we conclude that $\ex(\hyperlink{LOR}{sh_L})\cup
\ex(\hyperlink{OR}{S_{|\OR}})
\subseteq
\ex(\hyperlink{T2}{SB})$. 
In other words, every class expressible by forbidden orientations or
by forbidden linear orderings is expressible by forbidden $2$-arc-coloured
tournaments. The family of graph classes expressible by forbidden
orderings on three vertices is described in~\cite{feuilloleyJDM}, and
the family of graph classes expressible by forbidden oriented graphs on
three vertices is described in~\cite{guzmanAR}. Thus, it is only
natural to consider the following problem.

\begin{problem}\label{prob:expressions2arc}
List the graph classes expressible by forbidden $2$-arc-coloured
tournaments on three vertices.
\end{problem}

Towards answering Problem~\ref{prob:expressions2arc}, it is convenient
to regard $\hyperlink{T2}{SB\colon\T_2\to\mathcal{G}}$  as
the pullback of a pair of functors. Let $\overline S\colon\mathcal{DI\to G}$ be the
functor that maps a digraph to the complement of its symmetric
closure, i.e., for each digraph $G'$, the equality
$\overline S(G') = \overline{S(G')}$ holds.
Finally, let $F_b\colon\T_2\to \OR$ be the functor
that deletes the red arcs (and preserves the blue arcs), and
$F_r\colon\T_2\to \OR$ be the functor that deletes
the blue arcs (and preserves the red arcs). It is not hard to notice
that  $\hyperlink{T2}{SB}$ is the pullback of $\hyperlink{OR}{S_{|\OR}}$
and of $\overline S_{|\OR}$. 

\begin{center}
\begin{tikzcd}[column sep=1.8cm, row sep=1.1cm]
 \T_2 \arrow[d,"F_r"] \arrow[r, "F_b"] 
\arrow[dr, dashed, "\hyperlink{T2}{SB}"] 
& \OR \arrow[d,"\hyperlink{OR}{S_{|\OR}}"] \\
 \OR \arrow[r, "\overline S_{|\OR}"] & \mathcal{G} 
\end{tikzcd}
\end{center}

Therefore, by Lemma~\ref{lem:pullback}, if a pair of graph classes
$\calP$ and $\calQ$ are expressible by forbidden orientations, 
then the intersection $\calP\cap {\co}\calQ$ is locally expressible
by means of forbidden $2$-arc-coloured tournaments, i.e., 
$\calP\cap {\co}\calQ \in \ex(\hyperlink{T2}{SB})$. Moreover, 
by examining the proof of Lemma~\ref{lem:pullback} one can 
see that if $\calP$ and $\calQ$ are expressible by forbidden orientations
on three vertices,  then the intersection $\calP\cap {\co}\calQ$ is locally
expressible by means of forbidden $2$-arc-coloured tournaments on three vertices. 

\begin{example}\label{ex:2-arc-col}
    A connected graph $G$ is a proper circular-arc graph if and only if admits an
    orientation $G'$ where the in- and out-neighbourhood of every vertex
    is a tournament~\cite{skrienJGT6}. Also, it follows from the
    Roy-Gallai-Hasse-Vitaver Theorem~\cite{gallaiAMA18,hasseIMN28,royIRO1,vitaverDAN147}
    that a graph $G$ is bipartite if and only if it admits an orientation $G'$ with no
    directed walk on two vertices. It is not hard
    to translate both of these characterizations into expressions by forbidden oriented
    graphs on $3$ vertices (see, e.g., Proposition 15 in~\cite{guzmanAR}). So, using
    these expressions and the proof of Lemma~\ref{lem:pullback}, we see that a graph is a
    proper circular-arc co-bipartite graph if and only if its edges and non-edges can be
    oriented in such a way that avoids the $2$-edge-coloured tournaments depicted in 
    Figure~\ref{fig:2-arc-col}.
\end{example}

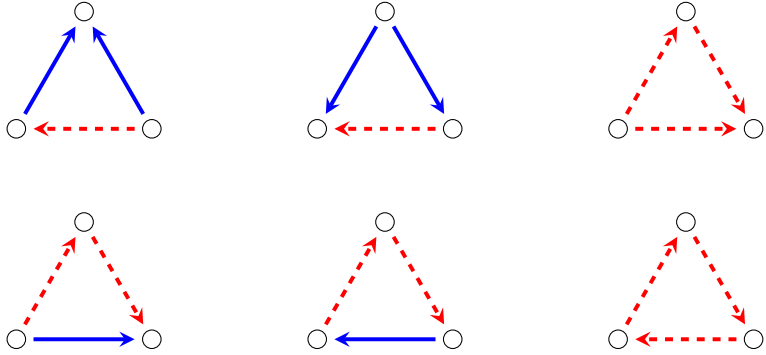
\begin{figure}[ht!]
\begin{center}
\begin{tikzpicture}[scale = 0.8]

\begin{scope}
    \foreach \i in {0,1,2}
      \draw ({(360/3)*\i + 90}:1.3) node(\i)[vertex]{};
      
    \draw [blueA] (1) to (0);
    \draw [blueA]  (2) to (0);
    \draw [redA] (2) to (1);
  \end{scope}
  
  \begin{scope}[xshift=5cm]
    \foreach \i in {0,1,2}
      \draw ({(360/3)*\i + 90}:1.3) node(\i)[vertex]{};

    \draw [blueA] (0) to (1);
    \draw [blueA]  (0) to (2);
    \draw [redA] (2) to (1);
  \end{scope}

    \begin{scope}[yshift=-3.5cm]
    \foreach \i in {0,1,2}
      \draw ({(360/3)*\i + 90}:1.3) node(\i)[vertex]{};

    \draw [redA] (1) to (0);
    \draw [redA] (0) to (2);
    \draw [blueA] (1) to (2);
  \end{scope}

   \begin{scope}[xshift=5cm,yshift=-3.5cm]
    \foreach \i in {0,1,2}
      \draw ({(360/3)*\i + 90}:1.3) node(\i)[vertex]{};
    
    \draw [redA] (1) to (0);
    \draw [blueA]  (2) to (1);
    \draw [redA] (0) to (2);
  \end{scope}

  \begin{scope}[xshift=10cm,yshift=-3.5cm]
    \foreach \i in {0,1,2}
      \draw ({(360/3)*\i + 90}:1.3) node(\i)[vertex]{};
    
    \draw [redA] (1) to (0);
    \draw [redA] (0) to (2);
    \draw [redA] (2) to (1);
  \end{scope}

  \begin{scope}[xshift=10cm]
    \foreach \i in {0,1,2}
      \draw ({(360/3)*\i + 90}:1.3) node(\i)[vertex]{};

    \draw [redA] (1) to (0);
    \draw [redA] (0) to (2);
    \draw [redA] (1) to (2);
  \end{scope}

\end{tikzpicture}
\caption{A set $\calF$ of $2$-arc-coloured tournaments that such that
a graph $G$ admits an orientation of edges and non-edges 
that avoids $\calF$ if and only if it is a proper circular-arc co-bipartite graph
--- orientations of edges are represented by blue arcs, and orientations of non-edges
are represented by red dashed arcs.}
\label{fig:2-arc-col}
\end{center}
\end{figure}

Recall that all graph classes characterized by a set of forbidden
linear ordered graphs on three vertices is recognizable in polynomial time
\cite{hellESA2014}. Since $2$-arc-coloured tournaments extend linear
orderings of graphs (as explained above), one can naturally consider the
following question. 

\begin{problem}\label{prob:liftingarc}
Determine the complexity of recognizing graph classes expressible by
forbidden $2$-arc-coloured tournaments on three vertices.
\end{problem}

\noindent\textbf{Partially and suitably ordered graphs.}
Above, we considered a simple but non-obvious way of extending local expressions by
forbidden linear orderings. The following examples arise as obvious extensions of
$\hyperlink{LOR}{sh_L\colon\mathcal{LOR\to G}}$.
A \textit{partially ordered graph}  is a graph $G$ together with a partial ordering
$\le$ of $V(G)$.  For instance, $(K_n,\varnothing)$ is the partially ordered graph
on $n$ pairwise adjacent vertices such that all vertices are incomparable in
the partial ordering. We denote by $\PO$ the class of partially
ordered graphs. A partially ordered graph  $(G,\le)$ is \textit{suitably ordered}
if for every pair $x$ and $y$ of adjacent vertices then  $x\le y$ or $y\le x$.
In other words, a suitably ordered graph is an $\{E,\le\}$-structure $(V,E,\le)$,
where $(V,E)$ is a spanning subgraph of the comparability
graph of $(V,\le)$. We denote by $\SO$ the class of 
suitable ordered graphs.
Note that if $(G,\le)$ is a linearly ordered graph, then the comparability
graph $C$ of $(G,\le)$ is a complete graph, so $G$ is a
spanning subgraph of $C$, i.e., $\LOR\subseteq \SO$. We denote by
$\hyperlink{PO}{sh_P\colon\mathcal{PO\to G}}$ and $\hyperlink{SO}{sh_S\colon\mathcal{SO\to G}}$
the functors that forget the ordering of partially and suitably ordered graphs.

\begin{example}
It is not hard to notice that $\LOR$ is a local subclass of $\SO$
since $\LOR$ is the class of suitably ordered graphs with no
pair of incomparable non-adjacent vertices. Similarly, $\SO$ is the
class of partially ordered graphs with no pair of incomparable adjacent vertices.
Thus, $\hyperlink{LOR}{sh_L}$ is a local restriction of $\hyperlink{SO}{sh_S}$ and
the latter is a local restriction of $\hyperlink{PO}{sh_P}$. Therefore, by
the second part of Proposition~\ref{prop:factor-restriction} we conclude
that  $\ex(\hyperlink{LOR}{sh_L}) \subseteq \ex(\hyperlink{SO}{sh_S})\subseteq
\ex(\hyperlink{PO}{sh_P})$.
\end{example}

Again, we are interested in knowing if either of the extensions of
$\hyperlink{LOR}{sh_L}$ to
$\hyperlink{SO}{sh_S}$ or $sh_S$ to $\hyperlink{PO}{sh_P}$ increase the
expressive power. 

\begin{question}\label{qst:AO-L-S-P}
Which of the inclusions $\ex(\hyperlink{AO}{S_{|\AO}})
\subseteq \ex(\hyperlink{LOR}{sh_L}) \subseteq
\ex(\hyperlink{SO}{sh_S})\subseteq \ex(\hyperlink{PO}{sh_P})$ are proper inclusions?
\end{question}

Allow us to exhibit some graph properties naturally expressible by
$\hyperlink{SO}{sh_S}$ (and thus by $\hyperlink{PO}{sh_P}$).

\begin{example}\label{ex:comparability}
    Denote by $(2K_1,\le)$ the unique linear ordering of $2K_1$. Clearly, a graph 
    $G$ admits a $(2K_1,\le)$-free linear ordering if and only if $G$ is a complete graph.
    Equivalently, if $G$ is the comparability graph of a linear ordered set. Similarly, 
    there is a $(2K_1,\le)$-free suitable ordering of $G$ if and only if $G$ is a
    comparability graph. 
\end{example}

Before considering the following example, recall that the \textit{height} of a poset $P$ is the
maximum length  of a chain in $P$.

\begin{example}\label{ex:comparabilityheightk}
    Notice that for every complete graph $K_n$ there is a unique element in
    $\hyperlink{SO}{sh_S^{-1}(K_n)}$; which we denote by $(K_n,\le)$. It is
    straightforward to observe that a graph if $\mathcal{Q}$ is
    the class of $\{(2K_1,\le), (K_{k+1},\le)\}$-free suitably ordered graphs,
    then $\hyperlink{SO}{sh_S[\mathcal{Q}]}$ is the class of comparability
    graphs of posets of height at most $k$. Thus, for every positive integer $k$
    the class of comparability graphs of height at most $k$ is locally expressible by
    $\hyperlink{SO}{sh_S\colon\mathcal{SO\to G}}$. 
\end{example}

\noindent\textbf{Topology and partially ordered graphs.}
It is well-known that finite topological spaces can be encoded as 
pre-ordered sets --- the minimal open neighborhood of a point $x$ is encoded by 
the upward-set of $x$ in the pre-order. This coding works for general
\textit{Alexandroff spaces}, i.e., a topological space where every point has a
minimal open neighborhood.  Moreover, the category of $T_0$-Alexandroff spaces is isomorphic to
the category of partially ordered sets~\cite{goubault2013}.  Hence, we can think of the class $\PO$
as the class of graphs together with a $T_0$-Alexandroff topology on its vertex set.
In this context, a suitably ordered graph $(G,\le)$ is a graph together with an
Alexandroff topology $\tau$ such that if a subspace $(D,\tau_D)$ is a discrete
topological space, then $D$ is an independent set of vertices.  We would be very
interested to see if this topological interpretation of partially and suitably
ordered graphs could yield some insight on the expressive powers of $\hyperlink{PO}{sh_P}$
and $\hyperlink{SO}{sh_S}$.
Also, there might be certain classes $\mathcal{Q}$ of graphs with a topology
on its vertex set  such that the family
$\hyperlink{PO}{sh_P[\mathcal{Q}]}$ has interesting structural properties.
To illustrate this idea we propose the following straightforward characterizations.

\begin{example}\label{ex:POcomp}
    A graph $G$ is a comparability graph if and only if there is a $T_0$-topology
    $\tau$ on $V(G)$ such that, for every $D\subseteq V(G)$, the subspace
	$(D,\tau_D)$ is a discrete topological space if and only if $D$ is an
	independent set of vertices.
\end{example}

\begin{example}\label{ex:POkcol}
For every positive integer $k$ and every graph $G$ with vertex set $V$,
the following statements are equivalent.
\begin{itemize}
	\item $G$ is $k$-colourable.
	\item There is a topology $\tau$ on $V$ such that each of its connected subspaces
 	is an independent set of vertices, and $(V,\tau)$ has at most $k$
	connected components.
	\item There is a $T_0$-topology $\tau$ on $V$ such that every discrete subspace
	is an independent set of vertices and any subspace $(D,\tau_D)$
	with at least $k+1$ points, contains a non-trivial discrete subspace.
\end{itemize}
\end{example}

\noindent\textbf{Genealogical graphs.}  A \textit{genealogical graph}
is a suitably ordered graph $(G,\le)$ such that the down-set of every vertex
is a chain; that is, for every $x\in V(G)$ if $y\le x$ and $z\le x$, then $y$ and $z$ are
comparable. Genealogical graphs can be coded by a function $p\colon V\to V$
defined by $p(x) = x$ for every $\le$-minimal element, and otherwise
$p(x) = \max\{y\in V(G)\colon y\le x, y\neq x\}$.
We interpret  $p$ as a parent function,  and clearly,  $\le$ is
recovered by the ancestor-descendent relation with respect to $p$.

\begin{example}
    Consider a graph $G$ is a graph, and let $p\colon V(G)\to V(G)$ be a parent
    function obtained from a  DFS-search in $G$. In this case,  $G$  together with the
    ancestor-descendent order $\le_p$ defined by $p$ is a genealogical
    graph $(G,\le_p)$.
\end{example}

It turns out that genealogical graphs are closely related to the independently
introduced notion of \textit{tree-layouts} of graphs from Paul and Protopapas~\cite{paulPREPRINT}.

\begin{example}\label{ex:tree-layout}
    A \emph{tree-layout}~\cite{paulPREPRINT} of a graph $G$ is an ordered triple $(T,r,\rho)$
    where $T$ is a tree with root $r$ and $\rho\colon V(G)\to V(T)$ is a bijection such
    that for each $xy\in E(G)$ the vertex $\rho(x)$ is an ancestor or a descendent of $\rho(y)$. 
    Clearly, $G$ together with the ancestor-descendent  relation  $\le_T$ defined by the
    tree-layout $T_G$ is a genealogical graph. Conversely, if $(G,\le)$ is a genealogical graph
    with a unique $\le$-minimal element $r$, then $(G,\le)$ can be coded as a tree-layout using the parent
    function defined above, where $r$ will be the root of the tree-layout.
\end{example}

Denote by $\mathcal{GEN}$ the class of genealogical
graphs and by $\hyperlink{GEN}{sh_G\colon \mathcal{GEN\to G}}$ the restriction of
$\hyperlink{SO}{sh_S}$ to $\mathcal{GEN}$.
It should not be hard to notice that $\mathcal{GEN}$ is a local set in
$\SO$ and $\LOR$ is a local set relative to $\mathcal{GEN}$.
So, with similar arguments as before we conclude that 
$\ex(\hyperlink{LOR}{sh_L}) \subseteq \ex(\hyperlink{GEN}{sh_G})
\subseteq \ex(\hyperlink{SO}{sh_S})$.

An embedding of tree-layout $T_G$ of a graph $G$ into a tree-layout $T_H$ of a graph $H$
is a graph embedding that preserves the ancestor-descendent relation with respect to the tree-layouts
$T_G$ and $T_H$. We say that a tree-layout $T_G$ of a graph $G$ is \textit{non-separable} if 
for each $x\in V(G)$ there is an edge $yz\in E(G)$ such that $\rho(y)$ is an ancestor of
$x$ and either $x = z$ or $\rho(z)$ is a descendent of $\rho(x)$. It is straightforward to
observe that if $T_G$ and $T_H$ are tree-layouts of a pair of graphs $G,H$,  a function
$f\colon V(G)\to V(H)$ is an embedding with respect to the tree-layouts
$T_G$ and $T_H$ if and only $f\colon (G,\le_T)\to (H,\le_T)$ is an embedding of genealogical
graphs as relational structures (see Example~\ref{ex:tree-layout} for the construction of $(G,\le_T)$.
Moreover, if $T_G$ is a non-separable tree-layout, then a graph $G'$ admits a tree-layout
that avoids $T_G$ if and only if it admits a partial ordering $\le'$ such that $(G,\le')$
is an $(G,\le_T)$-free genealogical graph. 

\begin{lemma}\label{lem:tree-layout}
Let $\calC$ be class of graphs. If $\calC$ admits a description by means of a finite forbidden
set of non-separable tree-layouts $\calF$, then $\calC$ is expressible by genealogical graphs,
i.e., $\calC\in \ex(\hyperlink{GEN}{sh_G})$. 
\end{lemma}
\begin{proof}
    This follows by coding the finite set of tree-layout $\calF$ as genealogical graphs
    according to Example~\ref{ex:tree-layout}, and then following the arguments in the paragraph above.
\end{proof}

\begin{example}
    Consider the finite set $\calF$ of consisting of the two linearly ordered graphs with vertex
    $\{1,2,3\}$,  ordering $1\le 2\le 3$, and such that $13\in E$ but $12\not\in E$. 
    It is well-known that a graph $G$ admits an $\calF$-free linear ordering  if and only if
    $G$ is an interval graph~\cite{feuilloleyJDM}. Now, if we allow partial orderings that are
    not necessarily total, but such that $(G,\le)$ is a genealogical graph, then a graph
    $G$ admits such an $\calF$-free partial ordering if and only if $G$ is  a chordal graph. 
    This follows from the  tree-layout characterization of chordal graphs~\cite{paulPREPRINT}, 
    via Lemma~\ref{lem:tree-layout} using the fact that $\calF$ consists of non-separable
    tree-layouts.
\end{example}

Similar as in the previous example, Paul and Protopapas show that connected
trivially-perfect graphs correspond to tree-layout that avoid the unique linear ordering
of $2K_1$. If we now consider genealogical graphs, we recover the whole class of
trivially-perfect graphs. Recall that $(2K_1,\le)$ denoted the unique linear ordering
of $2K_1$. 

\begin{example}
    In Example~\ref{ex:comparability} we noticed that $(2K_1,\le)$ characterizes
    complete graphs by means of forbidden linear orderings, and comparability 
    graph by means of suitably ordered graphs. Now, notice that a graph $G$ admits a partial 
    ordering $\le$ such that $(G,\le)$ is a $(2K_1,\le)$-free genealogical graph if and only
    if $G$ is the comparability graph of a poset where the down-set of ever vertex is a chain. 
    The latter are called trivially-perfect graphs, so $(2K_1,\le)$ characterizes
    trivially-perfect graphs by means of forbidden genealogical graphs.
\end{example}

\noindent\textbf{Linear orderings and orientations.}
To conclude this section, we consider two examples arising as
pullbacks of previously considered functors. First, we consider
the pullback $u_L\colon \OR \times_\calG \LOR \to \calG$. 
It is straightforward to observe that structures in $\OR \times_\calG \LOR$,
can be encoded as linearly ordered oriented graphs, and $u_L$ 
maps a linearly ordered oriented graph to its underlying graph. 
We write $\LOOR$ to denote the class of linearly ordered oriented graphs,
and in terms of commuting diagrams, $F$ forgets the linear ordering 
of a linearly ordered graph, and $G$ preservers the ordering 
and takes the symmetric closure of the oriented arcs, then $u_L$
is the functor that makes the following diagram commute.

\begin{center}
\begin{tikzcd}[column sep=1.8cm, row sep=1.1cm]
 \LOOR \arrow[d,"F"] \arrow[r, "G"] 
\arrow[dr, dashed, "\hyperlink{LOOR}{u_L}"] 
& \LOR \arrow[d,"\hyperlink{LOR}{sh_L}"] \\
 \OR \arrow[r, "S_{|\OR}"] & \calG
\end{tikzcd}
\end{center}

There is no set of linearly ordered graphs nor a set of oriented graphs on
three vertices that characterizes $3$-colourable graphs by forbidden linear
orderings nor by forbidden orientations, respectively. The following result
shows that there is such a set  when we consider the expressive
power of the pullback $\hyperlink{LOOR}{u_L\colon\mathcal{LOOR\to G}}$ --- we
denote this set by $\mathcal{F}_3$, and we depict the linearly ordered oriented
graphs  that belong to it in Figure~\ref{fig:LOOR}.

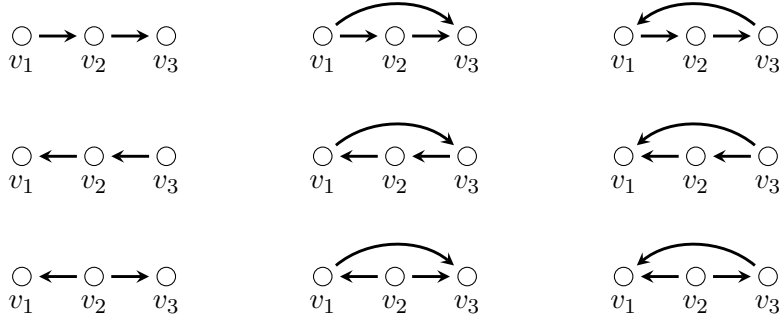
\begin{figure}[ht!]
\begin{center}

\begin{tikzpicture}[scale = 0.8]

\begin{scope}[yshift = 4cm, scale=0.6]
\node [vertex, label=below:{$v_1$}] (1) at (-2,0){};
\node [vertex, label=below:{$v_2$}] (2) at (0,0){};
\node [vertex, label=below:{$v_3$}] (3) at (2,0){};

\draw[arc]    (1)  edge  (2);
\draw[arc]    (2) edge (3);
\end{scope}

\begin{scope}[xshift=5cm, yshift = 4cm, scale=0.6]
\node [vertex, label=below:{$v_1$}] (1) at (-2,0){};
\node [vertex, label=below:{$v_2$}] (2) at (0,0){};
\node [vertex, label=below:{$v_3$}] (3) at (2,0){};

\draw[arc]    (1)  edge [bend right=-40] (3);
\draw[arc]    (1)  edge  (2);
\draw[arc]    (2) edge (3);
\end{scope}

\begin{scope}[xshift=10cm, yshift = 4cm, scale=0.6]
\node [vertex, label=below:{$v_1$}] (1) at (-2,0){};
\node [vertex, label=below:{$v_2$}] (2) at (0,0){};
\node [vertex, label=below:{$v_3$}] (3) at (2,0){};

\draw[arc]    (3)  edge [bend left=-40] (1);
\draw[arc]    (1)  edge  (2);
\draw[arc]    (2) edge (3);
\end{scope}

\begin{scope}[ yshift = 2cm, scale=0.6]
\node [vertex, label=below:{$v_1$}] (1) at (-2,0){};
\node [vertex, label=below:{$v_2$}] (2) at (0,0){};
\node [vertex, label=below:{$v_3$}] (3) at (2,0){};

\draw[arc]    (3)  edge  (2);
\draw[arc]    (2) edge (1);
\end{scope}

\begin{scope}[xshift=5cm, yshift = 2cm, scale=0.6]
\node [vertex, label=below:{$v_1$}] (1) at (-2,0){};
\node [vertex, label=below:{$v_2$}] (2) at (0,0){};
\node [vertex, label=below:{$v_3$}] (3) at (2,0){};

\draw[arc]    (1)  edge [bend right=-40] (3);
\draw[arc]    (3)  edge  (2);
\draw[arc]    (2) edge (1);
\end{scope}

\begin{scope}[xshift=10cm, yshift = 2cm, scale=0.6]
\node [vertex, label=below:{$v_1$}] (1) at (-2,0){};
\node [vertex, label=below:{$v_2$}] (2) at (0,0){};
\node [vertex, label=below:{$v_3$}] (3) at (2,0){};

\draw[arc]    (3)  edge [bend left=-40] (1);
\draw[arc]    (3)  edge  (2);
\draw[arc]    (2) edge (1);
\end{scope}

\begin{scope}[scale=0.6]
\node [vertex, label=below:{$v_1$}] (1) at (-2,0){};
\node [vertex, label=below:{$v_2$}] (2) at (0,0){};
\node [vertex, label=below:{$v_3$}] (3) at (2,0){};

\draw[arc]    (2)  edge  (1);
\draw[arc]    (2) edge (3);
\end{scope}

\begin{scope}[xshift=5cm, scale=0.6]
\node [vertex, label=below:{$v_1$}] (1) at (-2,0){};
\node [vertex, label=below:{$v_2$}] (2) at (0,0){};
\node [vertex, label=below:{$v_3$}] (3) at (2,0){};

\draw[arc]    (1)  edge [bend right=-40] (3);
\draw[arc]    (2)  edge  (1);
\draw[arc]    (2) edge (3);
\end{scope}

\begin{scope}[xshift=10cm, scale=0.6]
\node [vertex, label=below:{$v_1$}] (1) at (-2,0){};
\node [vertex, label=below:{$v_2$}] (2) at (0,0){};
\node [vertex, label=below:{$v_3$}] (3) at (2,0){};

\draw[arc]    (3)  edge [bend left=-40] (1);
\draw[arc]    (2)  edge  (1);
\draw[arc]    (2) edge (3);
\end{scope}

\end{tikzpicture}

\caption{A set of structures that characterizes $3$-colourable graphs by means of 
forbidden linearly ordered oriented graphs.
In each case the linear ordering is $v_1\le v_2\le v_3$.}
\label{fig:LOOR}
\end{center}
\end{figure}

\begin{example}\label{ex:3colLOOR}
Let $\calF$ be the set of linearly ordered oriented graphs
depicted in Figure~\ref{fig:LOOR}, and let $G$ be a $3$-colourable graph
with colour classes $X_1,X_2,X_3$.
Consider any linear ordering $\le$ of $V(G)$ such that $x\le y$ for every
$x\in X_i$ and $y\in X_j$ for $1\le i < j \le 3$, and let $G'$ be the orientation $G'$
of $G$ where all vertices in $X_2$ have out-degree zero, and all vertices in
$X_1$ have in-degree zero. It is not hard to notice that $(G',\le)$ is an
$\calF$-free ordered oriented graph and so, every $3$-colourable graph
can be expanded to a $\calF$-free linearly ordered oriented graph.
Conversely, suppose that $(G',\le)$ is  an $\calF$-free linearly ordered oriented
graph such that $\hyperlink{LOOR}{u_L(G',\le)} = G$. Consider the partition
$(X_1,X_2,X_3)$ defined as follows.
It is straightforward to observe that $G$ is $3$-colourable by considering
the sets $X_1 = \{x\in V(G)\colon$ there is vertex $y\in N^+(x), x\le y\}$,
$X_2 = \{x\in V(G)\colon$ there is vertex $y\in N^+(x), y\le x\}$, 
and  $X_3 = V(G)\setminus (X_1\cup X_2)$. 
Therefore, a graph $G$ is $3$-colourable if and only if it can be expanded to
a $\calF$-free linearly ordered oriented graph.
\end{example}

Oriented graphs and $2$-edge-coloured graphs 
 are very similar structures
in the sense that for every edge there are two possible orientations
and two possible colours to choose from. Nonetheless, there is no isomorphism
between the categories of oriented graphs and $2$-edge-coloured graphs
that preserves the underlying graph structure. To see this, note that
there are two non-isomorphic orientations of the triangle, while there are
four non-isomorphic $2$-edge-coloured triangles. On the contrary, 
the categories $\mathcal{LOR\times_G OR}$ and $\mathcal{LOR\times_G EG}_2$
(where $\hyperlink{EG2}{\mathcal{EG}_2}$ is the class of $2$-edge-coloured graphs)
are isomorphic. To see this,  simply interpret blue symmetric edges as forward edges
(with respect to the linear ordering), and red edges as backward edges, and notice
that is a one-to-one coding between linearly ordered $2$-edge-coloured (symmetric) graphs
and linearly ordered oriented graphs.
In particular, this shows that the expressive powers of the pullbacks
$P_1\colon\mathcal{LOR\times_G OR\to G}$ and 
$P_2\colon\mathcal{LOR\times_G EG}_2\to\mathcal{G}$ are the same.

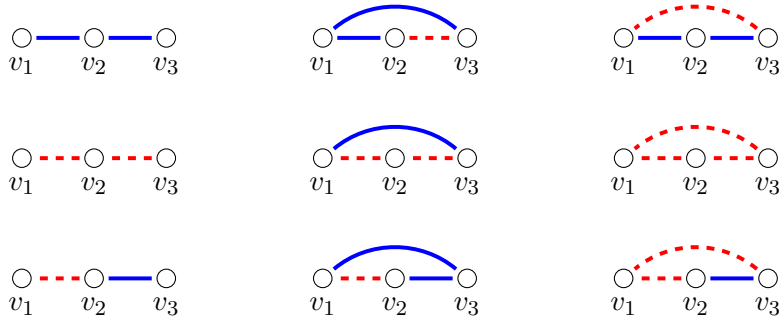
\begin{figure}[ht!]
\begin{center}

\begin{tikzpicture}[scale = 0.8]

\begin{scope}[yshift = 4cm, scale=0.6]
\node [vertex, label=below:{$v_1$}] (1) at (-2,0){};
\node [vertex, label=below:{$v_2$}] (2) at (0,0){};
\node [vertex, label=below:{$v_3$}] (3) at (2,0){};

\draw[blueE]    (1)  edge  (2);
\draw[blueE]    (2) edge (3);
\end{scope}

\begin{scope}[xshift=5cm, yshift = 4cm, scale=0.6]
\node [vertex, label=below:{$v_1$}] (1) at (-2,0){};
\node [vertex, label=below:{$v_2$}] (2) at (0,0){};
\node [vertex, label=below:{$v_3$}] (3) at (2,0){};

\draw[blueE]    (1)  edge [bend right=-40] (3);
\draw[blueE]    (1)  edge  (2);
\draw[redE]    (2) edge (3);
\end{scope}

\begin{scope}[xshift=10cm, yshift = 4cm, scale=0.6]
\node [vertex, label=below:{$v_1$}] (1) at (-2,0){};
\node [vertex, label=below:{$v_2$}] (2) at (0,0){};
\node [vertex, label=below:{$v_3$}] (3) at (2,0){};

\draw[redE]    (3)  edge [bend left=-40] (1);
\draw[blueE]    (1)  edge  (2);
\draw[blueE]    (2) edge (3);
\end{scope}

\begin{scope}[ yshift = 2cm, scale=0.6]
\node [vertex, label=below:{$v_1$}] (1) at (-2,0){};
\node [vertex, label=below:{$v_2$}] (2) at (0,0){};
\node [vertex, label=below:{$v_3$}] (3) at (2,0){};

\draw[redE]    (3)  edge  (2);
\draw[redE]    (2) edge (1);
\end{scope}

\begin{scope}[xshift=5cm, yshift = 2cm, scale=0.6]
\node [vertex, label=below:{$v_1$}] (1) at (-2,0){};
\node [vertex, label=below:{$v_2$}] (2) at (0,0){};
\node [vertex, label=below:{$v_3$}] (3) at (2,0){};

\draw[blueE]    (1)  edge [bend right=-40] (3);
\draw[redE]    (3)  edge  (2);
\draw[redE]    (2) edge (1);
\end{scope}

\begin{scope}[xshift=10cm, yshift = 2cm, scale=0.6]
\node [vertex, label=below:{$v_1$}] (1) at (-2,0){};
\node [vertex, label=below:{$v_2$}] (2) at (0,0){};
\node [vertex, label=below:{$v_3$}] (3) at (2,0){};

\draw[redE]    (3)  edge [bend left=-40] (1);
\draw[redE]    (3)  edge  (2);
\draw[redE]    (2) edge (1);
\end{scope}

\begin{scope}[scale=0.6]
\node [vertex, label=below:{$v_1$}] (1) at (-2,0){};
\node [vertex, label=below:{$v_2$}] (2) at (0,0){};
\node [vertex, label=below:{$v_3$}] (3) at (2,0){};

\draw[redE]    (2)  edge  (1);
\draw[blueE]    (2) edge (3);
\end{scope}

\begin{scope}[xshift=5cm, scale=0.6]
\node [vertex, label=below:{$v_1$}] (1) at (-2,0){};
\node [vertex, label=below:{$v_2$}] (2) at (0,0){};
\node [vertex, label=below:{$v_3$}] (3) at (2,0){};

\draw[blueE]    (1)  edge [bend right=-40] (3);
\draw[redE]    (2)  edge  (1);
\draw[blueE]    (2) edge (3);
\end{scope}

\begin{scope}[xshift=10cm, scale=0.6]
\node [vertex, label=below:{$v_1$}] (1) at (-2,0){};
\node [vertex, label=below:{$v_2$}] (2) at (0,0){};
\node [vertex, label=below:{$v_3$}] (3) at (2,0){};

\draw[redE]    (3)  edge [bend left=-40] (1);
\draw[redE]    (2)  edge  (1);
\draw[blueE]    (2) edge (3);
\end{scope}

\end{tikzpicture}

\caption{A set of structures that characterizes $3$-colourable graphs by means of 
forbidden linearly ordered $2$-edge-coloured graphs.
In each case the linear ordering is $v_1\le v_2\le v_3$.}
\label{fig:LOOR2}
\end{center}
\end{figure}

\begin{example}
    Let $\calF$ be the set of linearly ordered oriented graphs depicted in
    Figure~\ref{fig:LOOR2}. From Example~\ref{ex:3colLOOR}, and the coding
    described in the paragraph above, we conclude that a graph $G$ 
    is $3$-colourable if and only if there is a $\{blue,red\}$-edge-colouring $(E_b,E_r)$
    of $G$ together with a linear orderings such that $(V,E_b,E_r,\le)$ 
    if an $\calF$-free linearly ordered $2$-edge-coloured graph.
\end{example}

\noindent\textbf{Circular orderings and orientations.}
To conclude this section we briefly study the pullback of
$\hyperlink{OR}{S_{|\OR}\colon\mathcal{OR\to G}}$ and
$\hyperlink{COR}{sh_C\colon\mathcal{COR \to G}}$.
Denote by $\mathcal{COOR}$ the class of circularly ordered oriented graphs.
Similar as before, it is straightforward to notice that
$\mathcal{COOR \cong COR \times_G OR}$. 
Denote by $\hyperlink{COOR}{u_C\colon\mathcal{COOR\to G}}$ the functor
that maps a circularly ordered oriented graph to its underlying graph,
i.e., the functor that makes the following diagram commute,
where $F$ forgets the circular ordering, and $G$ takes the symmetric closure
of the arc relation.

\begin{center}
\begin{tikzcd}[column sep=1.8cm, row sep=1.1cm]
 \mathcal{COOR} \arrow[d,"F"] \arrow[r, "G"] 
\arrow[dr, dashed, "\hyperlink{COOR}{u_C}"] 
& \COR \arrow[d,"\hyperlink{COR}{sh_C}"] \\
 \OR \arrow[r, "S_{|\OR}"] & \calG
\end{tikzcd}
\end{center}

Tucker~\cite{tuckerBAMS76} proposed the following characterization of circular-arc graphs.
A graph $G$ is a circular-arc graph if and only if there is a circular ordering of
$V(G)$ such that for every edge $xy\in E(G)$ either $xz\in E(G)$ for every
$x\le z\le y$, or $yz\in E(G)$ for every $y\le z \le x$.
This characterization was later translated to the context of expressions by
forbidden circular and linear orderings~\cite{guzmanAMC438}. Now, we see that Tucker's
characterization can also be translate to forbidden circularly ordered oriented graphs.

\begin{example}\label{ex:circular-arc}
    Consider the set $\calF$ of circularly ordered oriented graphs depicted in
    Figure~\ref{fig:COOR-circulararc}. Let $\le$ be a circular ordering of $V(G)$
    satisfying Tucker's characterization.  Orient an edge $xy\in E(G)$ as the arc $(y,x)$
    if $xz\in E(G)$ for every $x\le z\le y$; otherwise chose the arc $(x,y)$.
    In case there is a symmetric pair of arcs, choose any of them and remove
    the remaining one. Let $G'$ be the previously defined orientation of $G$. 
    By the choice of $\le $, we conclude that if $(x,y)\in E(G')$, then
    for every $z$ in the circular interval from $x$ to $y$ there is an edge
    $xz\in E(G)$ so, either $(x,z)\in E(G')$ or $(z,x)\in E(G')$. Hence, 
    it follows that $(G',\le)$ is an $\calF$-free circularly ordered oriented
    graph. Conversely, consider an $\calF$-free circularly ordered oriented graph
    $(G',c(\le))$, with underlying graph $G$. With similar arguments as in the first
    part of this proof, one can notice that $\le$ satisfies the condition of Tucker's
    characterization. Hence, $G$ is a circular-arc graph if and only if there is a circular
    orderings of $V(G)$ together with an orientation that avoids the set $\calF$ depicted in 
    Figure~\ref{fig:COOR-circulararc}.
\end{example}

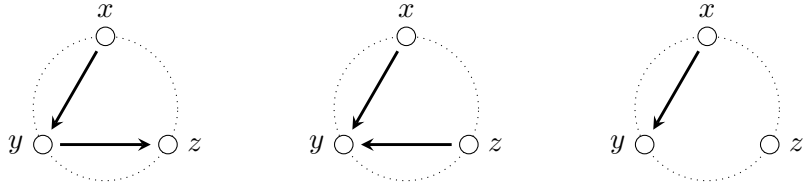
\begin{figure}[ht!]
\begin{center}

\begin{tikzpicture}[scale = 0.8]

\begin{scope}[xshift=-5cm, scale=0.6]
\draw [black, dotted] (0,0) circle[radius = 2];
\node [vertex, label=above:{$x$}] (1) at (90:2){};
\node [vertex, label=right:{$z$}] (2) at (330:2){};
\node [vertex, label=left:{$y$}] (3) at (210:2){};

\draw[arc]    (1)  edge  (3);
\draw[arc]    (3) edge (2);
\end{scope}

\begin{scope}[scale=0.6]
\draw [black, dotted] (0,0) circle[radius = 2];
\node [vertex, label=above:{$x$}] (1) at (90:2){};
\node [vertex, label=right:{$z$}] (2) at (330:2){};
\node [vertex, label=left:{$y$}] (3) at (210:2){};

\draw[arc]    (1)  edge  (3);
\draw[arc]    (2) edge (3);
\end{scope}

\begin{scope}[xshift=5cm, scale=0.6]
\draw [black, dotted] (0,0) circle[radius = 2];
\node [vertex, label=above:{$x$}] (1) at (90:2){};
\node [vertex, label=right:{$z$}] (2) at (330:2){};
\node [vertex, label=left:{$y$}] (3) at (210:2){};

\draw[arc]    (1)  edge  (3);

\end{scope}

\end{tikzpicture}

\caption{A set $\calF$ of circularly ordered oriented graphs, such that a graph $G$ admits a circular
ordering together with an orientation that avoids $\calF$ if and only if $G$ is a circularly ordered
oriented graph. Here, the circular ordering is the circular ordering defined by the clock-wise motion
around the dotted circumference.}
\label{fig:COOR-circulararc}
\end{center}
\end{figure}

Allow us to highlight that Example~\ref{ex:circular-arc} shows that there is a set
of circularly ordered oriented graphs on three vertices that characterizes
circular-arc graphs, while circular-arc graphs are not expressible by forbidden
linear (resp.\ circular) ordered graphs on three vertices~\cite{feuilloleyJDM}
(resp.~\cite{guzmanAMC438}), nor by forbidden oriented graphs on three
vertices~\cite{guzmanAR}.\\

\noindent\textbf{Summary.}
We summarize the relation between the concrete functors we
have introduced so far in the following diagram, where we represent
each of the functors by their domains; a hook
arrow $\mathcal{X\hookrightarrow Y}$
means that the functor $F_\mathcal X\colon \mathcal{X\to G}$ represented by
$\mathcal{X}$ is a local
restriction of the functor $F_\mathcal Y$ represented by $\mathcal  Y$, so
$\ex(F_\mathcal X)\subseteq \ex(F_\mathcal Y)$;
 and an arrow $\mathcal{X\to Y}$ represents that $F_\mathcal X$ factors
through $F_\mathcal Y$ so, $\ex(F_\mathcal Y)\subseteq \ex(F_\mathcal X)$.
Recall that the codomain of each functor considered here is the class of graphs.
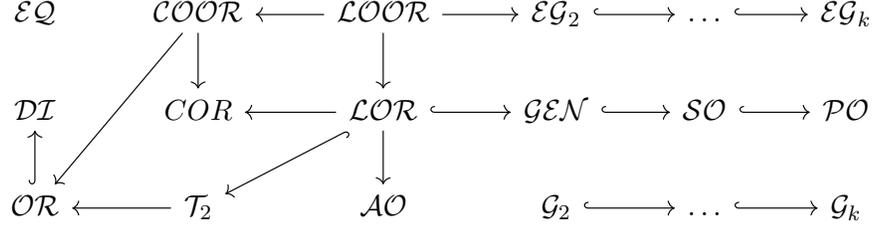
\begin{center}
\begin{tikzcd}
    \hyperlink{EQ}{\mathcal{EQ}} & \hyperlink{COOR}{\mathcal{COOR}} \arrow[ddl]  \arrow[d] &
    \hyperlink{LOOR}{\LOOR} \arrow[d] \arrow[l] \arrow[r] &\hyperlink{EG2}{\mathcal{EG}_2} \arrow[r, hook]
    & \dots \arrow[r, hook] &
   \hyperlink{EG2}{\mathcal{EG}_k} \\
   \hyperlink{DI}{\DI}  & \hyperlink{COR}{COR} &
  \hyperlink{LOR}{\LOR} \arrow[l] \arrow[dl, hook] \arrow[d] \arrow[r, hook] &
  \hyperlink{GEN}{\mathcal{GEN}} \arrow[r, hook] &
  \hyperlink{SO}{\SO} \arrow[r, hook] & \hyperlink{PO}{\PO}  \\
   \hyperlink{OR}{\OR} \arrow[u, hook] & \hyperlink{T2}{\T_2} \arrow[l] &
   \hyperlink{AO}{\AO}  & \hyperlink{Gk}{\calG_2} \arrow[r, hook] & \dots \arrow[r, hook] &
   \hyperlink{Gk}{\calG_k}  \\
\end{tikzcd}
\end{center}

Finally, we stand out that, similarly as in the case of linearly
ordered oriented graphs and circularly ordered oriented graphs, one can
combine the expanded graph structures used to characterize graph classes and
so, enrich the expressive power --- for instance, one can consider
vertex-coloured or arc-coloured oriented graphs. As observed above,
this method corresponds to taking the pullback of the corresponding concrete
functors.


\section{Realizations of concrete functors}
\label{sec:duality}

Several concrete functors considered above arise from expanding
the signature $\{E\}$ to some signature $\sigma$, then
considering some class $\calC$ of $\sigma$-structures, and finally consider
the functor $F\colon \calC\to \calG$ that forgets the relation symbols in 
$\sigma\setminus\{E\}$, i.e., for each $\bbA\in \calC$ its image $F(\bbA)$
is its  $\{E\}$-reduct.
Such reducts and forgetful functors between hereditary classes of relational 
structures are in one-to-one correspondence.  In the following section, we introduce
\textit{quantifier-free definitions} and \textit{quantifier-free reducts}.
The latter generalizes previously considered reducts, and as a corollary of the main
result of Section~\ref{sec:CatEmul} (Theorem~\ref{thm:EmMod}),
we see that concrete functors are in one-to-one correspondence with 
quantifier-free reducts. 

\subsection{Quantifier-free definitions and reducts}
\label{sec:emulationsfunctors}

Consider a (finite) relational signature $\tau$, and denote by
$\QF_\tau$ the language of quantifier-free $\tau$-formulas. It is
straightforward to observe that the formulas in $\QF_\tau$
(up to logical equivalence) together with $(\bot,\top,\lor,\land,\lnot)$
form a \textit{Boolean algebra}.

A \textit{quantifier-free $\sigma$-definition} of a  relation symbol $R\in
\tau$ of arity $k$ is a quantifier-free $\sigma$-formula with $k$-free
variables. A \textit{quantifier-free $\sigma$-definition} of $\tau$ is a
set of quantifier-free $\sigma$-formulas $\{\phi_R\}_{R\in\tau}$ such that
for each $R\in \tau$ the formula $\phi_R$ is a quantifier-free $\sigma$-definition
of $R$. To shorten our writing, we will sometimes write \textit{qf definition}
instead of quantifier-free definition.

 Notice that a quantifier-free $\sigma$-definition $\{\phi_R\}_{R\in\tau}$ of $\tau$
defines a function $\Delta\colon\QF_\tau \to \QF_\sigma$ as follows.
First, define $\Delta(\bot) = \bot$, $\Delta(\top) = \top$, $\Delta(x_1 = x_2)
= x_1 = x_2$, and for each $R\in \tau$ let $\Delta(R(\overline x)) = \phi_R(\overline{x})$.
Then, recursively define $\Delta(\phi \land \psi) = \Delta(\phi)\land \Delta(\psi)$,
$\Delta(\phi \lor \psi) = \Delta(\phi)\lor \Delta(\psi)$, and $\Delta(\lnot\phi) =
\lnot \Delta(\phi)$. Intuitively speaking, $\Delta(\phi)$ is obtained from $\phi$
by replacing the appearance of each relational symbol $R$, $R\in \tau$, by the formula
$\phi_R$, while all variables and symbols $\bot$, $\top$, $\land$, $\lor$,
$\lnot$, and $=$ remain unchanged. With a simple inductive argument we can
notice that $\phi$ and $\Delta(\phi)$ use the same free variables.

It is straightforward to notice that $\Delta\colon \QF_\tau\to \QF_\sigma$
defined above is a Boolean lattice homomorphism
$(\QF_\tau,\bot,\top,\lor,\land,\lnot)\to(\QF_\sigma,\bot,\top,\lor,\land,\lnot)$
such that for each $\phi$ and $\Delta(\phi)$ have the same number of free
variables for each $\phi\in \QF_\tau$. Conversely, every Boolean lattice
homomorphism  $\Delta\colon \QF_\tau\to \QF_\sigma$ that preserves the number
of free variables, defines a quantifier-free $\sigma$-definition $\{\phi_R\}_{R\in\tau}$
of $\tau$, with $\phi_R = \Delta(R)$. For this reason, we denote a quantifier-free
$\sigma$-definition of $\tau$ by $\Delta\colon \QF_\tau\to \QF_\sigma$.

A \textit{self (quantifier-free) definition} is a qf definition
$\Delta\colon \QF_\tau\to \QF_\tau$. For instance, consider the
signature of graphs and digraphs $\tau = \{E\}$ and the 
quantifier-free definition $S\colon \QF_\tau\to QF_\tau$ defined by
$S(E) = E(x,y) \lor E(y,x)$. In this case, if $\phi$ is the
formula $x=y\land \lnot E(x,y)$, then $S(\phi)$ is the formula
$x = y \land \lnot (E(x,y) \lor E(y,x))$. We will use the qf definition
$S$ throughout this section to illustrate forthcoming definitions. 

Given a quantifier-free definition  $\Delta\colon \QF_\tau\to \QF_\sigma$ and 
a $\sigma$-structure $\bbA$, we define the \textit{$\Delta$-reduct}, $sh_\Delta(\bbA)$,
to be the $\tau$-structure with vertex set $V(\bbA)$, and for each $R\in \tau$
of arity $k$ a $k$-tuple $\overline{a}$ of vertices belongs to the interpretation
of $R$ in $sh_\Delta(\bbA)$ if and only if $\bbA\models \Delta(R)(\overline{a})$.
In other words, for every $R\in \tau$ the interpretation $R(sh_\Delta(\bbA))$ is the
relation defined by $\Delta(R)$ in $\bbA$. We also say that $\bbA$ is a
$\Delta$\textit{-expansion} of $sh_\Delta(\bbA)$. We denote by $sh_\Delta\colon \Mod_\sigma\to \Mod_\tau$
the mapping that sends a $\sigma$-structure to its $\Delta$-reduct.
As the following lemma asserts, quantifier-free reducts are embedding preserving
mappings between classes of relational structures, i.e., concrete functors
according to the context of this paper. 

To illustrate $\Delta$-reducts consider again the quantifier-free definition
$S\colon \QF_{\{E\}}\to \QF_{\{E\}}$ previously defined. In this case, the $S$-reduct
of a digraph $D$ is its symmetric closure, equivalently  $sh_S(D)$ is
the underlying graph of $D$.

\begin{proposition}\label{prop:shadowsmodel}
Consider a pair $\sigma, \tau$ of relational signatures, a quantifier-free definition
$\Delta\colon QF_\tau \to \QF_\sigma$, and a $\sigma$-structure $\bbA$. Then, for
each quantifier-free $\tau$-formula $\phi$, and each tuple $\overline a$
of vertices of $\bbA$,
\[
\bbA\models \Delta(\phi)(\overline{a}) \text{ if and only if } sh_\Delta(\bbA)\models \phi(\overline{a}).
\]
\end{proposition}
\begin{proof}
By construction of $sh_\Delta(\bbA)$ the statement holds when $\phi$ is an atomic
formula. The claim follows inductively from the fact that $\Delta$ is a boolean
lattice homomorphism from $(\QF_\tau,\bot,\top,\lor,\land,\lnot)$ to 
$(\QF_\sigma,\bot,\top,\lor,\land,\lnot)$.
\end{proof}

\begin{corollary}\label{cor:welldef}
Consider a pair $\tau,\sigma$ of relational signatures, and a quantifier-free definition
$\Delta\colon \QF_\tau \to \QF_\sigma$. If $\phi,\psi$ is a pair of logically equivalent
quantifier-free $\tau$-formulas, then $\Delta(\phi)$ and $\Delta(\psi)$ are logically equivalent
$\sigma$-formulas.
\end{corollary}
\begin{proof}
Let $\bbA$ be a $\sigma$-structure and $\overline{a}$ a tuple of vertices of $\bbA$. 
By Proposition~\ref{prop:shadowsmodel}, $\bbA\models \Delta(\phi)(\overline{a})$
if and only if $sh_\Delta(\bbA)\models \phi(\overline{a})$. Since $\phi\leftrightarrow \psi$
then $sh_\Delta(\bbA)\models \phi(\overline{a})$ if and only if
$sh_\Delta(\bbA)\models \psi(\overline{a})$. Using Proposition~\ref{prop:shadowsmodel}
again, we conclude that $\bbA\models \Delta(\phi)(\overline{a})$ if and only if
$\bbA\models \Delta(\psi(\overline{a}))$, so $\Delta(\phi)\leftrightarrow \Delta(\psi)$.
\end{proof}

We say that a quantifier-free definition $\Delta\colon \QF_\tau \to \QF_\sigma$ is a
\textit{logically injective} quantifier-free definition if for every pair $\phi,\psi$ of
quantifier-free $\tau$-formulas  with $\Delta(\phi) \leftrightarrow \Delta(\psi)$,
it is the case that  $\phi$ and $\psi$ are also logically equivalent. By similar arguments
to the ones in the proof of Corollary~\ref{cor:welldef}
the following statement holds.

\begin{corollary}\label{cor:sur-inje}
Consider a pair $\tau,\sigma$ of signatures and a quantifier-free definition
$\Delta\colon \QF_\tau \to \QF_\sigma$. If  $sh_\Delta\colon Mod_\sigma\to Mod_\tau$ is a surjective
function, then $M$ is a logically injective quantifier-free definition.
\end{corollary}

Let us go back to our running example $S\colon \QF_{\{E\}}\to \QF_{\{E\}}$ and consider
the formula $\phi(x,y) = E(x,y)\land\lnot E(y,x)$. So $\phi$
is true of an ordered pair of vertices $(u,v)$ in a digraph $D$ if and only 
there is a non-symmetric arc in $D$ from $u$ to $v$.
In particular, for any digraph $D$ with a non-symmetric arc there is a pair
$u,v\in V(D)$ such that $D\models \phi(u,v)$. On the contrary, notice that
$S(\phi)$ is the formula
\[
(E(x,y) \lor E(y,x))\land\lnot (E(x,y) \lor E(y,x))
\]
which is clearly a contradiction. Thus, $S(\phi)$ is logically equivalent to 
$\bot$. Since $\bot = S(\bot)$, and $\phi$ and $\bot$ are not logically equivalent
formulas,   $S$ is not a logically injective qf-definition. 

It is straightforward to observe that for each qf definition $\Delta\colon\QF_\tau\to \QF_\sigma$,
the class of all $\Delta$-reducts of $\sigma$-structures, $sh_\Delta[\Mod_\sigma]$, 
is closed under isomorphisms. We use this claim to prove that the converse statement
of Corollary~\ref{cor:sur-inje} also holds. First, recall that given a $\tau$-structure
$\bbA$ we denote by $\chi_\tau(\bbA)$ the characteristic $\tau$-formula of $\bbA$, and this formula
has the property that  there is an embedding $\varphi\colon \bbA\to \bbB$ if and only if there
is a tuple $\overline{b}$ of vertices of $B$ such that $B\models \chi_\tau(\bbA)(\overline{b})$
(Lemma~\ref{lem:characteristic}).

\begin{lemma}\label{lem:image-characteristic}
Let $\tau$ and $\sigma$ be a pair of relational signatures, $\Delta\colon \QF_\tau \to \QF_\sigma$
a quantifier-free definition, and $\bbA$ a $\tau$-structure. Then, $\bbA\in sh_\Delta[Mod_\sigma]$ if and only
if $\Delta(\chi_\tau(\bbA))$ is not a contradiction.
\end{lemma}
\begin{proof}
Write  $\chi$ to denote $\chi_\tau(\bbA)$. Suppose that $\bbA\in sh_\Delta[Mod_\sigma]$,
where $V(A) = \{a_1,\dots, a_n\}$, and let $\bbA'$ be a $\Delta$-expansion of $\bbA$.
Without loss of generality $\bbA\models \chi(a_1,\dots, a_n)$. So, by
Proposition~\ref{prop:shadowsmodel}, $\bbA'\models \Delta(\chi)(a_1,\dots, a_n)$ and thus,
$\Delta(\chi)$ is not a contradiction. Conversely, suppose that $\Delta(\chi)$ is not a
contradiction and let $n = |V(A)|$. Thus, there is a  $\sigma$-structure $\bbC$, and
an $n$-tuple $\overline{c}$, $\overline{c} = (c_1,\dots, c_n)$, of vertices of $\bbC$ such that
$\bbC\models \Delta(\chi)(\overline{c})$.
Since $\Delta(\chi)$ is a  quantifier-free formula, if $\bbB$ is the substructure of $\bbC$
induced by the vertices in  $\overline{c}$, then $\bbB\models \Delta(\chi)(\overline{c})$. 
So, by Proposition~\ref{prop:shadowsmodel}, $sh_\Delta(\bbB)\models \chi(\overline{c})$
and thus, by Lemma~\ref{lem:characteristic}, the mapping $a_i\mapsto c_i$
defines an embedding $\varphi\colon \bbA\to sh_\Delta(\bbB)$.
We know that $V(sh_\Delta(\bbB)) = V(\bbB) = \{c_1,\dots, c_n\}$ so
$\varphi\colon A\to sh_\Delta(B)$ is an isomorphism, and thus  $\bbA\in sh_\Delta[Mod_\sigma]$.
\end{proof}

Lemma~\ref{lem:image-characteristic} together with Corollary~\ref{cor:sur-inje}
implie the following assertion.

\begin{proposition}\label{prop:insurduality}
Let $\tau$ and $\sigma$ be a pair of relational signatures. For a quantifier-free
definition $\Delta\colon \QF_\tau \to \QF_\sigma$ the following statements are equivalent.
\begin{itemize}
	\item $\Delta$ is a logically injective quantifier-free definition.
	\item For every quantifier-free $\tau$-formula $\phi$, its image $\Delta(\phi)$ is a contradiction
	if and only if $\phi$ is a contradiction.
	\item For every $\tau$-structure $\bbA$ the image $\Delta(\chi_\tau(\bbA))$ of its characteristic
	$\tau$-formula is not a contradiction.
	\item The class of $\Delta$-reducts of $\sigma$-structures is the class of all $\tau$-structures, i.e.,
    $sh_\Delta[\Mod_\sigma] =  \Mod_\tau$.
\end{itemize}
\end{proposition}

We say that a pair of quantifier-free definitions $\Gamma,\Delta \colon \QF_\tau\to \QF_\sigma$ are
\textit{logically equivalent} if for every $R\in \tau$
the formulas $\Gamma(R)$ and $\Delta(R)$ are logically equivalent. In this case we write
$\Gamma\sim \Delta$. The following observation can be easily proved via Proposition~\ref{prop:shadowsmodel}, 
and with similar arguments as the ones used so far.

\begin{observation}\label{obs:logicaleq}
Let $\tau,\sigma$ be a pair of relational signatures, and $\Gamma\colon \QF_\tau \to \QF_\sigma$ and
$\Delta\colon \QF_\tau \to \QF_\sigma$ a pair of quantifier-free definitions.
Then, $\Gamma$ and $\Delta$ are logically equivalent if and only if for each $\sigma$-structure
$\bbA$ the equality $sh_\Gamma(\bbA) = sh_\Delta(\bbA)$ holds.
\end{observation}

Given a qf definition $\Delta\colon \QF_\tau \to \QF_\sigma$ and a
class $\calC$ of $\sigma$-structures, we denote by
$sh_\Delta[\calC]$ the direct image of $\calC$ and call
$sh_\Delta[\calC]$ the \textit{$\Delta$-reduct} of
$\calC$. We also say that $\calC$ is a \textit{$\Delta$-expansion}
of $sh_\Delta(\calC)$.
Finally, if $\calC$ is a class of $\tau$-structures, we denote by 
$sh_\Delta^{-1}[\mathcal{C}]$ the class of all $\sigma$-structures $\bbX$
such that $sh_\Delta(\bbX) \in \calC$. We call $sh_\Delta^{-1}[\calC]$
the \textit{full $\Delta$-expansion} of $\mathcal{C}$. 
For instance, recall that $\OR$ denotes the class of oriented
graphs and $\mathcal{G}$ the class of graphs. Let $S$ be the previously defined
self qf definition  $E\mapsto E\lor \lnot E$. Notice that $\OR$ is an
$S$-expansion of $\calG$ since $sh_{S}[\OR] = \calG$ while the full
$S$-expansion of $\calG$ is the class of all loopless digraphs.

We conclude this subsection with a series of simple examples of quantifier-free
definitions and their corresponding redutcs.

\begin{example}\label{ex:sym-definition}
Our running example $S$ naturally generalizes to arbitrary relational
signatures. Let $\tau$ be a relational signature. The \emph{symmetric} 
quantifier-free definition of $\tau$ is the function $S_\tau\colon \QF_\tau\to
\QF_\tau$ where for each $R\in \tau$  of arity $k$, its image $S_\tau(R)$
is the join of all formulas $R(x_{i_1},x_{i_2},\dots,x_{i_n})$ where $(i_1,\dots,i_n)$
ranges over all permutations of $(1,\dots,n)$.
If $\bbA$ is a $\tau$-structure, then for each $R\in\tau$ the interpretation of $R$
in the $S_\tau$-reduct of $\bbA$  is the symmetric closure of $R(\bbA)$.  Thus, a
$\tau$-structure $\bbA$ is the $S_\tau$-reduct of some $\tau$-structure
if and only if $\bbA$ is a hypergraph.
\end{example}

The following are two simple examples correspond to standard operations in
graph theory. Namely, taking complement and removing loops in graphs. 

\begin{example}\label{ex:simplification-definition}
It is not hard to notice that for every positive integer $n$, there is a
quantifier-free formula with $n$ free variables that expresses that 
$a_1,\dots, a_n$ are pairwise distinct elements.  Let us denote by dif$_n$ any such
formula. The \textit{simplification} quantifier-free definition (of $\tau$) $SP_\tau\colon \QF_\tau\to
\QF_\tau \to \QF_\tau$
is the function that maps $R$ to $R(x_1,\dots,x_n) \land$ dif$_n(x_1,\dots, x_n)$.
Thus, for any digraph $D$, its $SP_{\{E\}}$-reduct is the digraph obtained from $D$ after 
removing loops.
\end{example} 

\begin{example}\label{ex:complement-definition}
The \textit{complement} quantifier-free definition (of $\tau$)
$CO_\tau\colon \QF_\tau\to \QF_\tau$ maps a relation symbol $R$ to $\lnot R$.
So the $CO_{\{E\}}$-reduct of a loopless graph $G$ is the graph obtained from
$\overline G$ after adding a loop on each vertex. 
\end{example}

The following is a more interesting well-known example of quantifier-free reducts.

\begin{example}\label{ex:two-graphs}
    A \emph{two-graph} is a $3$-uniform hypergraph such that every four vertices induce
    an even number of $3$-edges. Now, consider the quantifier-free $\{E\}$-formula
    $\phi$ with three free variable $x,y,z$ that states that $x$, $y$, and $z$ are
    different vertices, and the number of edges induced by $\{x,y,z\}$ is odd.
    Let $H$ be a ternary relation symbol, and $\Delta\colon \QF_{\{H\}}\to \QF_{\{E\}}$ be
    the quantifier-free definition that maps $H$ to $\phi$. It is well-known that
    the class of two-graphs corresponds to the $\Delta$-reduct of the class of graphs
    (see, e.g., \cite{cameronDM127}).
\end{example}


\subsection{Quantifier-free reducts and hereditary classes}

As the following lemma asserts, quantifier-free reducts are embedding preserving
mappings between classes of relational structures, i.e., concrete functors
according to the context of this paper. 

\begin{lemma}\label{lem:preservemb}
Let $\tau$ and $\sigma$ be a pair of signatures, and $\Delta\colon \QF_\tau \to \QF_\sigma$
a quantifier-free definition.
Consider any pair $\bbA$ and $\bbB$ of $\sigma$-structures, and a function
$\varphi\colon V(\bbA)\to V(\bbB)$. If
$\varphi\colon \bbA\to \bbB$ is an embedding, then $\varphi\colon sh_\Delta(\bbA)\to sh_\Delta(\bbB)$
is an embedding too. 
\end{lemma}
\begin{proof}
Let $\chi$ be the characteristic formula of $sh_\Delta(\bbA)$ in $\tau$, and without loss
of generality, suppose that $sh_\Delta(\bbA)\models \chi(a_1,\dots a_n)$ where
$\{a_1,\dots, a_n\}$ is the vertex set of $sh_\Delta(\bbA)$ and of $\bbA$.
Then, Proposition~\ref{prop:shadowsmodel}
implies that $\bbA\models \Delta(\chi)(a_1,\dots, a_n)$. Since $\Delta(\chi)\in \QF_\sigma$ 
and $\varphi\colon \bbA\to \bbB$ is an embedding (of $\sigma$-structures), it follows from
Lemma~\ref{lem:characteristic} that $\bbB\models \Delta(\chi)(\varphi(a_1),\dots, \varphi(a_n))$.
Again,  Proposition~\ref{prop:shadowsmodel} implies that $sh_\Delta(\bbB)\models
\chi(\varphi(a_1),\dots, \varphi(a_n))$. Finally, since $\chi$ is the characteristic formula
of $sh_\Delta(\bbA)$ in $\tau$, we conclude that $\varphi\colon sh_\Delta(\bbA)\to sh_\Delta(\bbB)$
is an embedding (again, via Lemma~\ref{lem:characteristic}).
\end{proof}

Lemma~\ref{lem:preservemb} shows that $\Delta$-reducts preserve embeddings,
the following lemma states that $\Delta$-expansions lift embeddings.

\begin{lemma}\label{lem:closedemulations}
Let $\tau$ and $\sigma$ by a pair of relational signatures, and let $\Delta\colon \QF_\tau \to \QF_\sigma$
be a quantifier-free definition. Consider a pair of $\tau$-structures $\bbA$ and $\bbB$ and
a $\sigma$-structure $\bbB'$ such that $sh_\Delta(\bbB') = \bbB$. If there is an embedding
$\varphi\colon \bbA\to \bbB$ then there is a $\sigma$-structure $\bbA'$ such that 
$sh_\Delta(\bbA') = \bbA$ and  $\varphi\colon \bbA'\to \bbB'$ is an embedding.
\end{lemma}
\begin{proof}
This statement has a straightforward proof using a similar technique to the proof of  
Lemma~\ref{lem:preservemb}. We include a brief
proof for the interested reader. We first suppose without loss of generality
that $i\colon \bbA\to \bbB$ is an inclusion embedding, i.e., $V(\bbA)\subseteq V(\bbB)$. 
Now we consider the characteristic formula  $\chi$ of $\bbA$ in $\tau$
and let $V(\bbA) = \{a_1,\dots, a_n\}$ and $\overline{a} = (a_1,\dots, a_n)$.
Then, we see that $\bbB\models \chi(\overline{a})$ (by Lemma~\ref{lem:characteristic}),
so $\bbB'\models \Delta(\chi)(\overline{a})$ (by Proposition~\ref{prop:shadowsmodel}).
Which implies that $\bbA'\models \Delta(\chi)(\overline{a})$  where $\bbA' = \bbB[V(\bbA)]$
(by Lemma~\ref{lem:characteristic}).
Then, $sh_\Delta(\bbA') \models \chi(\overline{a})$ (by Proposition~\ref{prop:shadowsmodel}),
so the identity function  $Id_\bbA\colon V(\bbA)\to V(\bbA)$ defines an isomorphism
$Id_\bbA\colon \bbA\to sh_\Delta(\bbA')$ (via Lemma~\ref{lem:characteristic}),  and thus
$\bbA = sh_\Delta(\bbA')$. The claim follows.
\end{proof}

Recall that a hereditary class is a class of structures closed under
induced substructures, equivalently, preserved under inverse-embeddings. Thus,
Lemmas~\ref{lem:preservemb} and~\ref{lem:closedemulations} imply the following  statement.

\begin{proposition}\label{prop:images}
Consider a pair $\tau$ and $\sigma$ of relational signatures and a quantifier-free definition
$\Delta\colon \QF_\tau \to \QF_\sigma$. Then, for every hereditary class
$\calC$ of $\sigma$-structures its $\Delta$-reduct $sh_\Delta[\calC]$
is a hereditary class of $\tau$-structures.
\end{proposition}

Furthermore, the same Lemmas~\ref{lem:preservemb} and \ref{lem:closedemulations}
imply the following statement.

\begin{proposition}\label{prop:preimages}
Consider a pair $\tau$ and $\sigma$ of relational signatures, and let 
$\Delta\colon \QF_\tau \to \QF_\sigma$ be a quantifier-free definition. If 
$\calC \subseteq sh_\Delta[\Mod_\sigma]$
is a hereditary class, then the full $\Delta$-expansion of $\calC$ is a 
hereditary class. Moreover, if $\calF$ is the set of minimal obstructions of
$\calC$ in  $sh_\Delta[\Mod_\sigma]$, then the full $\Delta$-expansion of $\calF$
is the set of $\sigma$-minimal obstructions of the full $\Delta$-expansion of $\calC$.
\end{proposition}

\begin{corollary}\label{cor:preslocal}
Consider a pair $\tau$ and $\sigma$ of relational signatures, a quantifier-free definition
$\Delta\colon \QF_\tau \to \QF_\sigma$, and a class $\calC \subseteq sh_\Delta[\Mod_\sigma]$.
Then, $\calC$ is a local class relative to $sh_\Delta[\Mod_\sigma]$ 
if and only if  its full $\Delta$-expansion, $sh_\Delta^{-1}[\calC]$, is a local class.
\end{corollary}

Corollary~\ref{cor:preslocal} asserts that certain relative localness of
classes of $\tau$-structures relates to (absolute) localness of their full $\Delta$-expansions.
On the contrary, it is not hard to find examples of quantifier-free definitions
$\Delta\colon \QF_\tau\to \QF_\sigma$ together with a  local class $\calC$ of
$\sigma$-structures whose $\Delta$-reduct is not a local class --- actually, such examples are
the motivation for studying local expressions of graph classes. Moreover, one can choose $\calC$
to be the class of all $\sigma$-structures.

\begin{example}
    Let $\tau$ be the signatures of graphs $\{E\}$, and $\sigma = \{E,U\}$ where $U$
    is a unary symbol. Consider the quantifier-free definition
    $\Delta\colon \QF_\tau\to \QF_\sigma$ defined by mapping $E$ to 
    \[
    (E(x,y) \lor E(y,x)) \land U(x)\land \lnot U(y) \bigvee (E(x,y) \lor E(y,x)) \land \lnot U(x)\land U(y).
    \]
    One can think of $\sigma$-structures as $2$-vertex coloured digraphs, and for
    each $\sigma$-structure $\mathbb D$ its $\Delta$-reduct is obtained by taking 
    the symmetric closure of the edge relation, then removing all edges incident in
    monochromatic vertices, and then forgetting the vertex colouring. Thus, the
    $\Delta$-reduct of the class of $\sigma$-structures is contained in the class
    of bipartite graphs. Actually, it is not hard to notice that $sh_\Delta[\Mod_\sigma]$
    is the class of bipartite graphs. Thus, the $\Delta$-reduct of the class of
    all $\sigma$-structures is not a local class.
\end{example}


\subsection{Category of quantifier-free definitions}
\label{sec:CatEmul}

Proposition~\ref{prop:insurduality} asserts that a quantifier-free definition
$\Delta\colon \QF_\tau \to \QF_\sigma$ is logically injective if and only
if $sh_\Delta\colon \Mod_\sigma\to \Mod_\tau$ is a surjective mapping. This
suggests  certain duality between quantifier-free definitions and quantifier-free reducts.
In this subsection we explore this duality in terms of category theory.

Consider three relational signatures $\tau$, $\sigma$, and $\pi$, and a
pair of quantifier-free definitions $\Gamma\colon \QF_\tau \to \QF_\sigma$ and $\Delta\colon
\QF_\sigma \to \QF_\pi$.  Recall that for any quantifier-free $\tau$-formula
$\phi$ with $n$ free variables, its image $\Gamma(\phi)$ is a quantifier-free
$\sigma$-formula  with $n$ free variables, and in turn, the $\pi$-formula
$\Delta(\Gamma(\phi))$ also has $n$ free variables. 
Thus, $\Delta \circ \Gamma\colon \QF_\tau\to \QF_\pi$ is a quantifier-free
$\pi$-definition of $\tau$, i.e., quantifier-free definitions compose.

To provide a simple example of composition of quantifier-free definitions, 
consider the complement qf definition $CO_\tau$
(Example~\ref{ex:complement-definition}). Given a relation symbol $R\in \tau$
we evaluate $CO_\tau \circ CO_\tau(R)$ by the
equalities 
\[
CO_\tau(CO_\tau(R)) = CO_\tau(\lnot R) = \lnot(CO_\tau(R)) = \lnot CO_\tau(R) = \lnot\lnot R.
\]
Thus $CO_\tau \circ CO_\tau(R)$ is logically equivalent to $R$, and so,
$CO_\tau\circ CO_\tau\sim Id_\tau$. 

As one would expect, composing quantifier-free definitions commutes with
composing quantifier-free reducts, strengthening the evidence of a
duality between qf definitions and qf reducts suggested by
Proposition~\ref{prop:insurduality}.

\begin{proposition}\label{prop:composition}
Consider three signatures $\tau$, $\sigma$, $\pi$, and a pair of quantifier-free 
definitions $\Gamma\colon \QF_\tau \to \QF_\sigma$ and $\Delta\colon \QF_\sigma\to \QF_\pi$.
For each $\pi$-structure $\bbA$ the following equality holds
\[
sh_{\Delta\circ \Gamma}(\bbA) = sh_\Gamma\circ sh_\Delta(\bbA).
\]
\end{proposition}
\begin{proof}
For any $\pi$-structure $\bbA$, the vertex sets of $sh_{\Delta\circ \Gamma}(\bbA)$
and of $sh_\Gamma(sh_\Delta(\bbA))$ are the same. Now,  consider a relation symbol 
$R\in \tau$. By definition of $sh_{\Delta\circ \Gamma}(\bbA)$, a tuple $\overline{a}$
of vertices  of $\bbA$ belongs to the interpretation $R(sh_{\Delta\circ \Gamma}(\bbA))$
if and only if  $\bbA\models \Delta\circ \Gamma(R)$. On the other hand,
$\overline{a}$ belongs to the interpretation of $R$  in $sh_\Gamma (sh_\Delta(\bbA))$ if
and only if $sh_\Delta(\bbA)\models \Gamma(R)(\overline{a})$. 
Since $\Gamma(R)$ is a quantifier-free formula of $\sigma$, we apply 
Proposition~\ref{prop:shadowsmodel} to $\Delta$,
and conclude that $sh_\Delta(\bbA)\models \Gamma(R)(\overline{a})$ if and only if
$\bbA\models \Delta(\Gamma(R))(\overline{a})$. Therefore, the interpretations of
$R$ in $sh_{\Delta\circ \Gamma}(\bbA)$ and in $sh_\Gamma(sh_\Delta(\bbA))$ are the
same for every relation symbol $R\in \tau$. Thus, $sh_{\Delta\circ \Gamma}(\bbA) =
sh_\Gamma(sh_\Delta(\bbA))$.
\end{proof} 

It is not hard to notice that the composition of quantifier-free definitions is associative,
and that for any quantifier-free definition $\Delta\colon \QF_\tau \to \QF_\sigma$
the equalities $Id_\tau \circ \Delta = \Delta$ and $\Delta \circ Id_\sigma = \Delta$ hold. Thus,
we define the \textit{category of quantifier-free definitions} whose objects  are  finite relational
signatures, and given a pair $\tau$ and $\sigma$ of relational signatures the morphisms
from $\tau$ to $\sigma$ are logically equivalence classes $[\Delta]_\sim$ of quantifier-free
definitions $\Delta\colon \tau \to \QF_\sigma$. We represent an equivalence class by any
qf definition in it, and we denote the category of quantifier-free definitions by
$\mathbf{QF}$.

Lemma~\ref{lem:preservemb}, shows that given a quantifier-free definition
$\Delta\colon \QF_\tau \to \QF_\sigma$, the mapping
$sh_\Delta\colon \Mod_\sigma\to \Mod_\tau$ defined by $\Delta$-reducts commutes
with embeddings. Given a quantifier-free definition  $\Delta\colon \QF_\tau \to \QF_\sigma$,
we will abuse notation and consider $sh_\Delta\colon \Mod_\sigma\to \Mod_\tau$
as the concrete functor defined by $sh_\Delta\colon \Mod_\sigma\to \Mod_\tau$ on objects
and $sh_\Delta(\varphi) = \varphi$ for each embedding of $\sigma$-structures.
For instance, if $\Delta\colon \QF_\tau \to \QF_\sigma$ is an inclusion then
$sh_\Delta$  is the forgetful functor $sh_\tau \colon \Mod_\sigma \to \Mod_\tau$. 

Since a pair of concrete functors $F\colon \mathcal{C}\to \mathcal{D}$
and $G\colon \mathcal{C}\to \mathcal{D}$ satisfy
that $F(\varphi) = \varphi = G(\varphi)$ for every embedding $\varphi$, 
then $F = G$ if and only if for every $C\in\mathcal{C}$ the equality 
$F(C) = G(C)$ holds. Thus, the following corollary is essentially a restatement
of Observation~\ref{obs:logicaleq}.

\begin{corollary}\label{cor:logequivemul}
Consider a pair $\tau, \sigma$ of relational signatures, and a pair
of quantifier-free definitions $\Gamma,\Delta\colon \QF_\tau \to \QF_\sigma$.
Then, $\Gamma \sim \Delta$ if and only if the functors
$sh_\Gamma\colon\Mod_\sigma \to \Mod_\tau$
and $sh_\Delta\colon\Mod_\sigma \to \Mod_\tau$ are equal.
\end{corollary}

Denote by $\mathbf{Mod}$ the category whose objects are categories of relational
structures $\Mod_\tau$ (whose morphisms are embeddings) where $\tau$ ranges over
all finite relational signatures, and the morphisms of $\mathbf{Mod}$ are concrete
functors between these categories. We define a contravariant functor
$\mathbf{sh\colon QF\to Mod}$ as follows.
For every relational signature $\tau$ its image $\mathbf{sh}(\tau)$ is the category
$\Mod_\tau$, and for every $\sim$-class $[\Delta]_\sim$ of quantifier-free
$\sigma$-definitions of $\tau$, we define $\mathbf{sh}([\Delta]_\sim)$ to be 
$sh_\Delta\colon \Mod_\sigma \to \Mod_\tau$.
The fact that $\mathbf{sh}([\Delta]_\sim)$ is well defined on $\sim$-classes follows
from Corollary~\ref{cor:logequivemul}, and the  fact that it is a contravariant functor
follows from Proposition~\ref{prop:composition}. Furthermore, 
Corollary~\ref{cor:logequivemul} also shows that $\mathbf{sh\colon QF\to Mod}$
is a faithful contravariant functor, i.e., it is injective on morphisms. We put these
brief observations together in the following statement.

\begin{proposition}\label{prop:contravariantfb}
The functor $\mathbf{sh\colon QF\to Mod}$ is a faithful contravariant 
functor bijective on objects.
\end{proposition}

The main result of this section asserts that $\mathbf{sh\colon Em\to Mod}$
is a also full contravariant functor, equivalently, for every concrete functor
$F\colon \Mod_\sigma\to \Mod_\tau$ there is a quantifier-free $\sigma$-definition
of $\tau$, $\Delta\colon \QF_\tau \to \QF_\sigma$, such that $sh_\Delta(\bbA) = F(\bbA)$ for
every $\sigma$-structure $\bbA$.  
We begin by studying a pair of universal properties in $\mathbf{QF}$  and $\mathbf{Mod}$.\\

\noindent \textbf{Initial and terminal objects.}
Let $\tau$ be a relational signature, and recall that $\Omega$ denotes the empty
signature.  It is not hard to notice that the unique quantifier-free $\tau$-definition of
$\Omega$ is the function $\varnothing_\tau\colon \QF_\Omega \to \QF_\tau$ that maps a
quantifier-free $\Omega$-formula $\phi$ (i.e., a conjunction and disjunction of equalities)
to $\phi$. We call the $\varnothing_\tau$ the \textit{empty quantifier-free}  $\tau$-definition.
This shows that the initial object in $\mathbf{QF}$ is the empty signature
$\Omega$. Dually, the category $Set$ of sets with injective functions
is the terminal object of $\mathbf{Mod}$. Indeed, for every relational signature
$\tau$ the unique concrete functor from $\Mod_\tau$ to $Set$ is
the forgetful functor. It is clear that $\mathbf{sh}(\Omega) = Set$
and $\mathbf{sh}([\varnothing_\tau]_\sim)$ is the forgetful functor
$sh_\Omega\colon \Mod_\tau\to Set$ for any relational signature $\tau$.\\

\noindent \textbf{Direct sum in $\mathbf{QF}$.}
It is not hard to notice that the category of quantifier-free definitions has direct sums.
Indeed, given a pair  of relational signatures $\tau,\sigma$ consider their
disjoint union $\tau\sqcup \sigma$, and for $\theta\in \{\tau,\sigma\}$ consider
the quantifier-free definition $i_\theta\colon \theta\hookrightarrow \tau\sqcup \sigma$
defined by the inclusion, i.e., $i_\theta(R) = R$ for each  $R\in \theta$. To see
that $\tau \sqcup \sigma$ together with $i_\tau$ and $i_\sigma$
is the direct sum of $\tau$ and $\sigma$, consider a relational signature
$\pi$, and a pair of qf definitions $\Gamma\colon \QF_\tau \to \QF_\pi$
and $\Delta\colon \QF_\sigma \to \QF_\pi$. We denote by  $\Gamma\oplus \Delta$
the Boolean algebra homomorphism defined by $R\mapsto \Gamma(R)$ if $R\in \tau$
and $R\mapsto \Delta(R)$ if $R\in \sigma$. It is straightforward to verify that
the following diagram commutes. 

\begin{center}
\begin{tikzcd}[column sep=1.8cm, row sep=1.1cm]
  & \QF_\pi  &  \\
  \QF_\tau \arrow[hookrightarrow, r, "i_\tau", swap] \arrow[ru, "\Gamma"] & \QF_{\tau \sqcup \sigma}\arrow[u, %
 "\Gamma \oplus \Delta", swap, dashed] &
 \QF_\sigma \arrow[hookrightarrow, l, "i_\sigma"] \arrow[lu, "\Delta", swap]
\end{tikzcd}
\end{center}

In order to show that $\tau \sqcup \sigma$ together with $i_\tau$ and $i_\sigma$ is
the direct sum of $\tau$ and $\sigma$, we still need to prove that $\Delta\oplus \Gamma$ is
unique up to logical equivalence. The reader may convince themself and skip the rest
of this paragraph. Suppose that $\Pi \colon \QF_{\tau \sqcup \sigma} \to \QF_\pi$ is a quantifier-free
definition such that $\Pi \circ i_\tau \sim \Gamma$ and $\Pi \circ i_\sigma
\sim \Delta$.  Since $i_\theta(R) = R\in \tau\sqcup \sigma$ for every $\theta\in \{\tau,\sigma\}$
and every relation symbol $R\in \theta$, it follows that $\Pi \circ i_\tau\sim \Gamma$
and $\Pi \circ i_\sigma \sim \Delta$ if and only if  $\Pi(R) \sim \Gamma(R)$ for
every $R\in \tau$ and $\Theta(R)\sim \Delta(R)$ for every $R\in \sigma$. By definition
of $\Gamma \oplus \Delta$, this implies that  $\Pi \sim \Gamma \oplus \Delta$.\\

\noindent \textbf{Direct product in $\mathbf{Mod}$.}
The dual universal property of the direct sum is the direct product. Similar as above, 
it is not hard to notice that the direct product of a pair of categories 
$\Mod_\tau$ and $\Mod_\sigma$ in the category $\mathbf{Mod}$, is the category
$\Mod_{\tau\sqcup \sigma}$ together with the forgetful functors $sh_\tau\colon
\Mod_{\tau\sqcup \sigma} \to \Mod_\tau$ and $sh_\sigma\colon \Mod_{\tau\sqcup \sigma}
\to \Mod_\sigma$.  To observe this simple assertion consider category $\Mod_\pi$, and a
pair of concrete functors $F\colon \Mod_\pi \to \Mod_\tau$ and $G\colon \Mod_\pi
\to \Mod_\sigma$. Now, we describe the unique functor $F\times G\colon \Mod_\pi\to
\Mod_{\tau\sqcup \sigma}$ that makes the following diagram commute. 

\begin{center}
\begin{tikzcd}[column sep=1.8cm, row sep=1.1cm]
  & \Mod_\pi \arrow[dl, "F", swap] \arrow[d, "F\times G", dashed]
 \arrow[dr, "G"] &  \\
 \Mod_\tau & \Mod_{\tau\sqcup \sigma} \arrow[l, "sh_\tau"] \arrow[r, "sh_\sigma", swap]
 & \Mod_\sigma
\end{tikzcd}
\end{center}

Let $\bbA$ be a $\pi$-structure.  The vertex set of $(F\times G)(\bbA)$ is
$V(\bbA)$ and  for every $R\in \tau \sqcup \sigma$ the interpretation of $R$ in
$(F\times G)(\bbA)$ is the interpretation  $R(F(\bbA))$ if $R\in \tau$, and
if $R\in \sigma$ the interpretation  of $R$ in $(F\times G)(\bbA)$ is $R(G(\bbA))$.
From this construction, it is evident that $F\times G$ is the unique functor that makes
the previous diagram commute. Furthermore, if there is a pair of quantifier-free definitions
$\Gamma\colon \QF_\tau\to \QF_\pi$ and $\Delta\colon \QF_\sigma \to \QF_\pi$ such that
$F = sh_\Gamma$ and $G = sh_\Delta$, then $F\times G = sh_{\Gamma \oplus \Delta}$. We state this
fact in the following lemma without proof since it is a direct implication of the definition
of $F\times G$ and $\Gamma \oplus \Delta$.

\begin{lemma}\label{lem:product}
Consider a pair of concrete functors $F\colon \Mod_\pi \to \Mod_\tau$
and $G\colon \Mod_\pi \to \Mod_\sigma$. If there is a pair of quantifier-free
definitions $\Gamma\colon \QF_\tau\to \QF_\pi$ and $\Delta\colon \QF_\sigma \to \QF_\pi$
such that $F = sh_\Gamma$ and $G = sh_\Delta$, then $F\times G = sh_{\Gamma \oplus \Delta}$.
\end{lemma}

Also, the definition of direct product in $\mathbf{Mod}$ naturally implies
that every concrete functor $F \colon \Mod_\pi \to \Mod_{\tau \sqcup \sigma}$
can be decomposed as the product $(sh_\tau \circ F)\times (sh_\sigma \circ F)$.
With a simple inductive argument, we conclude the following statement. 

\begin{proposition}\label{prop:decomposition}
Let $\sigma$ and $\tau$ be a pair of relational signatures, and $F\colon \Mod_\sigma \to
\Mod_\tau$ be a concrete functor. If $\tau \neq \varnothing$ then
\[
 F = \prod_{R\in \tau} sh_R \circ F
\]
where $sh_R$ denotes the forgetful functor from $\Mod_\tau$ to $\Mod_{\{R\}}$.
\end{proposition}

In simple words, Proposition~\ref{prop:decomposition} asserts that concrete
functors between categories of relational structures can be decomposed as
the product of functors whose codomains are categories of relational structures
with only one relation symbol.\\

\noindent \textbf{Main result.} Finally, we prove the main result of this
subsection: every concrete functor can be represented by a quantifier-free
definition. The proof will follow from a simple technical lemma, and then
using Proposition~\ref{prop:decomposition}. Before stating and proving this
lemma  we introduce some notation. First, it is not hard
to notice that for any integer $n$, there is a formula with $n$ free variables
which is logically equivalent to $\bot$, i.e., a contradiction. We denote
by $\bot_n$ any such formula. Given a relational signature $\tau$ and a
$\tau$-structure $\bbA$,  we say that a tuple  $\overline{a}$ of $V(\bbA)$
is a \textit{spanning tuple} if $\bbA$ is all vertices of $\bbA$ appear in 
$\overline a$. Furthermore, if $R\in \tau$, we say that $\overline a$
is a \textit{spanning $R$-tuple} if $\overline{a}\in R(A)$ (and $\overline a$
is a spanning tuple). In particular, if $\bbA$ has a spanning $R$-tuple then
the number of vertices of $\bbA$ is bounded by the arity of $R$.

Finally, recall that $\chi_\tau(\bbA)$ denotes the characteristic $\tau$-formula
of $\bbA$. Now, it will be convenient to consider a slight variation of these
formulas. For a $\tau$-structure $\bbA$ and a spanning tuple $\overline a$ of 
$\bbA$ we denote by $\chi_\tau(\bbA,\overline a)$ the conjunction of all
atomic and negated atomic formulas which are true of $\overline a$ in $\bbA$. 
In particular, if $\overline a$ is a spanning $|V(\bbA)|$-tuple of $\bbA$, 
then $\chi_\tau(\bbA,\overline a)$ and $\chi_\tau(\bbA)$ are logical equivalent
formulas. 

\begin{lemma}\label{lem:singularemul}
Consider a pair $\tau,\sigma$ of relational signatures, and a concrete functor
$F\colon \Mod_\sigma \to \Mod_\tau$. If $|\tau| \le 1$, then there is 
a quantifier-free definition $\Delta\colon \QF_\tau \to \QF_\sigma$  such that
$F = sh_\Delta$.
\end{lemma}
\begin{proof}
If $|\tau| = 0$, then $\tau$ is the empty signature $\Omega$, and we already
discussed this case when considering initial and terminal objects in $\mathbf{QF}$
and $\mathbf{Mod}$. Now, let $R$ be the unique relation symbol in $\tau$, and let
$r$ be the arity of $R$.
In order to construct a quantifier-free definition $\Delta\colon \QF_\tau \to \QF_\sigma$
we only need to define a quantifier-free $\sigma$-formula $\Delta(R)$. We first consider
the trivial case when every  $\bbA\in F[\Mod_\sigma]$ has an empty interpretation of
$R$.  In this case, simply let $\Delta(R)$ be $\bot_r$, and clearly $F = sh_\Delta$.
Otherwise, notice that there is at least one $\tau$-structure $\bbA\in$ $ F[\Mod_\sigma]$ with
an $R$-spanning tuple. Indeed, we are assuming that there is at least one $\tau$-structure
$\bbA \in F[\Mod_\sigma]$, and $F[\Mod_\sigma]$ is a hereditary class (Lemma~\ref{lem:images}),
thus there is such a $\tau$-structure $\bbA$. For each positive integer $n$, let $SP_n$
be the set of $\tau$-structures on $n$ vertices in $F[\Mod_\tau]$ that have a spanning
$R$-tuple, and let $N$ be the set of integer such that $SP_n \neq \varnothing$. Notice that for every
 $n > r$ the set $SP_n$ is the empty set, so $N$ is a finite non-empty set.
For every $n\in N$ let $\mathcal{S}_n$ be $F^{-1}[SP_n]$. For a $\sigma$-structure $\bbB\in \mathcal S_n$,
we denote by $\phi_\bbB$ the disjunction of all formulas $\chi_\sigma(\bbB,\overline b)$
where $\overline b$ ranges over all spanning $R$-tuples of $F(\bbB)$. It is not hard to notice
that these formulas are well-defined: on the one hand, every spanning $R$-tuple $\overline b$ of
$F(\bbB)$ is a spanning tuple of $F(\bbB)$, and by Proposition~\ref{prop:concretefunct} we know that
$V(\bbB) = V(F(\bbB))$, so $\overline b$ is a spanning tuple of $\bbB$, so each disjunt of $\phi_\bbB$
is well-defined; on the other one,  there are finitely many spanning $R$-tuples of $F(\bbB)$, 
so $\phi_\bbB$ is a finite disjunction.
Finally, we define $\Delta(R)$ as follows.
\[
\Delta(R) = \bigvee_{n\in N} \bigvee_{\bbB\in\mathcal{S}_n} \phi_\bbB.
\]
The formula $\Delta(R)$ is a quantifier-free $\sigma$-formula
since each formula $\phi_\bbB$ is a quantifier-free $\sigma$-formula,
and because $N$ is a finite set, and there are only finitely many non-isomorphic
$\sigma$-structures in $\mathcal{S}_n$ for each $n\in N$.

Now, we want to show that for every $\sigma$-structure $\bbC$, a tuple
$\overline{c}$ of vertices of $\bbC$ belongs to the interpretation $R(F(\bbC))$ if
and only  if $\bbC\models \Delta(R)(\overline{c})$. First suppose
that $\overline{c}$ belongs to $R(F(\bbC))$, and let $\bbB$
be the substructure of $\bbC$ induced by the vertices in $\overline{c}$.
Since $F$ is a concrete functor, it preserves embeddings so $F(\bbB)$ is the
substructure of $F(\bbC)$ induced by the vertices in $\overline{c}$. Thus, 
$\overline{c}$ is a spanning $R$-tuple of $F(\bbB)$ and so, $\bbB\in \mathcal{S}_n$
and clearly $\bbB\models\chi_\sigma(\bbB,\overline c)(\overline{c})$. By definition
of $\Delta(R)$, the latter implies that $\bbB\models \Delta(R)(\overline{c})$. Since
$\Delta(R)$ is a quantifier-free formula, and $\bbB$ is the substructure of $\bbC$
induced by the vertices in $\overline{c}$, it follows by Lemma~\ref{lem:characteristic}
that $\bbC\models \Delta(R)(\overline{c})$.

Conversely, suppose that $\bbC\models \Delta(R)(\overline{c})$ for some tuple
$\overline{c}$ of vertices of $\bbC$. By definition of $\Delta(R)$ there
is an integer $n$, a $\sigma$-structure $\bbB\in \mathcal S_n$, and a spanning
$R$-tuple $\overline b$ of $F(\bbB)$ such that
$\bbC\models \chi_\sigma(\bbB,\overline b)(\overline{c})$. 
Since $\overline{b}$ is a spanning tuple of $B$, it follows via Lemma~\ref{lem:characteristic}
that the mapping $b_i\mapsto c_i$ defines an embedding $\varphi\colon \bbB\to \bbC$.
Since $F$ is a concrete functor, the same function defines an embedding
$\varphi\colon F(\bbB)\to F(\bbC)$. Finally, since $\overline{b}$ is a spanning
$R$-tuple of $R(F(\bbB))$, then $\overline{b}\in R(\bbB)$ and thus $\overline{c}\in R(F(\bbC))$.
This concludes the proof.
\end{proof}

It is straightforward to observe that the following statement holds via a simple
inductive argument using the decomposition stated in Proposition~\ref{prop:decomposition},
together with Lemmas~\ref{lem:product} and~\ref{lem:singularemul}.

\begin{proposition}\label{prop:duality}
Consider a pair $\tau$ and $\sigma$ of relational signatures. If
$F\colon \Mod_\sigma\to \Mod_\tau$ is a concrete functor, then
there is a quantifier-free definition $\Delta\colon \QF_\tau \to
\QF_\sigma$ such that $F = sh_\Delta$.
\end{proposition}

Proposition~\ref{prop:contravariantfb} asserts that $\mathbf{sh\colon QF\to Mod}$
is a faithful contravariant functor bijective on objects. Proposition~\ref{prop:duality}
implies that $\mathbf{sh\colon QF\to Mod}$ is a full functor, i.e., surjective on morphisms. 
Recall that given a category $\mathcal D$ we denote by $\mathcal D^{opp}$ the dual
category, i.e., the category obtained from  $\mathcal D$ by reversing all morphisms. 

\begin{theorem}\label{thm:EmMod}
The categories $\mathbf{QF}$ and $\mathbf{Mod}^{opp}$  are isomorphic categories.
\end{theorem}
\begin{proof}
    Immediate implication of Propositions~\ref{prop:contravariantfb} and~\ref{prop:duality}.
\end{proof}

An application of these results as that concrete functors defined on the class
of all $\sigma$-structures are computable in
polynomial-time. This is, for every concrete concrete functor $F\colon \Mod_\sigma\to \Mod_\tau$
there is a polynomial-time algorithm that, given an input $\sigma$-structure $\bbA$,
computes $F(\bbA)$ in polynomial-time (with respect to $|V(\bbA)|$).

\begin{corollary}\label{cor:duality-pol-computation}
Let $\tau$ and $\sigma$ be a pair of relational signatures. Then, every
concrete functor
$F\colon \Mod_\sigma\to \Mod_\tau$ is computable in polynomial-time.
\end{corollary}
\begin{proof}
It follows from Proposition~\ref{prop:duality} that there is a qf $\tau$-definition
$\Delta\colon\QF_\tau\to \QF_\sigma$ such that $F(\bbA) = sh_\Delta(\bbA)$
for each $\sigma$-structure $\bbA$. Thus, one can efficiently construct $F(\bbA)$ 
by first defining $V(F(\bbA)) = V(\bbA)$ (see Proposition~\ref{prop:concretefunct}),
and the for each $R\in \tau$ using the quantifier-free definition $\Delta(R)$ to 
construct the interpretation of $R$ in $F(\bbA)$. It is straight forward to verify that
this construction can be done in $O(|V(\bbA)|)^r)$ time, where $r$ is the largest arity
of a relation symbol in $\tau$. 
\end{proof}

\subsection{Extensions of concrete functors}

In this brief section, we study extensions of concrete functors $F\colon \calC\to \calD$
defined on a proper subclass of $\sigma$-structures, to concrete functors
$F'\colon \Mod_\sigma\to \Mod_\tau$ defined on the class of all $\sigma$-structures.
In particular, we are interested in extending
Proposition~\ref{prop:duality} to concrete functors $F\colon\mathcal{C\to D}$,
where $\calC$ might be a proper subclass of all $\sigma$-structures. And thus, show
that every concrete functor $F\colon\mathcal{C\to D}$ is computable in polynomial-time. 

Consider a pair $\tau$ and $\sigma$ of relational signatures, and let 
$F\colon\calC\to \calD$ be a concrete functor where $\calC$ and $\calD$
are hereditary classes of $\sigma$-structures and $\tau$-structures,
respectively. We define the \textit{weak extension} $F^\ast\colon\Mod_\sigma
\to \Mod_\tau$ as follows.  Naturally, for every $\bbC\in\calC$, the image
$F^\ast(\bbC)$ equals $F(\bbC)$. Consider a structure
$\bbA\in \Mod_\sigma\setminus\mathcal{C}$.  We define  $F^\ast(\bbA)$ by
considering the induced substructures $\bbC < \bbA$ such that $\bbC\in\mathcal{C}$, 
and for each of these, construct a copy of $F(\bbC)$ in $F^\ast(\bbA)$.
Formally, we define $F^\ast(\bbA)$ as follows.
\begin{itemize}
    \item The vertex set  of $F^\ast(\bbA)$ is the vertex set $V(\bbA)$ of $\bbA$.
    \item For each $R\in \tau$ of arity $r$, there is an  $r$-tuple $\overline a$ of
    $V(F^\ast(\bbA))$ in $R(F^\ast(\bbA))$ if and only if there is a structure
    $\bbC\in \mathcal C$, an $r$-tuple $\overline c$ of $V(\bbC)$, and an embedding
    $\varphi\colon\bbC\to \bbA$ such that: $\varphi(\overline c) = \overline a$,
    and $\overline c\in R(F(\bbC))$.
\end{itemize}
It directly follows from the construction of $F^\ast$ that $F^\ast\colon\Mod_\sigma\to
\Mod_\tau$ preserves embeddings, i.e., $F^\ast$ is a concrete functor. The following
observation can be also proved immediately from the definition of $F^\ast$
and the definition of concrete functors. 

\begin{observation}\label{obs:weakmin}
Consider a pair of hereditary classes $\calC$ and $\calD$ of 
$\sigma$-structures and $\tau$-structures, respectively, 
and let $F\colon\mathcal{C\to D}$ be a concrete functor. 
If $F'\colon\Mod_\sigma \to \Mod_\tau$ is an extension
of $F$ then, for every $\sigma$-structure $\bbA$, and each relation
symbol $R\in \tau$, the inclusion $R(F^\ast(\bbA))\subseteq R(F'(\bbA))$
holds.
\end{observation}

This observation motivated the name ``weak extension''. We provide
a simple example to illustrate that concrete functors might have 
different extensions. 

\begin{example}\label{ex:non-unique-ext}
    Recall that $\DI$ is the class of irreflexive digraphs, and  let $\tau$ be its signature,
    i.e., $\tau = \{E\}$.  Consider the symmetric functor $\hyperlink{DI}{S\colon \mathcal{DI
    \to G}}$ that maps an irreflexive digraph to its symmetric closure. Then,  the weak extension
    $S^\ast\colon\Mod_\tau \to\Mod_\tau$ maps a digraph $D$ (with possible loops) to the
    underlying simple graph of $D$, i.e., $S^\ast$ takes the symmetric closure of $D$ and it
    erases all loops.  Now, consider the functor $S'\colon\Mod_\tau\to \Mod_\tau$ that
    preserves loops and takes the symmetric closure of the non-loop arcs. Clearly, $S'$ also
    extends $\hyperlink{DI}{S}$ and the arc set of $S^\ast(D)$ is a subset of the arc set
    of $S'(D)$ for each digraph $D$ (as Observation~\ref{obs:weakmin} asserts).
\end{example}

In Example~\ref{ex:non-unique-ext}, we were able to construct different extensions
of $S$ because $S$ was not defined on certain ``small'' structures. In general, it
is not hard to create different extensions of a concrete functor $F\colon \calC\to \calD$
when $F$ is not defined on ``small'' structures. In the following proposition, we observe
that only in such cases $F$ has different extensions to $F\colon\Mod_\sigma\to \Mod_\tau$.

\begin{proposition}\label{prop:uniqueext}
Consider a pair  of hereditary classes  $\calC,~\calD$ of  $\sigma$-structures
and $\tau$-structures, respectively, and let $F\colon\mathcal{C\to D}$ be a
concrete functor.  There is a unique extension $F'\colon\Mod_\sigma\to \Mod_\tau$
of $F$ if and only if the arity of every relation symbol in $\tau$ is strictly
larger than the vertex of each  minimal obstruction of $\calC$. In this case,
$F' = F^\ast$. 
\end{proposition}
\begin{proof}
Suppose that there is a minimal obstruction $\mathbb M$ of $\calC$,
and a relation symbol $R\in \tau$ of arity $r$ such that $|V(\mathbb M)|\le r$. 
Consider the following extension $F_1$ defined on $\mathcal{C'}$,
$\mathcal{C'} = \mathcal{C}\cup \{\mathbb M\}$. For every $\bbC\in\mathcal{C}$,
$F_1(\bbC) = F(\bbC)$. The vertex set of $F_1(\mathbb M)$ is $V(\mathbb M)$ and
the interpretation of each symbol $S\in \tau\setminus\{R\}$ in $F_1(\mathbb M)$
is the same as the interpretation in the weak extension $F^\ast(\mathbb M)$.
Finally, the interpretation of $R$ in $F_1(\mathbb M)$ is the union of
$R(F^\ast(\mathbb M))$ together with all spanning $r$-tuples of 
$V(\mathbb M)$ (recall that a spanning $r$-tuple of $V(\mathbb M)$ is an $r$-tuple
that contains all vertices of $V(\mathbb M)$). Since we only added spanning
tuples to the interpretation of $R$ in $F^\ast(\mathbb M)$, and $\mathbb M$ is a minimal
obstruction of $\mathcal{C}$, then $F_1\colon\mathcal{C'\to D}$ is a
concrete functor, and clearly $F^\ast(\mathbb M)\neq F_1(\mathbb M)$. 
Finally, the weak extension $F_1^\ast$ of $F_1$ is an extension of $F$
that does not equal the weak extension $F^\ast$ of $F$.

To prove the converse implication, suppose that $F'\colon\Mod_\sigma\to
\Mod_\tau$ is an extension of $F$, and let $\bbA$ be a $\sigma$-structure. 
We show that $F^\ast(\bbA) = F'(\bbA)$. Clearly, if $\bbA\in\mathcal{C}$, then
$F'(\bbA) = F(\bbA) = F^\ast(\bbA)$. Suppose that $\bbA\in \Mod_\sigma\setminus
\calC$ and let $R$ be a relation
symbol  of $\tau$ of arity $r$. By Observation~\ref{obs:weakmin}, we know that
$R(F^\ast(\bbA))\subseteq R(F'(\bbA))$, we show that the inverse inclusion
holds. Let $\overline{a}$ be an $r$-tuple in $R(F'(\bbA))$. Since
the $r$ is strictly less that the size of the vertex set any minimal
obstruction of $\mathcal{C}$, then the substructure $\bbB$ of $\bbA$ induced
by the vertices in $\overline{a}$ belongs to $\mathcal{C}$. 
So, by construction of $F^\ast(\bbA)$, the vertices in $\overline{a}$
induce a copy of $F(\bbC)$ in $F^\ast(\bbA)$. Finally, since $F'$ preserves
embeddings, it follows that $\overline{a}\in R(F(\bbC))$, and thus
$\overline{a}\in R(F^\ast(\bbA))$.
\end{proof}

\begin{corollary}
Let $\tau$ and $\sigma$ be a pair of relational signatures, $m$ be the maximum arity
of a relation symbol in $\tau$, and consider a pair of concrete functors
$F,G\colon\Mod_\sigma \to\Mod_\tau$.
If $F(\bbC) = G(\bbC)$ for every $\sigma$-structure $\bbC$ on at most $m$ vertices, then
$F = G$.
\end{corollary}
\begin{proof}
Consider the class $\mathcal{C}$ of $\sigma$-structures on at most $m$ vertices. 
Let $F_m\colon\calC \to \Mod_\tau$ and $G_m\colon\calC\to \Mod_\tau$
be the restrictions of $F$ and $G$ to $\mathcal{C}$. So, 
$F_m = G_m$, and since the minimal obstructions of $\mathcal{C}$
have $m+1$ vertices, by Proposition~\ref{prop:uniqueext} we conclude
that $F = F_m^\ast = G_m^\ast = G$.
\end{proof}

It would be interesting to characterize the image of the weak extension
$F^\ast$ of a concrete functor $F\colon\mathcal{C\to D}$ in terms of the image
 $F[\mathcal{C}]$. In particular, it could be interesting to characterize
the concrete functors $F\colon\mathcal{C\to D}$ such that
$F[\mathcal{C}] = F^\ast[\mathcal{C}]$.
Nonetheless, the main purpose of this section
is to extend Proposition~\ref{prop:duality} to concrete functors
between hereditary classes of relational structures. We do so in the following
theorem.

\begin{theorem}\label{thm:concrete-emul}
Consider a pair of relational signatures $\tau,\sigma$, and let $\calC$
be a hereditary class of $\sigma$-structures. If $\calD$ is a class of
$\tau$-structures and $F\colon\mathcal{C\to D}$ is a concrete functor,
then there is a quantifier-free definition $\Delta\colon \QF_\tau \to
\QF_\sigma$ such that $F(\bbC) = sh_\Delta(\bbC)$ for every $\bbC\in\calC$.
\end{theorem}
\begin{proof}
By Proposition~\ref{prop:duality}, there is an quantifier-free definition
$\Delta\colon \QF_\tau \to \QF_\sigma$ such that $F^\ast = sh_\Delta$.
Since $F^\ast$ is an extension of $F$, then $F(\bbC) = sh_\Delta(\bbC)$
for every $\bbC\in\mathcal{\calC}$.
\end{proof}

The following application of  Theorem~\ref{thm:concrete-emul} can be proved
with similar arguments as Corollary~\ref{cor:duality-pol-computation}.

\begin{corollary}\label{cor:concrete-poly}
Consider a pair of relational signatures $\tau,\sigma$, and let $\calC$ and
$\calD$ be hereditary classes of $\sigma$ and $\tau$ structures, respectively. 
 Then, every concrete functor $F\colon\mathcal{C\to D}$ can be computed
 in polynomial-time.
\end{corollary}

 
\section{Pseudo-local classes}
\label{sec:pseudo-local}

Recall that a local class of graphs is a hereditary class with a finite
set of minimal obstructions, up to isomorphism. In general, a local class
of $\sigma$-structures is a hereditary class of $\sigma$-structures with
finitely many minimal obstructions, equivalently, an $\forall_1$-definable
class. In this second-to-last section we introduce the notion of pseudo-local
(graph) classes, we study some elementary properties of such classes, and
notice that, similar to the relation between local classes an $\forall_1$-definable
classes, pseudo-local classes correspond to definable classes
in the logic SNP. 

Let $\tau$ be a relational signature, and $\calD$ a class of $\tau$-structures. 
We say that $\calD$ is a  \textit{pseudo-local class} if there is a local
class $\calC$ of $\sigma$-structures (sore some finite relational signature $\sigma$),
and a surjective concrete functor $F\colon\calC\to \calD$. That is,
pseudo-local classes are the images of local classes under concrete functors. 
It is straightforward to observe that every local class is a pseudo-local class, 
and that there are pseudo-local classes which are not local (see, e.g.,
Example~\ref{ex:non-unique-ext}).

In Section~\ref{sec:local-expresions}, we introduced the expressive power of
a given concrete functor $F$. We briefly relate pseudo-local classes with expressive
power of concrete functors. Clearly, if $\calD$ is a pseudo-local class, then
$\calD$ is locally expressible by a concrete functor with local domain. Indeed, 
let $F\colon\calC\to \calD$ be a surjective functor where $\calC$ is a local class. 
Then, $\calD\in \ex(F)$ and in particular, this implies that $\calD$ is locally expressible
by some concrete functor with local domain. Conversely, suppose that $\calD$ is locally
expressible by some functor $F\colon \calC\to \calD$ with local domain, and 
let $\calD'\subseteq_l \calC$ be a local class relative to $\calC$ such that
$F[\calD'] = \calD$. In this case, $\calD'$ is a local class
(Observation~\ref{obs:localtransitive}) and the restriction $F'\colon \calD'\to \calD$
is a surjective functor with local domain. Hence, $\calD$ is a pseudo-local class 
if and only if there is a  functor $F$ with local domain such that $\mathcal{D}$ is locally
expressible by $F$. The following statement shows that we can actually substitute
``local domain'' by ``pseudo-local domain'' in the previous sentence. 

\begin{proposition}\label{prop:pseudoloceq}
For every class $\mathcal{D}$ of relational structures (with finite signature) the following
statements are equivalent.
\begin{itemize}
	\item $\mathcal{D}$ is a pseudo-local class.
	\item There is a surjective concrete functor
	$F\colon \mathcal{C}\to \mathcal{D}$ for some pseudo-local class $\calC$.
	\item There is a  functor $F$ with pseudo-local domain such that $\calD\in \ex(F)$.
\end{itemize}
\end{proposition}
\begin{proof}
Since every local class is a pseudo-local class, the first item implies the second
one.  With similar arguments as in the paragraph above, one should be able to
notice that the second item implies the third one.  To see that the third one implies
the the first one, let $F\colon\mathcal{C\to D'}$ be a pseudo-local functor such
that $\mathcal{D}\in \ex(F)$. Since $\calC$ is a pseudo-local class, there
is a surjective functor $G\colon\mathcal{A\to C}$ with $\calA$ a local class. Clearly, 
$F\circ G$ factors through $F$ so, by the first part of Proposition~\ref{prop:factor-restriction},
$\ex(F)\subseteq \ex(F\circ G)$ and thus, $\mathcal{D}\in \ex(F\circ G)$. We showed in
the paragraph preceding this proposition that if $\calD\in \ex(F)$ for some functor
$F$ with local domain, then $\calD$ is a pseudo-local class. The claim now follows.
\end{proof}

Observation~\ref{obs:localtransitive} asserts that if $\mathcal{C}$
is a local class and $\mathcal{D}$ is a local class relative to $\mathcal{C}$, 
then $\mathcal{D}$ is a local class. This also holds for pseudo-local classes. 

\begin{corollary}\label{cor:plocaltransitive}
Let $\calC$ and $\calD$ be a pair of hereditary classes
such that $\mathcal{D}\subseteq_l\mathcal{C}$. If $\calC$
is a pseudo-local class, then $\calD$ is a pseudo-local class.
\end{corollary}
\begin{proof}
By the equivalence between the first and third items of
Proposition~\ref{prop:pseudoloceq}, if $\calC$ is a pseudo-local class,
then there is a pseudo-local functor  $F$ such that $\calD\in
\ex(F)$. So, if $\calD\subseteq_l\calC$ then, by
Proposition~\ref{prop:basicexpressiveproperties}, $\calD\in \ex(F)$,
and thus $\calD$ is a pseudo-local class.
\end{proof}

Notice that each concrete functor considered in Section~\ref{sec:examples}
has a pseudo-local domain --- actually, all except 
$\hyperlink{AO}{S_{|\AO}}\colon\AO\to\calG$ have local domain. Thus,
all examples considered in Section~\ref{sec:examples} provide further examples
of pseudo-local graph classes such as matrix-partition classes,  subclasses of
circular-arc graphs, amongst several others.

Recall that local classes are closed
under finite unions and finite intersections (Corollary~\ref{cor:inter+union}).
From this, it is straightforward to verify that pseudo-local classes are closed
under unions too: if $F\colon \calC_1\to \calD_1$ and $G\colon \calC_2\to \calD_2$
are surjective concrete functors with local domain, then $F\oplus G$ is a functor
with local domain that surjects onto $\calD_1\cup \calD_2$. 
To show that the intersection of pseudo-local classes is a pseudo-local class
we show that the pullback of local classes is a local class.

\begin{lemma}\label{lem:pulllocal}
Consider a pair  $F\colon\mathcal{A\to D}$ and $G\colon\mathcal{B\to D}$ of
concrete functors. If $\mathcal{A}$ and $\mathcal{B}$ are local classes
then $\mathcal{A \times_\mathcal{D} B}$ is a local class.
\end{lemma}
\begin{proof}
To show  that $\mathcal{A\times_\mathcal{D} B}$ is a local class, we will 
use the duality between concrete functors and qf definitions stated in 
Theorem~\ref{thm:concrete-emul}, and the equivalence between local
classes and  $\forall_1$-definable classes stated in Theorem~\ref{thm:equivlocal}.
Let $\pi$ and $\sigma$ be the respective (disjoint) signatures of
the structures in $\calA$ and $\calB$. Recall that the signature
of $\mathcal{A\times_\mathcal{D} B}$ is the disjoint union of $\pi$ and $\sigma$.
Denote by $sh_\pi$ and $sh_\sigma$ the forgetful functors from $\pi\sqcup \sigma$-structures to 
$\pi$-structures and to $\sigma$-structures, respectively. 
Since $\calA$ and $\calB$ are local classes, there is a pair $\phi_A$
and $\phi_B$ of $\forall_1$-formulas that define $\calA$ and $\calB$ in
their corresponding signatures. Thus, it is not hard to notice that
$\phi_A\land \phi_B$ is an $\forall_1$-formula with signature $\pi\sqcup \sigma$
that defines  $sh_\pi^{-1}[\calA ]\cap sh_\sigma^{-1}[\calB]$. Now, recall that
$\mathcal{A\times_\mathcal{D} B}$ is the subclass of structures $\bbC$ in
$sh_\pi^{-1}[\calA]\cap sh_\sigma^{-1}[\calB]$ such that
$G\circ sh_\sigma(\bbC) = F\circ sh_\pi(\bbC)$.  Let $\tau$ be the signature
of $\calD$, and $\Gamma\colon \QF_\pi\to \QF_\tau$
and $\Delta\colon \QF_\sigma\to \QF_\tau$ be the qf definitions
such that $sh_\Gamma(\bbA) = F(\bbA)$ and $sh_\Delta(\bbB) = G(\bbB)$ for
all $\bbA\in\mathcal{A}$ and all $\bbB\in \mathcal{B}$. Thus, it is not hard to
notice that for a structure $\bbC\in sh_\pi^{-1}[\mathcal{A}]\cap
sh_\sigma^{-1}[\calC]$ it is the case that $G\circ sh_\sigma(\bbC) = F\circ sh_\tau(\bbC)$
if and only if $\bbC$ models the formula $\psi$ defined as
\[
\bigwedge_{R\in \tau} \forall c_1,\dots, c_r(\Delta(R)(c_1,\dots,c_r) \iff \Gamma(R)(c_1,\dots, c_r)),
\]
where $r$ is the arity of $R$. Therefore, $\mathcal{A\times_\mathcal{D} B}$ is defined by
the $\forall_1$-formula  $\phi_A\land \phi_B\land \psi$, and the claim
follows.
\end{proof}

It is not hard to notice that the image of the pullback
$P\colon \mathcal{A\times_\calD B\to D}$ is the intersection of the images $F[\mathcal{A}]$
and $G[\mathcal{B}]$. Thus, as a consequence of Lemma~\ref{lem:pulllocal}
we conclude that the intersection of pseudo-local classes is a pseudo-local class. 

\begin{proposition}\label{prop:pllattice}
The class of pseudo-local classes is a lattice under the inclusion order,
where the join and meet are the set theoretic union and intersection. 
\end{proposition}
\begin{proof}
    We already argued that pseudo-local classes are closed under unions,
    and as mentioned above, it follows from Lemma~\ref{lem:pulllocal}
    that pseudo-local classes are closed under intersections. 
\end{proof}

Proposition~\ref{prop:pllattice} does not extend to countable unions
nor to countable intersections. This assertion follows from the observations
that there are countably many pseudo-local (graph) classes, there are uncountably
many hereditary (graph) classes, and every hereditary class can be
recovered by both, a countable union and a countable intersection
of local (graph)  classes.

Local graph classes are fairly easy to understand both intuitively, and
from a logical point of view  (see, e.g., Theorem~\ref{thm:equivlocal}).
Moreover, it is straightforward to notice that local classes belong to the
famous computational class  P, i.e.,  deciding if an input graph $G$ belongs
to a given local class $\calC$ can be decided in polynomial-time: if $\calF$ is
the (finite) set of minimal obstructions of a local class $\calC$, then,
on an input graph $G$ run through all $f$-tuples (where $f$ is the maximum
number of vertices of a graph in $\calF$) and verify if one of these tuples
induces a graph in $\calF$.  Thus, even though the computational class P
is far from being understood, and usually ad-hoc polynomial-time
algorithms are needed for different problems,  the previous algorithms yields a simple
and ``universal'' method for solving the decision problem associated to local classes.

Neither of the previous aspects is as clear when we consider pseudo-local (graph)
classes. In the rest of this section, we study some simple meta-properties
of pseudo-local graph classes that transfer from local classes. To begin with,
we address the computational complexity aspect of pseudo-local classes.
Similar as in the case of local classes, there is a ``universal'' certification
method: let $\calD$ be a pseudo-local class of $\tau$-structures (resp.\ graphs),
and $F\colon\calC\to \calD$ a surjective concrete functor $\calC$ a local
class of $\sigma$-structures; then a yes-certificate for a $\tau$-structure
(resp.\ graph) $\bbD\in \calD$ is a $\sigma$-structure $\bbC\in \calC$ such
that $F(\calC) = \bbD$. The ``universal'' polynomial-time certification algorithm
follows from the polynomial-time recognition of local classes explained above, 
and the polynomial-time computation of concrete functors from
Corollary~\ref{cor:concrete-poly}.

\begin{proposition}\label{prop:pseudoNP}
Every pseudo-local class belongs to NP.
\end{proposition}

Finally, we observe that pseudo-local classes correspond to the 
classes in \textit{SNP} which stands for \textit{strict NP}. Given
a relational signature $\tau$, an SNP sentence is a formula of the
form $\exists R_1,\dots, R_k(\phi)$ where $R_1,\dots, R_k$ are relational
symbols not in $\tau$, and $\phi$ is an $\forall_1$-sentence in the
expanded signature $\tau\cup\{R_1,\dots, R_k\}$. A class $\calC$ is
an SNP-definable class if there is an SNP sentence $\varphi$ such that
a $\tau$-structure $\bbC$ belongs to $\calC$ if and only if $\bbC\models \varphi$.

\begin{theorem}\label{thm:equivpseudolocal}
The following statements are equivalent for a class of relational
structures $\calD$ (with finite relational signature).
\begin{itemize}
	\item $\calD$ is a pseudo-local class.
    \item $\calD$ is a quantifier-free reduct of a local class. 
    \item $\calD$ is a reduct of a local class. 
    \item $\calD$ is an SNP-definable class. 
\end{itemize}
\end{theorem}
\begin{proof}
The equivalence between the first two statements follows from Theorem~\ref{thm:concrete-emul}, 
while the equivalence between the last two statements follows from definition
of SNP. The fact that the third item implies the second one is trivial.
We conclude by showing the the first item implies the third one. 
Let $F\colon\mathcal{C\to D}$ be a surjective concrete with $\calC$ a local
class of $\sigma$-structures. Let $\phi$ be an $\forall_1$-formula that defines
$\calC$ in the signature $\sigma$, and $\Delta\colon \QF_\tau\to \QF_\sigma$
a quantifier-free definition such that $F = sh_\Delta$. Similar as in the proof
of Lemma~\ref{lem:pulllocal}, we see that the pullback 
$P\colon \mathcal{A\times_\calD D\to D}$ of $F$ with the identity $Id_\calD$
has a local domain. Let $\psi$ be the $\sigma\sqcup \tau$-formula
\[
\bigwedge_{R\in \tau} \forall c_1,\dots, c_r(\Delta(R)(c_1,\dots,c_r) \iff R(c_1,\dots, c_r)),
\]
where $r$ is the arity of $R$. It is straightforward to verify that 
$\calA\times_\calD \calD$ is defined by the formula $\phi\land \psi$.
Clearly, $\calD$ is the $\tau$-reduct of $\calA\times_\calD \calD$ and thus,
every pseudo-local class is a reduct of a local class.
\end{proof}


\section{Conclusion and future research}
\label{sec:conclusion}

The main contribution of this work is to present the foundations of local
expressions of graph classes. As we saw above, this framework intends to
study several ways of characterizing graph classes under the same scope. 
In particular, in Section~\ref{sec:examples}, we presented new, simple,
and nice examples of such characterizations arising from the general
framework presented in the beginning of Section~\ref{sec:local-expresions}. 
In Section~\ref{sec:duality}, we built on the notion of quantifier-free
definitions to introduce the category of quantifier-free definitions,
and proved that qf definitions encode all embedding preserving functors
between classes of relational structures. Thus, we showed
that each local expression of a graph class $\mathcal{C}$
yields a polynomial time certificate to the membership in $\mathcal{C}$.
This motivated one of the open  questions in this framework, 
\textit{what can be certified by local expressions?}

\subsection*{Three generic problems}

Local expressions by concrete functors intend to grasp the fundamental
idea of expressions by forbidden orientations and forbidden orderings. 
Several observations about oriented and ordered expressions
generalize to local expressions by concrete functors.
Naturally, some of the driving questions and problems in the context of
forbidden orientations and forbidden orderings extend to the framework of
local expressions. Here, we propose three main generic problems of local
expressions by concrete functors, and briefly survey some open instances
of these.\\

\noindent
\textbf{Characterization Problem.}
\textit{Consider a concrete functor $F\colon\calC\to \calG$ and a finite set
$\calF\subseteq \mathcal{C}$. Characterize the class of graphs $G$
 for which  there is an  $\calF$-free structure $\bbC\in\mathcal{C}$
 such that $F(\bbC) = G$.}\\

In this terms, Skrien~\cite{skrienJGT6} studied the Characterization Problem
of $\hyperlink{OR}{S_{|\mathbf{OR}}}$ and when $\calF$ is a set of oriented paths
on three vertices. Also, Damaschke~\cite{damaschkeTCGT1990} was the first to
consider the characterization problem for $\hyperlink{LOR}{sh_L}$, and recently
Feuilloley and Habib~\cite{feuilloleyJDM} completed the task when $\calF$ is a set
of linearly ordered graphs on three vertices. Similarly, the Characterization 
Problem for genealogical graphs we considered in~\cite{paulPREPRINT} (in terms of
tree-layouts), and in~\cite{bokIP} this problem was considered for $2$-edge-coloured
graphs when $\calF$ is a set of structures on three vertices.  Regarding
forbidden orientations, it is still open to list all minimal obstructions to 
the class of graphs that admit a $\{B_1\}$-free orientation, where $B_1$ is
the orientation of $P_3$ with one vertex of out-degree $2$. As mentioned before, 
this problem was introduced by Skrien~\cite{skrienJGT6}, and studied by Hartinger
and Milanic~\cite{hartingerDM340,hartingerJGT85}. Similarly, it is still open 
to characterize the class of graphs that admit a $\{B_1,\overrightarrow{C_3}\}$-free
orientation~\cite{guzmanAR}. Open instances of this problem regarding forbidden
linearly ordered graphs on four vertices can be found in~\cite{feuilloleyAR}, where the
authors investigate the relation of forbidden linearly ordered graphs and intersection
graph  of certain geometric objects. Finally, the Characterization Problem for remains 
unexplored for most of the functors considered in Section~\ref{sec:examples}. In particular, 
we believe that the Characterization Problem for forbidden $2$-arc-coloured tournaments
on three vertices (Problem~\ref{prob:expressions2arc}) could yield several interesting
graph classes (see, e.g., Example~\ref{ex:2-arc-col} and the paragraph preceding it).\\

\noindent
\textbf{Complexity Problem.}
\textit{Consider a functor $F\colon\mathcal{C\to G}$ with pseudo-local domain $\calC$,
and a finite set $\calF\subseteq \mathcal{C}$. Determine the complexity of deciding 
if for some input graph $G$ there is an  $\calF$-free structure $\calC\in\mathcal{C}$
such that $F(\calC) = G$.}\\

As far as we are aware, Duffus, Ginn, and R\"odl~\cite{duffusRSA7} where the first
authors to  consider a particular instance of this problem. They considered the
Complexity Problem for $\hyperlink{LOR}{sh_L}$ and when $\calF = \{(G,\le)\}$ for
an arbitrary $2$-connected graph $G$. In particular, they showed that for almost
all $2$-connected graphs, the $\{(G,\le)\}$-free orderings problem is NP-complete.
Later,  Hell, Mohar, and Rafiey~\cite{hellESA2014} showed that
the $\calF$-free linear ordering problem is polynomial-time solvable
whenever $\calF$ is a set of linearly ordered graphs on at most three vertices.
Again, some open instances of this problem, regarding forbidden linear
orderings motivated by intersection graphs of geometric objects can be 
found in~\cite{feuilloleyAR}. An a conjecture from~\cite{duffusRSA7}
asserts that the Complexity Problem for $\{(G,\le)\}$ where $G$ is $2$-connected
is NP-complete, unless $G$ is a complete graph.
Similarly, the Complexity Problem was considered in~\cite{guzmanAR} for
forbidden orientations and when $\calF$ is a set of oriented graphs on $3$ vertices. 
Moreover, the Complexity Problem for forbidden orientations was recently settled
when $\calF$ is a set of tournaments~\cite{bodirskyAR}.  Regarding forbidden
orientation on three vertices, it only remains open to determine the
complexity of deciding if an input graph admits a $\{B_1,\overrightarrow{C_3}\}$-free
orientation~\cite{guzmanAR}.\\

\noindent
\textbf{Expressibility Problem.}
\textit{Consider a concrete functor $F\colon\mathcal{C\to G}$ and a class of
graphs $\mathcal{P}$. Determine if $\mathcal{P}$ is locally 
expressible by $F$.}\\

To date, this is the least considered problem in the literature. As far as we are
aware, the first instance of the Expressibility Problem was asked by
Damaschke~\cite{damaschkeTCGT1990} asked if the class of circle graphs is
an $FOSG$-class, i.e., the Expressibility Problem for $\hyperlink{LOR}{sh_L\colon\LOR\to \calG}$
and circle graphs (still open). In the same context, Hell and Hern\'andez-Cruz
asked the following question.

\begin{question}\cite{hellPC}
Is the class of even-hole free graphs expressible by forbidden linear
orderings?
\end{question}

The only known progress in this direction is the work of Guzm\'an-Pro and
Hern\'andez-Cruz~\cite{guzmanEJC28}. Motivated by understanding
the expressive power of forbidden orientations regarding classes with forbidden holes,
these authors asked if the class of odd-hole free graphs is expressible by forbidden
orientations. Other instances of the Expressibility Problem can be found
in~\cite{guzmanEJC105} for forbidden orientations and $H$-colouring classes;
and in~\cite{guzmanAMC438} it was asked whether the class of bipartite
graphs is expressible by forbidden circular orderings.

\subsection*{Local expressions and pseudo-local classes}

The three generic problems intend to encompass several addressed
and open problems that lie within the realm of structural and algorithmic
graph theory. Several questions from Sections~\ref{sec:local-expresions} and
Section~\ref{sec:duality} lie outside the scope of structural and
algorithmic theory. Nonetheless, we believe that these could still yield
interesting research. Several of these address the problem of comparing expressive
powers of different functors. In particular, we
asked if there is a meaningful relation between functors
with the same expressive power (Question~\ref{qst:equivalencefunctors}).
Other questions in the spirit of comparing expressive powers
include Questions~\ref{qst:k-k+1}, \ref{qst:OR-S}, and~\ref{qst:AO-L-S-P}. 
 It would also be interesting to know if by considering functors with larger
domain we always increase the expressive power (we ask it for graphs, but
it can also be asked for general classes of relational structures).

\begin{question}
    Given a concrete functor $F\colon \calC\to \calG$ for some local 
    class of $\tau$-structures $\calC$, is there a finite extension
    $\tau'$ of $\tau$ and an extension $F'\colon\calC'\to \calG$  of $F$
    such that $\ex(F')$ properly extends $\ex(F)$?
\end{question}

\noindent
\textbf{Logic and computational complexity.}
Feder and Vardi~\cite{federJC28} proved that the class of SNP-definable classes of 
relational structures has the full computational power of NP up to polynomial-time
equivalence. Thus, it follows from Theorem~\ref{thm:equivpseudolocal}
that the complexity of every language in NP, is represented by a pseudo-local class,
up to polynomial-time equivalence. This raises the question whether this remains
true if we restrict to pseudo-local graph classes. 

\begin{question}
    Is it true that for every language in NP there is a polynomial-time equivalent
    pseudo-local class of graphs?
\end{question}

This questions addresses the expressive power of pseudo-local graph classes up to
polynomial-time equivalence, but it also makes sense to ask about the literal expressive
power: every pseudo-local (graph) class is a hereditary class in NP, is the converse true?

\begin{question}
    Is there a hereditary class of graphs that belongs to NP but it is not a pseudo-local
    graph class?
\end{question}

\begin{question}
    Is there a hereditary class of graphs that belongs to P but it is not a pseudo-local
    graph class?
\end{question}

Similar questions had been previously considered for general relational structures in the context
of SNP and CSPs  by Bodirsky and Rudolph~\cite{bodirsky-personal} (see also, open problem 31
in~\cite{bodirsky2021}). It is not hard to notice that every finite CSP
is a pseudo-local class.  In particular, for each graph $H$ the class of $H$-colourable graphs is
a pseudo-local class. Moreover, Example~\ref{ex:equivalence-CSP} shows that there is a functor
$sh_\sim\colon\mathcal{EQ}\to \calG$
such that every class of $H$-colourable graphs is locally expressible by $F$. Now, we ask
if forbidden orientations and forbidden linear orderings also express all classes of $H$-colourable graphs.

\begin{question}
    Is every class of $H$-colourable graphs expressible by forbidden orientations?
\end{question}

\begin{question}
    Is every class of $H$-colourable graphs expressible by forbidden linear orderings?
\end{question}

\section*{Acknowledgements}
The author initiated this work during his PhD studies supervised by C\'esar Hern\'andez-Cruz
whose valuable comments, feedback, and several discussions helped improving the present work.
All these are gratefully acknowledged.  The author also thanks Manuel Bodirsky for point out
the logic SNP, reference~\cite{federJC28}, and proposing Example~\ref{ex:equivalence-CSP}.


\appendix 

\section{Glossary of concrete functors}
\label{ap:concrete-functors}

This appendix serves a double purpose: it provides a quick reference for some concrete 
functors mentioned throughout this work, and it summarizes several observations about
their expressive power. We list these functors according to their domain ---
in all cases the codomain is $\mathcal{G}$. 
Recall that we consider categories defined by hereditary classes of relational structures
with emebedings (of relational structures). To shorten our writing, given a positive
integer $k$ we denote by $\mathcal{C}_k$ and by $\mathcal{C}_{\le k}$ the set of structures
with exactly $k$ vertices and with at most $k$ vertices. Non-standard classes
of relational structures are can be found in Appendix~\ref{ap:graph-classes}.
In order to keep the information of each functor together, we use one page per
functor.\newpage

\hypertarget{AO}{
\noindent
\textbf{\large{{Acyclic oriented graphs -- $\AO$}}}}\\
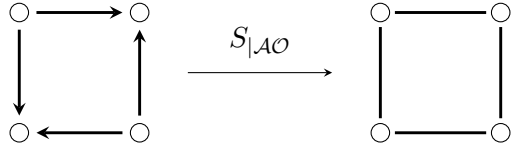
\begin{figure}[ht!]
\begin{center}
\begin{tikzpicture}[scale = 0.8]

\node [vertex] (1) at (-4,0){};
\node [vertex] (2) at (-4,2){};
\node [vertex]  (3) at (-2,2){};
\node [vertex] (4) at (-2,0){};
\foreach \from/\to in {2/3, 4/3, 4/1, 2/1}
\draw[arc]    (\from)  edge  (\to);

\node [vertex] (a) at (2,0){};
\node [vertex] (b) at (2,2){};
\node [vertex]  (c) at (4,2){};
\node [vertex] (d) at (4,0){};

\foreach \from/\to in {a/b, b/c, c/d, d/a}
\draw[edge]    (\from)  to  (\to);

\node[] (D) at (-1.5,1){};
\node[] (C) at (1.5,1){};
\node[] at (0,1.5){$S_{|\AO}$};
\draw[arc, thin]    (D)  edge  (C);

\end{tikzpicture}
\caption[Functor $S_{|\AO} \colon \mathcal{AO\to G}$]{To the left, an
acyclic oriented graph $A$ such that $S_{|\AO}(A) = C_4$.}
\end{center}
\end{figure}
\begin{itemize}
	\item \textit{Functor:}
	Maps an acyclic oriented graph to its
	underlying graph; also defined by self quantifier-free definition
    $E\mapsto E(x,y) \lor E(y,x)$.
	\item \textit{Equivalent expressions:}
	Classes expressible by forbidden acyclic orientations~\cite{guzmanEJC105},
	and $F^\ast$-graph classes~\cite{skrienJGT6}.
	\item \textit{Locally expressible examples:}
	Chordal graphs and forests~\cite{skrienJGT6}; $k$-colourable graphs
    (RGHV-Theorem); homomorphism classes of
    odd cycles~\cite{guzmanEJC28}.
	\item \textit{Not locally expressible examples:}
	Even-hole free graphs~\cite{guzmanEJC105}.
	\item \textit{Relation to other expressive powers:}
	$ex(S_{|\AO}) \subseteq ex(sh_L)$ (Proposition~\ref{prop:factor-restriction}).
	\item \textit{Complexity Problem:}
	Polynomial time solvable for $F\subseteq \AO_{\le3}$
    (follows from~\cite{hellESA2014} via Proposition~\ref{prop:factor-restriction}).
	For every $k\ge 4$ there is a set $F\subseteq \AO_k$
	such that its Complexity Problem is $NP$-complete (follows from RGHV-Theorem).
	\item \textit{Questions:} Do the classes of
	co-bipartite graphs or co-chordal graphs belong to $ex(S_{|\AO})$?
	Is $ex(S_{|\AO}) \subseteq ex(sh_L)$ a proper inclusion?\\
\end{itemize}	
\newpage

\hypertarget{T2}{
\noindent
\textbf{\large{{$\mathbf 2$-arc-coloured tournaments  -- $\T_2$}}}}\\
\begin{figure}[ht!]
\begin{center}
\begin{tikzpicture}[scale = 0.8]

\node [vertex] (1) at (-4,0){};
\node [vertex] (2) at (-4,2){};
\node [vertex]  (3) at (-2,2){};
\node [vertex] (4) at (-2,0){};
\foreach \from/\to in {3/4, 4/1, 1/2, 2/3}
\draw[arc, blue]    (\from)  edge  (\to);

\foreach \from/\to in {1/3, 2/4}
\draw[arc, red, dashed]    (\from)  edge  (\to);

\node [vertex] (a) at (2,0){};
\node [vertex] (b) at (2,2){};
\node [vertex]  (c) at (4,2){};
\node [vertex] (d) at (4,0){};

\foreach \from/\to in {a/b, b/c, c/d, d/a}
\draw[edge]    (\from)  to  (\to);

\node[] (D) at (-1.5,1){};
\node[] (C) at (1.5,1){};
\node[] at (0,1.5){$SB$};
\draw[arc, thin]    (D)  edge  (C);

\end{tikzpicture}
\caption[Functor $SB\colon\mathcal{T_2\to G}$]{To the left, a $2$-arc-coloured
tournament $A$ such that $SB(A) = C_4$.}
\end{center}
\end{figure}
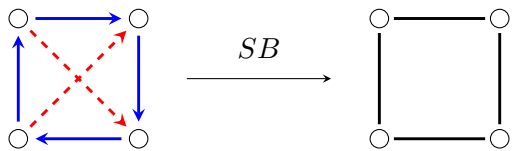
\begin{itemize}
	\item \textit{Functor:}
	Forgets red arcs and takes the symmetric closure of the blue arcs;
	also defined by the quantifier-free definition $E \mapsto E_b(x,y) \lor E_b(y,x)$.
	\item \textit{Equivalent expressions:}
	Local expressions by the pullback of $S_{|\OR}$ and
	$co\circ S_{|\OR}$ where $co(G) = \overline{G}$.
	\item \textit{Locally expressible examples:}
	Proper circular-arc co-bipartite graphs (Example~\ref{ex:2-arc-col})
	\item \textit{Not locally expressible examples:}
	Not considered yet.
	\item \textit{Relation to other expressive powers:}
	$ex(S_{|\OR}) \cup ex(sh_L)\subseteq ex(SB)$
	(Proposition~\ref{prop:factor-restriction}).
	\item \textit{Complexity Problem:}
	Not considered yet.
	\item \textit{Questions:} 
	Which classes are expressible by $2$-arc-coloured
	tournaments on $3$ vertices?
	Complexity Problem when $F$ contains tournaments on three vertices.\\
\end{itemize}
\newpage

\hypertarget{COR}{
\noindent
\textbf{\large{{Circularly ordered graphs -- $\COR$}}}}\\
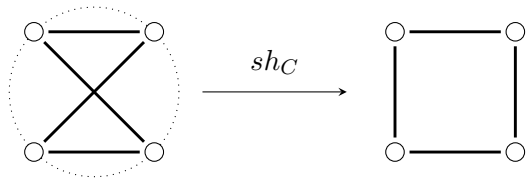
\begin{figure}[ht!]
\begin{center}
\begin{tikzpicture}[scale = 0.8]

\draw [black, dotted] (-3,1) circle[radius = 1.41];
\node [vertex, fill= white] (1) at (-2,0){};
\node [vertex, fill= white] (2) at (-2,2){};
\node [vertex, fill= white]  (3) at (-4,2){};
\node [vertex, fill= white] (4) at (-4,0){};

\foreach \from/\to in {1/3, 3/2, 2/4, 4/1}
\draw[edge]    (\from)  to  (\to);

\node [vertex] (a) at (2,0){};
\node [vertex] (b) at (2,2){};
\node [vertex]  (c) at (4,2){};
\node [vertex] (d) at (4,0){};

\foreach \from/\to in {a/b, b/c, c/d, d/a}
\draw[edge]    (\from)  to  (\to);

\node[] (D) at (-1.5,1){};
\node[] (C) at (1.5,1){};
\node[] at (0,1.5){$sh_C$};
\draw[arc, thin]    (D)  edge  (C);

\end{tikzpicture}
\caption[Functor $sh_C\colon\mathcal{COR\to G}$]{To the left, a circularly ordered
graph $A$ such that $sh_C(A) = C_4$. Vertices are circularly ordered
by the clockwise ordering.}
\end{center}
\end{figure}
\begin{itemize}
	\item \textit{Functor:}
	Forgets the circular ordering;
	also defined by the quantifier-free definition $E \mapsto E(x,y)$.
	\item \textit{Equivalent expressions:}
	Classes expressible by forbidden circular orderings (see, e.g., \cite{guzmanAMC438}).
	\item \textit{Locally expressible examples:}
	Circular-arc graphs, outerplanar graphs, graphs with circular chromatic
	number less than $k$~\cite{guzmanAMC438}.
	\item \textit{Not locally expressible examples:}
	Unknown.
	\item \textit{Relation to other expressive powers:}
	$ex(sh_C)\subseteq ex(sh_L)$~\cite{guzmanAMC438}.
	\item \textit{Complexity Problem:}
	Polynomial time solvable when $F\subseteq \COR_{\le 3}$.
	For every $k\ge 5$ there is a set $F\subseteq \COR_k$
	such that the Complexity Problem is $NP$-complete~\cite{guzmanAMC438}.
	\item \textit{Questions:} 
	Is the class of bipartite graphs locally
	expressible by $sh_C$?  Is $ex(sh_C)\subseteq ex(sh_L)$ a proper
	inclusion?
	What is the complexity of the Complexity Problem when
	$F \subseteq \COR_4$?\\ \newpage
\end{itemize}

\hypertarget{COOR}{
\noindent
\textbf{\large{{Circularly ordered oriented graphs -- $\mathcal{COOR}$}}}}\\
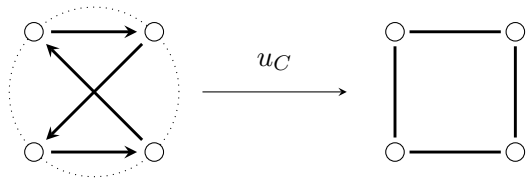
\begin{figure}[ht!]
\begin{center}
\begin{tikzpicture}[scale = 0.8]

\draw [black, dotted] (-3,1) circle[radius = 1.41];
\node [vertex, fill= white] (1) at (-2,0){};
\node [vertex, fill= white] (2) at (-2,2){};
\node [vertex, fill= white]  (3) at (-4,2){};
\node [vertex, fill= white] (4) at (-4,0){};

\foreach \from/\to in {1/3, 3/2, 2/4, 4/1}
\draw[arc]    (\from)  to  (\to);

\node [vertex] (a) at (2,0){};
\node [vertex] (b) at (2,2){};
\node [vertex]  (c) at (4,2){};
\node [vertex] (d) at (4,0){};

\foreach \from/\to in {a/b, b/c, c/d, d/a}
\draw[edge]    (\from)  to  (\to);

\node[] (D) at (-1.5,1){};
\node[] (C) at (1.5,1){};
\node[] at (0,1.5){$u_C$};
\draw[arc, thin]    (D)  edge  (C);

\end{tikzpicture}
\caption[Functor $u_C\colon\mathcal{COOR\to G}$]{To the left, a circularly ordered
oriented graph $A$ such that $u_C(A) = C_4$. Vertices are circularly ordered
by the clockwise ordering.}
\end{center}
\end{figure}
\begin{itemize}
	\item \textit{Functor:}
	Forgets the circular ordering and the orientation of the arcs;
	also defined by quantifier-free definition $E \mapsto E(x,y) \lor E(y,x)$.
	\item \textit{Equivalent expressions:}
	Locally expressible classes by the pullback
	\hfill \newline $P\colon\mathcal{OR \times_G COR\to G}$.
	\item \textit{Locally expressible examples:}
	Circular-arc graphs (Example~\ref{ex:circular-arc}).
	\item \textit{Not locally expressible examples:}
	Not considered yet.
	\item \textit{Relation to other expressive powers:}
	$ex(sh_C\oplus S_{|\OR})\subseteq ex(u_C)$ 
	(Proposition~\ref{cor:pullback-disjoint}).
	\item \textit{Complexity Problem:}
	Not considered yet.
	\item \textit{Questions:} Is $ex(sh_C\oplus S_{|\OR})\subseteq ex(u_C)$ 
	a proper inclusion? Is there a set $F\subseteq \mathcal{COOR}_3$ such
	that its Complexity Problem is NP-complete?\\
\end{itemize}
\newpage

\hypertarget{DI}{
\noindent
\textbf{\large{{Digraphs -- $\DI$}}}}
\begin{figure}[ht!]
\begin{center}
\begin{tikzpicture}[scale = 0.8]

\node [vertex] (1) at (-4,0){};
\node [vertex] (2) at (-4,2){};
\node [vertex]  (3) at (-2,2){};
\node [vertex] (4) at (-2,0){};
\foreach \from/\to in {3/4, 4/1, 1/2}
\draw[arc]    (\from)  edge  (\to);

\draw[arc]    (2)  edge[bend right]  (3);
\draw[arc]    (3)  edge[bend right]  (2);

\node [vertex] (a) at (2,0){};
\node [vertex] (b) at (2,2){};
\node [vertex]  (c) at (4,2){};
\node [vertex] (d) at (4,0){};

\foreach \from/\to in {a/b, b/c, c/d, d/a}
\draw[edge]    (\from)  to  (\to);

\node[] (D) at (-1.5,1){};
\node[] (C) at (1.5,1){};
\node[] at (0,1.5){$S$};
\draw[arc, thin]    (D)  edge  (C);

\end{tikzpicture}
\caption[Functor $S\colon\mathcal{DI\to G}$]{To the left, a digraph $A$ such that
$S(A) = C_4$.}
\end{center}
\end{figure}
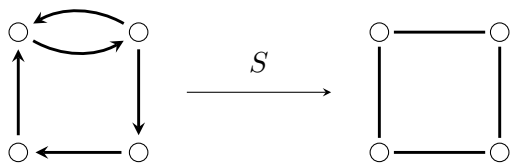
\begin{itemize}
	\item \textit{Functor:}
	Maps a digraph to its  underlying graph; also defined by the self
	quantifier-free definition $E\mapsto E(x,y) \lor E(y,x)$.
	\item \textit{Locally expressible examples:}
	$\mathcal{P}$-mixed graphs when $\mathcal{P}$ is a local class
	(Example~\ref{ex:Pmixed}).
	\item \textit{Not locally expressible examples:} Not considered yet.
	\item \textit{Relation to other expressive powers:}
	$ex(S_{|\OR}) \subseteq ex(S)$ (Proposition~\ref{prop:factor-restriction}).
	\item \textit{Complexity Problem:}
	For every $k\ge 4$ there is a set $F\subseteq \DI_k$
	such that its Complexity Problem is NP-complete (follows from RGHV-Theorem).
	\item \textit{Questions:} 
	Is there a graph property locally expressible by $S$ but not by $S_{|\OR}$?
	Is there a set $F\subseteq\DI_3$ such that its Complexity Problem
	is NP-complete?\\
\end{itemize}
\newpage

\hypertarget{EG2}{
\noindent
\textbf{\large{Edge-coloured graphs  -- $\mathcal{EG}_2$}}}\\
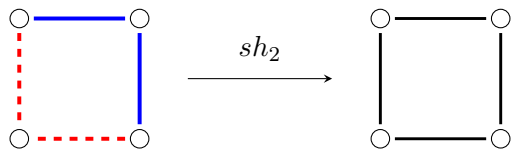
\begin{figure}[ht!]
\begin{center}
\begin{tikzpicture}[scale = 0.8]

\node [vertex] (1) at (-4,0){};
\node [vertex] (2) at (-4,2){};
\node [vertex]  (3) at (-2,2){};
\node [vertex] (4) at (-2,0){};
\foreach \from/\to in {4/1, 2/1}
\draw[redE]    (\from)  edge  (\to);

\foreach \from/\to in {2/3, 4/3}
\draw[blueE]    (\from)  edge  (\to);

\node [vertex] (a) at (2,0){};
\node [vertex] (b) at (2,2){};
\node [vertex]  (c) at (4,2){};
\node [vertex] (d) at (4,0){};

\foreach \from/\to in {a/b, b/c, c/d, d/a}
\draw[edge]    (\from)  to  (\to);

\node[] (D) at (-1.5,1){};
\node[] (C) at (1.5,1){};
\node[] at (0,1.5){$sh_2$};
\draw[arc, thin]    (D)  edge  (C);

\end{tikzpicture}
\caption[Functor $sh_2\colon\mathcal{EG_2\to G}$]{To the left, a
$2$-edge-coloured graph $A$ such that $sh_2(A) = C_4$.}
\end{center}
\end{figure}
\begin{itemize}
	\item \textit{Functor:}
	Forgets the edge colouring; also defined by quantifier-free definition
	$E\mapsto E_r(x,y) \lor E_b(x,y)$.
	\item \textit{Locally expressible examples:}
	Co-bipartite graphs, line-graphs of bipartite graphs, and
    $2$-edge-colourable graphs~\cite{bokIP}.
	\item \textit{Not locally expressible examples:}
	Not considered yet.
	\item \textit{Relation to other expressive powers:}
	Not considered yet.
	\item \textit{Complexity Problem:}
	In P for forbidden colourings of $3K_1$, $K_1 + K_2$, and $P_3$, and
    NP-complete for certain forbidden colourings of $K_3$~\cite{bokIP}.
	\item \textit{Questions:}
	Is there a meaningful relation between expressions by forbidden oriented graphs
    and by forbidden	$2$-edge-coloured graphs?\\
\end{itemize}
\newpage
\hypertarget{EQ}{
\noindent
\textbf{\large{Equivalence graphs -- $\mathcal{EQ}$}}}\\
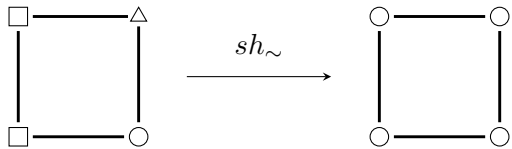
\begin{figure}[ht!]
\begin{center}
\begin{tikzpicture}[scale = 0.8]

\node [vertex, rectangle] (1) at (-4,0){};
\node [vertex, rectangle] (2) at (-4,2){};
\node [vertex, regular polygon, regular polygon sides=3]  (3) at (-2,2){};
\node [vertex] (4) at (-2,0){};
\foreach \from/\to in {2/3, 4/3, 4/1, 2/1}
\draw[edge]    (\from)  edge  (\to);

\node [vertex] (a) at (2,0){};
\node [vertex] (b) at (2,2){};
\node [vertex]  (c) at (4,2){};
\node [vertex] (d) at (4,0){};

\foreach \from/\to in {a/b, b/c, c/d, d/a}
\draw[edge]    (\from)  to  (\to);

\node[] (D) at (-1.5,1){};
\node[] (C) at (1.5,1){};
\node[] at (0,1.5){$sh_\sim$};
\draw[arc, thin]    (D)  edge  (C);

\end{tikzpicture}
\caption[Functor $sh_\sim\colon\mathcal{EQ\to G}$]{To the left, an
equivalence graph $A$ such that $sh_\sim(A) = C_4$.}
\end{center}
\end{figure}
\begin{itemize}
	\item \textit{Functor:}
	Forgets the equivalence relation; also defined by quantifier-free definition
	$E\mapsto E(x,y)$.
	\item \textit{Locally expressible examples:}
	Classes of $H$-colourable graphs (Example~\ref{ex:equivalence-CSP}).
	\item \textit{Not locally expressible examples:}
	Not considered yet.
	\item \textit{Relation to other expressive powers:}
	Not considered yet.
	\item \textit{Complexity Problem:}
	Not considered yet.
	\item \textit{Questions:} Not considered yet.\\
\end{itemize}
\newpage
\hypertarget{GEN}{
\noindent
\textbf{\large{{Genealogical graphs  -- $\mathcal{GEN}$}}}}\\
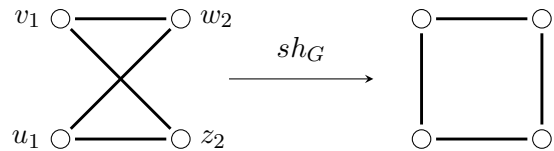
\begin{figure}[ht!]
\begin{center}
\begin{tikzpicture}[scale = 0.8]

\node [vertex, label = right:{$z_2$}] (1) at (-2,0){};
\node [vertex,  label = right:{$w_2$}] (2) at (-2,2){};
\node [vertex,  label = left:{$v_1$}]  (3) at (-4,2){};
\node [vertex,  label = left:{$u_1$}] (4) at (-4,0){};

\foreach \from/\to in {1/3, 3/2, 2/4, 4/1}
\draw[edge]    (\from)  to  (\to);

\node [vertex] (a) at (2,0){};
\node [vertex] (b) at (2,2){};
\node [vertex]  (c) at (4,2){};
\node [vertex] (d) at (4,0){};

\foreach \from/\to in {a/b, b/c, c/d, d/a}
\draw[edge]    (\from)  to  (\to);

\node[] (D) at (-1.5,1){};
\node[] (C) at (1.5,1){};
\node[] at (0,1.5){$sh_G$};
\draw[arc, thin]    (D)  edge  (C);

\end{tikzpicture}
\caption[Functor $sh_G\colon\mathcal{GEN\to G}$]{To the left, a genealogical graph
$A$ such that $sh_G(A) = C_4$. Vertices are ordered
from left to right.}
\end{center}
\end{figure}
\begin{itemize}
	\item \textit{Functor:}
	Forgets the ordering on the vertex set; also defined by quantifier-free definition
	$E \mapsto E(x,y)$.
	\item \textit{Locally expressible examples:}
	Proper chordal-graphs~\cite{paulPREPRINT} via Lemma~\ref{lem:tree-layout}.
	\item \textit{Not locally expressible examples:}
	Not considered yet.
	\item \textit{Relation to other expressive powers:}
	$\ex(sh_L) \subseteq \ex(sh_G)\subseteq \ex(sh_S)$
	(Proposition~\ref{prop:factor-restriction}).
	\item \textit{Complexity Problem:}
	Not considered yet.
	\item \textit{Questions:} 
	Which of the inclusions $\ex(S_L) \subseteq \ex(sh_G) \subseteq \ex(sh_S)$
	are proper inclusions?\\
\end{itemize}
\newpage

\hypertarget{G}{
\noindent
\textbf{\large{{Graphs -- $\mathcal{G}$}}}}\\
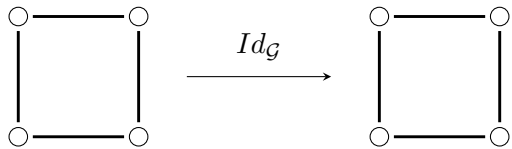
\begin{figure}[ht!]
\begin{center}
\begin{tikzpicture}[scale = 0.8]

\node [vertex] (1) at (-4,0){};
\node [vertex] (2) at (-4,2){};
\node [vertex]  (3) at (-2,2){};
\node [vertex] (4) at (-2,0){};
\foreach \from/\to in {2/3, 4/3, 4/1, 2/1}
\draw[edge]    (\from)  edge  (\to);

\node [vertex] (a) at (2,0){};
\node [vertex] (b) at (2,2){};
\node [vertex]  (c) at (4,2){};
\node [vertex] (d) at (4,0){};

\foreach \from/\to in {a/b, b/c, c/d, d/a}
\draw[edge]    (\from)  to  (\to);

\node[] (D) at (-1.5,1){};
\node[] (C) at (1.5,1){};
\node[] at (0,1.5){$Id_{\mathcal{G}}$};
\draw[arc, thin]    (D)  edge  (C);

\end{tikzpicture}
\caption[Functor $Id_\mathcal{G}\colon\mathcal{G\to G}$]{To the left, a graph $A$
such that $Id_\mathcal{G}(A) = C_4$.}
\end{center}
\end{figure}
\begin{itemize}
	\item \textit{Functor:}
	The identity on $\mathcal{G}$; also defined by the self quantifier-free definition
	$E\mapsto E(x,y)$.
	\item \textit{Equivalent expressions:}
	Local and $\forall_1$-definable graph classes
	(Theorem~\ref{thm:equivlocal}). Same expressive power as
	complement functor $co(G) = \overline{G}$
	(Proposition~\ref{prop:symmetries}).
	\item \textit{Locally expressible examples:}
	Cographs and triangle-free graphs.
	\item \textit{Not locally expressible examples:}
	Bipartite graphs and forests.
	\item \textit{Relation to other expressive powers:}
	$\ex(Id_\mathcal{G})\subseteq \ex(F)$ for any concrete functor
	$F\colon\mathcal{C\to G}$.
	\item \textit{Complexity Problem:}
	Polynomial-time solvable.\\
\end{itemize}
\newpage	

\hypertarget{LOR}{
\noindent
\textbf{\large{{Linearly ordered graphs -- $\LOR$}}}}\\
\begin{figure}[ht!]
\begin{center}
\begin{tikzpicture}[scale = 0.8]

\node [vertex, fill= white] (1) at (-2,1){};
\node [vertex, fill= white] (2) at (-3.5,1){};
\node [vertex, fill= white]  (3) at (-5,1){};
\node [vertex, fill= white] (4) at (-6.5,1){};

\foreach \from/\to in {1/3, 4/2, 1/4}
\draw[edge]    (\from)  edge [bend right]  (\to);
\draw[edge]    (3)  edge (2);

\node [vertex] (a) at (2,0){};
\node [vertex] (b) at (2,2){};
\node [vertex]  (c) at (4,2){};
\node [vertex] (d) at (4,0){};

\foreach \from/\to in {a/b, b/c, c/d, d/a}
\draw[edge]    (\from)  to  (\to);

\node[] (D) at (-1.5,1){};
\node[] (C) at (1.5,1){};
\node[] at (0,1.5){$sh_\tau$};
\draw[arc, thin]    (D)  edge  (C);

\end{tikzpicture}
\caption[Functor $sh_L\colon\mathcal{LOR\to G}$]{To the left, a linearly ordered
graph $A$ such that $sh_L(A) = C_4$. Vertices are ordered
from left to right.}
\end{center}
\end{figure}
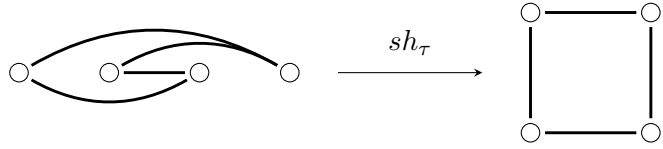
\begin{itemize}
	\item \textit{Functor:}
	Forgets the linear ordering; also defined by
	the quantifier-free definition $E\mapsto E(x,y)$.
	\item \textit{Equivalent expressions:}
	Classes expressible by forbidden linear orderings,
	$FOSG$-classes \cite{damaschkeTCGT1990}, and forbidden patterns
	\cite{feuilloleyJDM, hellESA2014}.
	\item \textit{Locally expressible examples:}
	Forests, proper-interval graphs, and bipartite graphs~\cite{damaschkeTCGT1990, feuilloleyJDM}.
	\item \textit{Not locally expressible examples:}
	Unknown.
	\item \textit{Relation to other expressive powers:}
	$\ex(S_{|\AO})\cup \ex(sh_C) \subseteq \ex(sh_L)$
	(Proposition~\ref{prop:factor-restriction}).
	\item \textit{Complexity Problem:}
	Polynomial time solvable for $F\subseteq \LOR_{\le3}$
	\cite{hellESA2014}.
	For every $k\ge 4$ there is a set $F\subseteq \LOR_k$
	such that its Complexity Problem is NP-complete
	(follows from RGHV-Theorem via Proposition~\ref{prop:factor-restriction}).
	\item \textit{Questions:} Is the class of even-hole free graphs
	expressible by forbidden linear orderings?~\cite{hellPC}.\\
\end{itemize}
\newpage

\hypertarget{LOOR}{
\noindent
\textbf{\large{{Linearly ordered oriented graphs  -- $\LOOR$}}}}\\
\begin{figure}[ht!]
\begin{center}
\begin{tikzpicture}[scale = 0.8]

\node [vertex, fill= white] (1) at (-2,1){};
\node [vertex, fill= white] (2) at (-3.5,1){};
\node [vertex, fill= white]  (3) at (-5,1){};
\node [vertex, fill= white] (4) at (-6.5,1){};

\foreach \from/\to in {1/3, 4/2, 1/4}
\draw[arc]    (\from)  edge [bend right]  (\to);
\draw[arc]    (3)  edge (2);

\node [vertex] (a) at (2,0){};
\node [vertex] (b) at (2,2){};
\node [vertex]  (c) at (4,2){};
\node [vertex] (d) at (4,0){};

\foreach \from/\to in {a/b, b/c, c/d, d/a}
\draw[edge]    (\from)  to  (\to);

\node[] (D) at (-1.5,1){};
\node[] (C) at (1.5,1){};
\node[] at (0,1.5){$u_\tau$};
\draw[arc, thin]    (D)  edge  (C);

\end{tikzpicture}
\caption[Functor $u_L\colon\mathcal{LOOR\to G}$]{To the left, a linearly ordered
oriented graph $A$ such that $u_L(A) = C_4$. Vertices are ordered
from left to right.}
\end{center}
\end{figure}
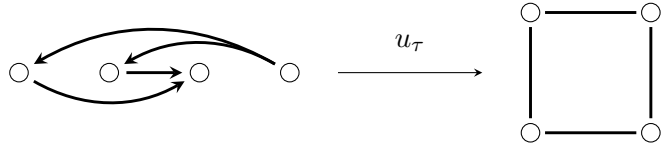
\begin{itemize}
	\item \textit{Functor:}
	Forgets the linear ordering and the orientation of the arcs;
	also defined by quantifier-free definition $E \mapsto E(x,y) \lor E(y,x)$.
	\item \textit{Equivalent expressions:}
	Isomorphic to the pullbacks \break
	$P\colon \mathcal{OR \times_G LOR\to G}$ and 
	$Q\colon \mathcal{EG}_2\mathcal{ \times_G LOR\to G}$.
	\item \textit{Locally expressible examples:}
	$3$-colourable graphs (Example~\ref{ex:3colLOOR}).
	\item \textit{Not locally expressible examples:}
	Not considered yet.
	\item \textit{Relation to other expressive powers:}
	$ex(sh_L\oplus S_{|\OR})\subseteq ex(u_L)$ 
	(Proposition~\ref{cor:pullback-disjoint})
	\item \textit{Complexity Problem:}
	For every $k\ge 3$ there is a set $F\subseteq \LOOR_k$ such
	that its Complexity Problem is NP-complete
	(Example~\ref{ex:3colLOOR}).
	\item \textit{Questions:} Which classes are expressible by linearly ordered
	oriented graphs on $3$ vertices?\\
\end{itemize}
\newpage

\hypertarget{OR}{
\noindent
\textbf{\large{{Oriented graphs  -- $\OR$}}}}\\
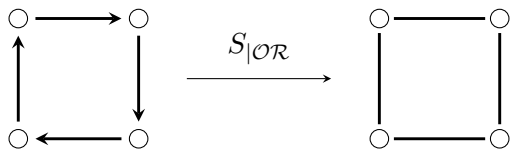
\begin{figure}[ht!]
\begin{center}
\begin{tikzpicture}[scale = 0.8]

\node [vertex] (1) at (-4,0){};
\node [vertex] (2) at (-4,2){};
\node [vertex]  (3) at (-2,2){};
\node [vertex] (4) at (-2,0){};
\foreach \from/\to in {2/3, 3/4, 4/1, 1/2}
\draw[arc]    (\from)  edge  (\to);

\node [vertex] (a) at (2,0){};
\node [vertex] (b) at (2,2){};
\node [vertex]  (c) at (4,2){};
\node [vertex] (d) at (4,0){};

\foreach \from/\to in {a/b, b/c, c/d, d/a}
\draw[edge]    (\from)  to  (\to);

\node[] (D) at (-1.5,1){};
\node[] (C) at (1.5,1){};
\node[] at (0,1.5){$S_{|\OR}$};
\draw[arc, thin]    (D)  edge  (C);
\end{tikzpicture}
\caption[Functor $S_{|\OR}\colon\mathcal{OR\to G}$]{To the left, an
oriented graph $A$ such that $S_{|\OR}(A) = C_4$.}
\end{center}
\end{figure}
\begin{itemize}
	\item \textit{Functor:}
	Maps an oriented graph to its
	underlying graph; also defined by the self quantifier-free definition $E\mapsto E(x,y) \lor E(y,x)$.
	\item \textit{Equivalent expressions:}
	Classes expressible by forbidden orientations~\cite{guzmanAR,guzmanEJC105,guzmanEJC28},
	and $F$-graph classes~\cite{skrienJGT6}.
	\item \textit{Locally expressible examples:}
	Proper circular-arc graphs~\cite{skrienJGT6}, unicyclic \break graphs~\cite{guzmanAR}, and homomorphism classes
	of odd cycles~\cite{guzmanEJC28}.
	\item \textit{Not locally expressible examples:} Forests,
	chordal graphs and even-hole free graphs~\cite{guzmanEJC105}.
	\item \textit{Relation to other expressive powers:}
	$ex(S_{|\OR})~\subseteq~ex(S) \cap ex(SB)$
	(Proposition~\ref{prop:factor-restriction}).
	\item \textit{Complexity Problem:}
	Known polynomial for $F\subseteq \OR_{\le3}\setminus\big\{\{B_1,\overrightarrow{C}_3\},$ $\{B_2,
	\overrightarrow{C}_3\}\big\}$~\cite{guzmanAR}.
	For every $k\ge 4$ there is a set $F\subseteq \OR_k$
	such that its Complexity Problem is NP-complete
	(follows from RGHV-Theorem).
	\item \textit{Questions:} 
	Is the class of odd-hole free graphs expressible by forbidden orientations?
	The Complexity and Characterization Problem for $\{B_1,\overrightarrow{C}_3\}$.\\
\end{itemize}
\newpage

\hypertarget{PO}{
\noindent
\textbf{\large{{Partially ordered graphs  -- $\PO$}}}}\\
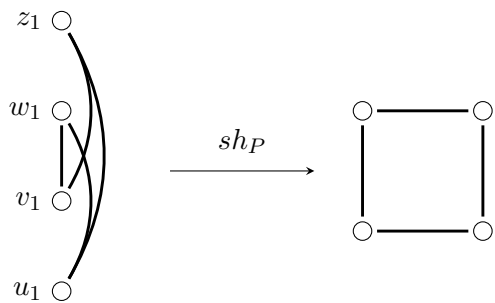
\begin{figure}[ht!]
\begin{center}
\begin{tikzpicture}[scale = 0.8]

\node [vertex, fill= white, label = left:{$u_1$}] (1) at (-3,-1){};
\node [vertex, fill= white, label = left:{$v_1$}] (2) at (-3,0.5){};
\node [vertex, fill= white, label = left:{$w_1$}]  (3) at (-3,2){};
\node [vertex, fill= white, label = left:{$z_1$}] (4) at (-3,3.5){};

\foreach \from/\to in {1/3, 2/4, 1/4}
\draw[edge]    (\from)  to [bend right]  (\to);
\draw[edge]    (3)  to   (2);

\node [vertex] (a) at (2,0){};
\node [vertex] (b) at (2,2){};
\node [vertex]  (c) at (4,2){};
\node [vertex] (d) at (4,0){};

\foreach \from/\to in {a/b, b/c, c/d, d/a}
\draw[edge]    (\from)  to  (\to);

\node[] (D) at (-1.5,1){};
\node[] (C) at (1.5,1){};
\node[] at (0,1.5){$sh_P$};
\draw[arc, thin]    (D)  edge  (C);
\end{tikzpicture}
\caption[Functor $sh_P\colon\mathcal{PO\to G}$]{To the left, a partially
ordered graph $A$ such that $sh_P(A) = C_4$. Vertices are ordered
from left to right.}
\end{center}
\end{figure}
\begin{itemize}
	\item \textit{Functor:}
	Forgets the partial ordering; also defined by quantifier-free definition
	$E \mapsto E(x,y)$.
	\item \textit{Equivalent expressions:}
	Graphs with a $T_0$-Alexandroff Topology on its vertex
	set (see, e.g.,~\cite{goubault2013}, or Example~\ref{ex:POkcol}) .	
	\item \textit{Locally expressible examples:}
	Comparability graphs and $k$-colourable graphs
	(Examples~\ref{ex:POcomp}, \ref{ex:POkcol}).
	\item \textit{Not locally expressible examples:}
	Not considered yet.
	\item \textit{Relation to other expressive powers:}
	$ex(sh_L) \subseteq \ex(sh_S)\subseteq ex(sh_P)$
	(Proposition~\ref{prop:factor-restriction}).
	\item \textit{Complexity Problem:}
	Not considered yet.
	\item \textit{Questions:} 
	Is $ex(S_L)$ properly contained in $ex(sh_P)$?\\
\end{itemize}
\newpage

\hypertarget{SO}{
\noindent
\textbf{\large{{Suitably ordered graphs -- $\SO$}}}}\\
\begin{figure}[ht!]
\begin{center}
\begin{tikzpicture}[scale = 0.8]

\node [vertex, label = right:{$z_3$}] (1) at (-2.5,1){};
\node [vertex, label = above:{$v_2$}] (2) at (-3.91,2.41){};
\node [vertex, label = left:{$u_1$}]  (3) at (-5.32,1){};
\node [vertex, label = below:{$w_2$}] (4) at (-3.91,-0.41){};

\foreach \from/\to in {2/3, 3/4, 2/1, 1/4}
\draw[edge]    (\from)  to  (\to);

\node [vertex] (a) at (2,0){};
\node [vertex] (b) at (2,2){};
\node [vertex]  (c) at (4,2){};
\node [vertex] (d) at (4,0){};

\foreach \from/\to in {a/b, b/c, c/d, d/a}
\draw[edge]    (\from)  to  (\to);

\node[] (D) at (-1.5,1){};
\node[] (C) at (1.5,1){};
\node[] at (0,1.5){$sh_S$};
\draw[arc, thin]    (D)  edge  (C);

\end{tikzpicture}
\caption[Functor $sh_S\colon\mathcal{SO\to G}$]{To the left, a suitably
ordered graph $A$ such that $sh_S(A) = C_4$. Vertices are ordered
from left to right.}
\end{center}
\end{figure}
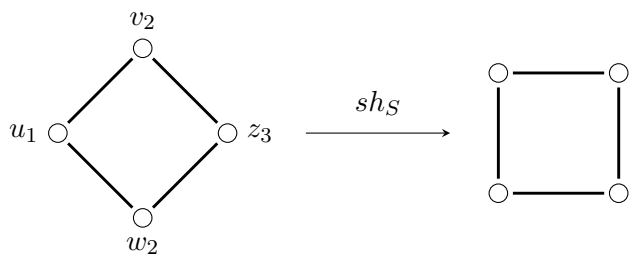
\begin{itemize}
	\item \textit{Functor:}
	Forgets the ordering on the vertex set; also defined by quantifier-free definition
	$E \mapsto E(x,y)$.
	\item \textit{Locally expressible examples:}
	Comparability graph of posets of height at most $k$
	(Example~\ref{ex:comparabilityheightk}).
	\item \textit{Not locally expressible examples:}
	Not considered yet.
	\item \textit{Relation to other expressive powers:}
	$ex(sh_L) \subseteq ex(sh_S)\subseteq ex(sh_P)$
	(Proposition~\ref{prop:factor-restriction}).
	\item \textit{Complexity Problem:}
	Not considered yet.
	\item \textit{Questions:} 
	Which of the inclusions $ex(S_L) \subseteq ex(sh_S) \subseteq ex(sh_P)$
	are proper inclusions?\\
\end{itemize}
\newpage

\hypertarget{Gk}{
\noindent
\textbf{\large{Vertex-coloured graphs -- $\mathcal{G}_k$}}}\\
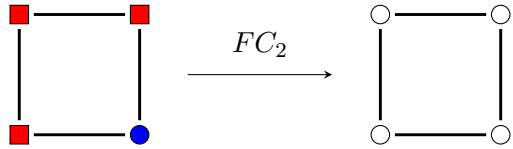
\begin{figure}[ht!]
\begin{center}
\begin{tikzpicture}[scale = 0.8]

\node [vertex, fill = red, rectangle] (1) at (-4,0){};
\node [vertex, fill = red, rectangle] (2) at (-4,2){};
\node [vertex, fill = red, rectangle]  (3) at (-2,2){};
\node [vertex, fill = blue] (4) at (-2,0){};
\foreach \from/\to in {2/3, 4/3, 4/1, 2/1}
\draw[edge]    (\from)  edge  (\to);

\node [vertex] (a) at (2,0){};
\node [vertex] (b) at (2,2){};
\node [vertex]  (c) at (4,2){};
\node [vertex] (d) at (4,0){};

\foreach \from/\to in {a/b, b/c, c/d, d/a}
\draw[edge]    (\from)  to  (\to);

\node[] (D) at (-1.5,1){};
\node[] (C) at (1.5,1){};
\node[] at (0,1.5){$FC_2$};
\draw[arc, thin]    (D)  edge  (C);

\end{tikzpicture}
\caption[Functor $FC_2\colon\mathcal{G_2\to G}$]{To the left, a 
$2$-vertex-coloured graph $A$ such that $FC_2(A) = C_4$.}
\end{center}
\end{figure}
\begin{itemize}
	\item \textit{Functor:}
	Forgets the vertex colouring; also defined by quantifier-free definition
	$E\mapsto E(x,y)$.
	\item \textit{Equivalent expressions:}
	$FC_1$ is isomorphic to $Id_\mathcal{G}$.
	\item \textit{Locally expressible examples:}
	Generalized $M$-partition classes for $k\times k$ matrices over
	$\{0,1,\ast,\oplus\}$ (Example~\ref{ex:M-partition}).
	\item \textit{Not locally expressible examples:}
	Not considered yet.
	\item \textit{Relation to other expressive powers:}
	$\ex(FC_k)\subseteq \ex(FC_{k+1})$ for every integer $k\ge 0$.
	\item \textit{Complexity Problem:}
	Polynomial-time solvable for $FC_1$. Not considered yet
	for $FC_2$. When $k\ge 3$ there is a set $X$ of $k$-coloured
	graphs on $2$ vertices such the Complexity Problem for $FC_k$
	and $X$ is NP-complete.
	\item \textit{Questions:} Is there a positive integer $k$ such 
	that $\ex(FC_k) = \ex(FC_{k+1})$?\\
\end{itemize}
\newpage


\section{Glossary of graph classes}
\label{ap:graph-classes}

This appendix intends to be a quick reference for all graph classes mentioned
in this work. We also include some classes of relational structures considered
in Section~\ref{sec:examples} as domains of certain concrete functors.\\

\noindent
\textbf{1-split graphs.}
A graph $G$ is a $1$-split graph if $G$ is
either a complete graph or a complete graph minus an edge, or the complement
of such a graph, i.e., an independent set, plus possibly an edge. 
\textit{Minimal obstructions:} All graphs on four vertices with
at least $2$ edges and $2$ non-edges.\\

\noindent
\textbf{Arborescence comparability graphs.}
A poset $X$ is an arborescence
if for every $x\in X$ the set $\{y\in X\colon
y\le x\}$ is linearly ordered in $X$. A graph $G$ is an arborescence 
comparability graph if it is isomorphic to the comparability graph of an
arborescence poset. \textit{Minimal obstructions:}
$C_4$ and $P_4$~\cite{golumbicDM24}.
This is the same class as nested interval graphs, and as trivially-perfect graphs.\\

\noindent
\textbf{Augmented clique.}
 A graph is an augmented clique if it is a complete
graph plus one additional node with arbitrary adjacency, and possibly some
isolated nodes. \textit{Minimal obstructions:} Finite since it is $\forall_1$-definable. \\

\noindent
\textbf{Bipartite graphs.}
A graph $G$ is a bipartite graph if its vertex set
can be partitioned into two independent sets. \textit{Minimal obstructions:}
Odd cycles \cite{bondy2008}.\\

\noindent
\textbf{Caterpillars.}
A graph $G$ is a caterpillar if it is a tree and it contains a
path $P$ such that every vertex has a neighbor in $P$.\\

\noindent
\textbf{Caterpillar forest.}
A caterpillar is a forest where every connected component
is a caterpillar. \textit{Minimal obstructions:} All cycles and $T_2$ ($T_2$ is obtained
from $K_{1,3}$ by subdividing each edge one).\\

\noindent
\textbf{Chordal graphs.}
A graph $G$ is chordal if it contains no induced
cycle on more that three vertices. Equivalently, a graph $G$ is chordal
if every cycle of $G$ of length at least $4$ has a chord. \textit{Minimal obstructions:}
Cycles on at least $4$ vertices.\\

\noindent
\textbf{Circular-arc graphs.}
A graph $G$ is a circular-arc graph if there is a family
$\mathcal{S}$ of circular arcs of a circumference
such that $G \cong I(\mathcal{S})$. \textit{Minimal obstructions:} Unknown.\\

\noindent
\textbf{Circularly ordered graphs.}
A circularly ordered graph is an $\{E,C\}$-structure $(V(A), E(A), C(A))$,
such that $(V(A),E(A))$ is a graph
and \break$(V(A), C(A))$ is a circular ordering.  To simplify notation we denote
a partially ordered graph by $(G,c(\le))$ where $G$ is a graph and $c(\le)$
is circular ordering on $V(G)$.\\

\noindent
\textbf{Cographs.}
The family of cographs is the minimal class containing
$K_1$ that it is closed under disjoint unions and  complements.  Equivalently, 
a graph is a cograph if it contains no induced path on four vertices.
\textit{Minimal obstructions:} $P_4$ \cite{corneilDAM3}.\\

\noindent
\textbf{Comparability graphs.}
The comparability graph of a poset $(X,\le)$ has
as vertex set $X$ and for every pair of comparable elements, $x \le y$ or $y\le x$,
we add an edge $xy$. A graph $G$ is a comparability graph if it is isomorphic to the
comparability graph of a poset. \textit{Minimal obstructions:} Known \cite{gallaiAMA18}.\\

\noindent
\textbf{Complete bipartite graphs.}
Given a pair of non-negative integers
$n$ and $m$, the graph $K_{n,m}$ is the graph with vertex set $\{x_1,\dots, x_n,
y_1,\dots, y_m\}$ and edge set $\{x_iy_j~| 1\le i\le n, 1\le j \le m\}$. A graph
$G$ is a complete bipartite if it is isomorphic to $K_{n,m}$ for some
non-negative integers $n$ and $m$.\\

\noindent
\textbf{Complete multipartite graphs.}
A graph $G$ is a complete multipartite
if its vertex set admits a partition $(V_1,\dots, V_k)$ into independent
set such that every pair of vertices in different classes are adjacent.
\textit{Minimal obstructions:} $K_1+K_2$.\\

\noindent
\textbf{Complete graphs.}
 A complete graph is a graph $G$
such that every pair of (different) vertices in $V(G)$ are adjacent. The unique
(up to isomorphism) complete graph on $n$ vertices is denoted by $K_n$.
\textit{Minimal obstructions:} $2K_1$.\\

\noindent
\textbf{Cycles.}
Given a non-negative integer $n$, a cycle on $n$ vertices is graph 
isomorphic to the graph with vertex set $\{v_0,\dots, v_{n-1}\}$ and edge
set $\{v_iv_{i+1}~|~i\in\{0,\dots, n-2\}\}\cup\{v_{n-1},v_0\}$. The unique
(up to isomorphism) cycle on $n$ vertices is denoted by $C_n$.\\

\noindent
\textbf{Digraphs.}
We consider a digraph $D$ as an
$\{E\}$-structure $(V(D),$ $ E(D))$ where $E(D)$ is an
irreflexive binary relation on $V(D)$.\\

\noindent
\textbf{Directed cycles.}
Given a non-negative integer $n$,  a directed cycle on
$n$ vertices is graph isomorphic to the graph with vertex set
$\{v_0,\dots, v_{n-1}\}$ and arc set
$\{(v_i,v_{i+1})~|~i\in\{0,\dots, n-2\}\}\cup\{(v_{n-1},v_0)\}$. The unique
(up to isomorphism) directed cycle on $n$ vertices is denoted by
$\overrightarrow{C}_n$.\\

\noindent
\textbf{Directed paths.}
Given a non-negative integer $n$,  a directed path on
$n$ vertices is graph isomorphic to the graph with vertex set
$\{v_0,\dots, v_{n-1}\}$ and arc set
$\{(v_i,v_{i+1})~|~i\in\{0,\dots, n-2\}\}$. The unique
(up to isomorphism) directed path on $n$ vertices is denoted by
$\overrightarrow{P}_n$.\\

\noindent
\textbf{Even-hole free graphs.}
A graph $G$ is an even-hole free graph if it
contains no induced even cycle. \textit{Minimal obstructions:} Even cycles.\\

\noindent
\textbf{$\mathbf{F}$-free graphs.}
Given a set of graphs $F$, a graph $G$ is
$F$-free if it contains no induced subgraph in $F$.\\

\noindent
\textbf{$\mathbf{Forb(F)}$.}
Given a set of digraphs $F$, a digraph $D$
belongs to $Forb(F)$ if for each $H\in F$, there is no homomorphism
from $H$ to $D$; in symbols $H\not\to D$ for every $H\in F$.\\

\noindent
\textbf{$\mathbf{Forb_e(F)}$.} Given a set of digraphs $F$, a digraph $D$
belongs to $Forb_e(F)$ if $D$ contains no induced digraph if $F$.
Equivalently, if for each $H\in F$, there is no embedding from $H$
to $D$.\\

\noindent
\textbf{Forests.}
A forest is an acyclic graph. Equivalently, 
a forest is a disjoint union of trees. \textit{Minimal obstructions:} All cycles.\\

\noindent
\textbf{Genealogical graphs.}
A genealogical graph is
a suitably ordered graph $(G,\le)$ such that the down-set of every vertex
is a chain; that is, for every $x\in V(G)$ if $y,z\le x$ then $y$ and $z$ are
comparable. Genealogical graphs can be represented by a parent
function $p\colon V(G)\to V(G)$ where $p(x) =
\max\{y\in V(G)\colon y\le x, y\neq x\}$ if $x$ is not a $\le$-minimal
element, and $p(x) = x$ for every $\le$-minimal element.\\

\noindent
\textbf{Generalized $\mathbf{M}$-partitionable graphs.}
Consider  an
$m\times m$ matrix $M$ over $\{0,1,\ast,\oplus\}$ and a graph $G$. 
An $M$-partition of $G$ is a partition (with possibly empty classes) $(V_1,\dots,V_n)$ 
of $V(G)$, such that for every $x\in V_i$ and every $y\in V_j$ where
$m_{i,j}\neq \oplus$ the following holds: if $m_{ij} = 0$ then $xy\not \in E(G)$, and
if $m_{ij} = 1$  then $xy \in E(G)$; otherwise, for every $i,j\in\{1,\dots, m\}$
such that $m_{ij} = \oplus$ it holds that: if $i = j$ then $|V_i| \le 1$, and if
$i\neq j$ then one of $V_i$ and $V_j$ is empty.
A graph $G$ is $M$-partitionable if it admits an $M$-partition.\\

\noindent
\textbf{Graphs.}
We consider a graph $G$ as an $\{E\}$-structure
$(V(G), E(G))$ such that $E(G)$ is a symmetric irreflexive binary relation on $V(G)$.\\

\noindent
\textbf{$\mathbf{H}$-colourable digraphs.}
Given a digraph $H$, a digraph
$D$ is $H$-colourable
if there is a homomorphism $\varphi\colon D\to H$. The class of $D$-colourable digraphs
is denoted by $CSP(H)$. The same notation is used for $H$-colourable graphs;
it is always clear from context whether we are referring to graphs or to digraphs.\\

\noindent
\textbf{$\mathbf{H}$-colourable graphs.} Given a graph $H$, a graph $G$ is $H$-colourable
if there is a homomorphism $\varphi\colon G\to H$. The class of $H$-colourable graphs
is denoted by $CSP(H)$. \textit{Minimal obstructions:}
$K_2$ if $H$ is an independent graph; odd cycles when $H$ is a non-empty
bipartite graph; unknown otherwise \cite{hell2004}.\\

\noindent
\textbf{$\mathbf{\mathcal{M}}$-colourable digraphs.} Given a set of
digraphs $\mathcal{M}$, a
digraph $D$ is $\mathcal{M}$-colourable if there is some digraph $H\in \mathcal{M}$
such that $G\to H$. The class of $\mathcal{M}$-colourable digraphs is denoted by
$CSP(\mathcal{M})$. The same notation is used for $\mathcal{M}$-colourable graphs;
it is always clear from context whether we are referring to graphs or to digraphs.\\

\noindent
\textbf{$\mathbf{\mathcal{M}}$-colourable graphs.} Given a set of graphs $\mathcal{M}$, a
graph $G$ is $\mathcal{M}$-colourable if there is some graph $H\in M$ such that
$G\to H$. The class of $\mathcal{M}$-colourable graphs is denoted by
$CSP(\mathcal{M})$.\\


\noindent
\textbf{Independent graphs.}
An independent (or empty) graph is a graph
with empty set of edges. The unique (up to isomorphism) independent graph on 
$n$ vertices is denoted by $nK_1$. \textit{Minimal obstructions:} $K_2$.\\

\noindent
\textbf{Intersection graphs.}
Given a finite family
of sets $\mathcal{S}$ the intersection graph of $\mathcal{S}$,
$I(\mathcal{S})$, has a vertex set $\mathcal{S}$ and two vertices $S,T\in
\mathcal{S}$ are adjacent if $S\neq T$ and $S\cap T\neq \varnothing$.\\

\noindent
\textbf{Interval graphs.}
A graph $G$ is an interval graph if there is a family
of intervals of the real line $\mathcal{S}$ such that $G \cong I(\mathcal{S})$.\\

\noindent
\textbf{$\mathbf{k}$-colourable graphs.}
Given a positive integer $k$, a graph
$G$ is $k$-colourable its vertex set can be partitioned into a most $k$ independent
sets. Equivalently, $G$ is $k$-colourable if there is a homomorphism 
$\varphi\colon G\to K_k$. \textit{Minimal obstructions:}
$K_2$ if $k =1$; odd cycles if $k = 2$;  unknown when $k\ge 3$
\cite{hell2004}.\\

\noindent
\textbf{$\mathbf{k}$-coloured graphs.}
We consider a $k$-coloured graph $G$ as an $\{E,U_1,\dots, U_k\}$-structure
such that  $(V(G), E(G))$ is a graph and $(U_1(G),\dots,$ $ U_k(G))$ is a partition of 
$V(G)$. Not to be confused with properly coloured graphs nor with 
$k$-colourable graphs.\\

\noindent
$\mathbf{(k,l)}$\textbf{-graphs.}
A graph $G$ is a $(k,l)$-graph if its vertex set
admits a partition $(V_1,\dots, V_{k'}, U_1,\dots, U_{l'})$ such that 
$V_i$ is a complete graph and $U_j$ an independent set for 
each $i\in\{1,\dots, k'\}$ and $j\in\{1,\dots, l'\}$, and $k'\le k$ and $l'\le l$.\\

\noindent
\textbf{Linear forest.}
A linear forest is a forest such that each connected component
is a path. \textit{Minimal obstructions:} All cycles and $K_{1,3}$.\\

\noindent
\textbf{Linearly ordered graph.}
A linearly ordered graph is an
$\{E,\le\}$-structure $A= (V(A), E(A), {\le}(A))$ such that $(V(A),E(A))$ is a graph
and $(V(A), {\le}(A))$ is a linearly ordered set. To simplify notation we denote
a linearly ordered graph by $(G,\le)$ where $G$ is a graph and $\le$ is 
a linear ordering on $V(G)$.\\

\noindent
\textbf{Locally $\mathcal{P}$ graphs.} Given a graph property $\mathcal{P}$, 
a graph $G$ is locally $\mathcal{P}$, if the neighborhood of each
vertex of $G$ induces a graph in $\mathcal{P}$.\\

\noindent
\textbf{$\mathbf{M}$-partitionable graphs.}
Consider  an $(m\times m)$ matrix $M$ over
$\{0,1,\ast\}$ and a graph $G$. An  $M$-partition of $G$ is a partition
$S_1,S_2,\dots, S_m$ (with some possible empty classes) of $V(G)$ such
that for any two vertices $v\in  S_i$ and $u\in S_j$ the following hold:
$uv\in E_G$ if $M(i,j) = 1$, and $uv\not\in E_G$ if $M(i,j)=0$.
A graph $G$ is $M$-partitionable if it admits an $M$-partition.\\

\noindent
\textbf{Nested interval graphs.}
A graph $G$ is an interval graph if there is a family
of intervals of the real line $\mathcal{S}$ such that $G \cong I(\mathcal{S})$, and
$\mathcal{S}$ satisfies that for every pair $S,T\in\mathcal{S}$ with non-empty
intersection then $S\subseteq T$ or $T\subseteq S$.
\textit{Minimal obstructions:} $C_4$ and $P_4$ \cite{golumbicDM24}.
This is the same class as arborescence comparability graphs, and as
trivially-perfect graphs.\\

\noindent
\textbf{Oriented graphs.}
We consider an oriented graph $G$ as an $\{E\}$-structure $(V(G), E(G))$
such that $E(G)$ is an antisymmetric irreflexive binary relation on $V(G)$.\\

\noindent
\textbf{Out-tournament.}
An oriented graph $D$ is an out-tournament if the
out-neighborhood of every vertex is a tournament.\\

\noindent
\textbf{Outerplanar graphs.}
A graph is outerplanar if it admits
a planar embedding where all vertices belong to the same face.\\

\noindent
\textbf{$\mathbf{\mathcal{P}}$-mixed graph.}
Consider a hereditary property
$\mathcal{P}$. A $\mathcal{P}$-mixed
graph is a graph $G$ whose edge set can be
partitioned into $E_1$ and $E_2$, where $(V(G),E_1)\in \mathcal{P}$, and
$(V(G),E_2)$ admits a transitive orientation $(V(G),E'_2)$
such that if $(x,y)\in E'_2(G)$ and $yz\in E_1$, then $xz\in E_1$
(see~\cite{gavrilIPL73}).\\

\noindent
\textbf{Partially ordered graphs.}
A partially ordered graph is an
$\{E,\le\}$-structure $A= (V(A), E(A), {\le}(A))$ such that $(V(A),E(A))$ is a graph
and $(V(A), {\le}(A))$ is partially ordered set. To simplify notation we denote
a partially ordered graph by $(G,\le)$ where $G$ is a graph and $\le$ is 
a partial ordering on $V(G)$.\\

\noindent
\textbf{Paths.}
Given a non-negative integer $n$,  a path on $n$ vertices is
graph isomorphic to the graph with vertex set $\{v_0,\dots, v_{n-1}\}$ and edge set
$\{v_iv_{i+1}~|~i\in\{0,\dots, n-2\}\}$. The unique
(up to isomorphism) path on $n$ vertices is denoted by $P_n$.\\

\noindent
\textbf{Perfect graphs.}
A graph $G$ is a perfect graph if for every induced
subgraph $G'$ of $G$, the chromatic number $\chi(G')$ is
the same as the maximum complete graph in $G'$. 
\textit{Minimal obstructions:} Odd cycles and complements of odd cycles
\cite{chusdnovskyAM164}.\\

\noindent
\textbf{Perfectly orientable graphs.}
A graph $G$ is perfectly orientable if
it admits an orientation as an out-tournament. Equivalently, a graph $G$ is
perfectly orientable if it admits an orientation as an in-tournament. 
\textit{Minimal obstructions:} Unknown (see partial list in \cite{hartingerJGT85,hartingerDM340}).\\

\noindent
\textbf{Planar graphs.}
A graph is planar if it can be embedded in the plane
(drawn with points for vertices and curves for edges) where the end points
of a curve representing $e$ correspond to the points representing the
end-vertices of $e$, and every pair of curves intersect only at end points.\\

\noindent
\textbf{$\mathbf{(\mathcal{P},\mathcal{Q})}$-colourable graphs.}
Given a pair $\mathcal{P}$ and  $\mathcal{Q}$ of hereditary properties,
a graph $G$ is $(\mathcal{P},\mathcal{Q})$-colourable if there is partition $(A,B)$
of $V(G)$ such that $G[A]\in\mathcal{P}$ and $G[B]\in\mathcal{Q}$.\\

\noindent
\textbf{Proper circular-arc graphs.}
A graph $G$ is a proper circular-arc graph if there
is a family $\mathcal{S}$ of  incomparable circular arcs of a circumference
such that $G \cong I(\mathcal{S})$. \textit{Minimal obstructions:}
Known \cite{tuckerDM7}.\\

\noindent
\textbf{Proper Helly circular-arc graphs}
A family of sets $\mathcal{A}$ is said to have the Helly property if, for any subfamily
$\mathcal{B}\subseteq\mathcal{A}$ such that any two sets $A,B\in\mathcal{B}$,
$A\cap B\neq\varnothing$, the intersection of all sets in $\mathcal{B}$ is
non empty. A graph $G$ is a proper Helly circular-arc graph if there is a family
of incomparable intervals of the real line $\mathcal{S}$ such that
$G \cong I(\mathcal{S})$, and such that $\mathcal{S}$ has the Helly property.
\textit{Minimal obstructions:} Known \cite{linDAM161}.\\

\noindent
\textbf{Proper interval graphs.}
A graph $G$ is a proper interval graph if there is a family
of incomparable intervals of the real line $\mathcal{S}$ such that
$G \cong I(\mathcal{S})$.\\

\noindent
\textbf{Rational complete graph.}
Given a pair of positive integers
$p$ and $q$, $q \le p$, the rational complete graph $K_{\sfrac{p}{q}}$
has vertices $\{0,1,\dots, p-1\}$ and there
is an edge $ij$ if and only if the circular distance between $i$ and $j$ is at
least $q$.\\

\noindent
\textbf{Split graphs.}
A graph $G$ is a split graph if $G$ and $\overline{G}$
are chordal graphs. Equivalently, a graph $G$ is a split graph if its vertex
set can be partitioned into an independent set and a complete graph.
 \textit{Minimal obstructions:} $2K_2$, $C_4$ and $C_5$ \cite{foldes1997}.\\
 
\noindent
\textbf{Stars.}
A graph $G$ is a star if it is connected and contains one
distinguished vertex $v$ such that every edge of $G$ is incident in $v$.
Equivalently, a graph $G$ is a star is it isomorphic to a complete
bipartite graph $K_{1,n}$ for some non-negative integer $n$.\\

\noindent
\textbf{Star forest.}
A graph $G$ is a star forests if it is a forest such that every component is
a star. \textit{Minimal obstructions:} $C_3$, $C_4$ and $P_4$.\\

\noindent
\textbf{Suitably ordered graphs.}\index{ordered graph!suitably ordered}
A partially ordered graph $(G,\le)$ is suitably ordered if,
for every pair $x$ and $y$ of adjacent vertices,  $x\le y$ or $y\le x$.\\

\noindent
\textbf{Threshold graphs.}
A graph $G$ is threshold graph if there exists
a real number $s \in \mathbb{R}$ and a set of real weights $\{a_x\}_{x\in V(G)}$,
such that $xy \in E(G)$ if and only if $a_x+a_y \ge s$. \textit{Minimal obstructions:}
$2K_2$, $C_4$ and $P_4$ \cite{chvatal1977}.\\

\noindent
\textbf{Transitive-perfectly orientable graphs.}
A graph $G$ is transitive-perfectly 
orientable if it admits an orientation as an out-tournament, where every
tournament is transitively oriented.
\textit{Minimal obstructions:} Unknown.\\

\noindent
\textbf{Transitive tournament.}
A transitive tournament is an tournament with no
directed cycle. Equivalently, a transitive tournament is a tournament such that
its arc set induces an irreflexive linear order on its vertex set.
The unique (up to isomorphism) transitive tournament on $n$ vertices is denoted
by $TT_n$.\\

\noindent
\textbf{Trivially-perfect graph.}
A trivially-perfect graph is a $\{C_4,P_4\}$-free graph. This is the same class
as arborescence comparability graphs, and as nested interval
graphs.\\

\noindent
\textbf{Tournaments.}
A tournament is an orientation of a complete graph.\\

\noindent
\textbf{Trees.}
A tree is a connected acyclic graph.\\

\noindent
\textbf{Unicyclic graphs.}
A graph $G$ is a unicyclic graph if it contains at most one cycle.\\

\noindent
\textbf{Wheels.}
A wheel is a graph obtained from a cycle by adding a
universal vertex. The (unique up to isomorphism) wheel obtained from
$C_n$ is  denoted by $W_n$.

\vspace{2mm}


\bibliographystyle{abbrv}
\bibliography{global.bib}

\end{document}